\documentclass[10pt]{amsart}
\usepackage{amsmath}
\usepackage{amssymb}
\usepackage{amsfonts}
\usepackage{amsthm}
\usepackage{amsxtra}
\usepackage{portland}
\usepackage{rotating}
\usepackage{nicefrac}
\usepackage{float}
\usepackage{epsfig}
\usepackage[dvips]{color}
\usepackage{array}
\usepackage{pifont}
\usepackage{multirow}
\usepackage{curves}

\theoremstyle{plain}

\newtheorem{corollary}{Corollary}[section]
\newtheorem{lemma}{Lemma}[section]
\newtheorem{proposition}{Proposition}[section]

\theoremstyle{definition}

\theoremstyle{remark}
\newtheorem{remark}{Remark}[section]

\newcommand{\CC}{\mathbb C}
\newcommand{\RR}{\mathbb R}
\newcommand{\ZZ}{\mathbb Z}
\newcommand{\NN}{\mathbb N}

\newcommand{\ScSt}{\scriptstyle}
\newcommand{\PV}{${\rm P}_{\rm V}\:$}

\newcommand{\PIII}{${\rm P}_{\rm III}\:$}
\newcommand{\PIIIprime}{${\rm P}_{\rm III^{\prime}}\:$}
\newcommand{\IIId}{${\rm III^{\prime}}\;$}
\newcommand{\half}{
        {\lower0.00ex\hbox{\raise.6ex\hbox{\the\scriptfont0 1}
                           \kern-.5em\slash\kern-.1em\lower.45ex
                                     \hbox{\the\scriptfont0 2}}}}
\newcommand{\quarter}{
        {\lower0.00ex\hbox{\raise.6ex\hbox{\the\scriptfont0 1}
                           \kern-.5em\slash\kern-.1em\lower.45ex
                                     \hbox{\the\scriptfont0 4}}}}
\newcommand{\tquarter}{
        {\lower0.00ex\hbox{\raise.6ex\hbox{\the\scriptfont0 3}
                           \kern-.5em\slash\kern-.1em\lower.45ex
                                     \hbox{\the\scriptfont0 4}}}}
\newcommand{\eighth}{
        {\lower0.00ex\hbox{\raise.6ex\hbox{\the\scriptfont0 1}
                           \kern-.5em\slash\kern-.1em\lower.45ex
                                     \hbox{\the\scriptfont0 8}}}}
\newcommand{\othird}{
        {\lower0.00ex\hbox{\raise.6ex\hbox{\the\scriptfont0 1}
                           \kern-.5em\slash\kern-.1em\lower.45ex
                                     \hbox{\the\scriptfont0 3}}}}

\begin{document}

\title[Distribution of the first Eigenvalue Spacing ...]
{The Distribution of the first Eigenvalue Spacing at the Hard Edge of the Laguerre Unitary Ensemble}

\author{Peter J. Forrester and Nicholas S. Witte}
\address{Department of Mathematics and Statistics,
University of Melbourne,Victoria 3010, Australia}
\email{\tt p.forrester@ms.unimelb.edu.au}\email{\tt n.witte@ms.unimelb.edu.au}

\begin{abstract}
The distribution function for the first eigenvalue spacing in the Laguerre
unitary ensemble of finite rank random matrices is found in terms of a Painlev\'e V
system, and the solution of its associated linear isomonodromic system. In particular
it is characterised by the polynomial solutions to the isomonodromic equations which are
also orthogonal with respect to a deformation of the Laguerre weight. In the scaling to
the hard edge regime we find an analogous situation where a certain Painlev\'e \IIId  
system and its associated linear isomonodromic system characterise the scaled distribution.
We undertake extensive analytical studies of this system and use this knowledge to 
accurately compute the distribution and its moments for various values of the parameter 
$ a $. In particular choosing $ a=\pm 1/2 $ allows the first eigenvalue spacing 
distribution for random real orthogonal matrices to be computed.
\end{abstract}

\subjclass[2000]{15A52, 33C45, 33E17, 42C05, 60K35, 62E15}
\keywords{random matrices, eigenvalue distribution, Wishart matrices, Painlev\'e equations, isomonodromic deformations}
\maketitle

\section{Introduction}
\setcounter{equation}{0}
The Laguerre unitary ensemble (${\rm LUE}_{n,a}$) of random matrices is specified
by the eigenvalue probability density function (p.d.f.)
\begin{multline}
 p(\lambda_1,\ldots,\lambda_n) \\
 := \frac{1}{n!c_{n,n+a}}
 \prod^{n}_{j=1}e^{-\lambda_j}\lambda_j^a\prod_{1 \leq j<k \leq n}(\lambda_k-\lambda_j)^2 ,
 \quad \lambda_1,\ldots,\lambda_n \in [0,\infty) ,
\label{LUE_pdf}
\end{multline}
where
\begin{equation}
  c_{m,n} := \frac{1}{m!}\prod^{m}_{j=1}\Gamma(n-j+1)\Gamma(m-j+2) .
\label{LUE_norm}
\end{equation}
The naming relates to the fact that (\ref{LUE_pdf}) is the eigenvalue p.d.f. of
complex Hermitian matrices $ X $ with measure invariant under unitary conjugation
$ X \mapsto UXU^{-1} $, proportional to the generalised Laguerre form
\begin{equation}
  \left(\det X\right)^ae^{-{\rm Tr}(X)} . 
\end{equation}

In multivariate statistics (\ref{LUE_pdf}) is realised as the eigenvalue p.d.f.
for the complex case of the so-called Wishart matrices $ X=Y^{\dagger}Y $.
Here $ Y $ is an $ N\times n \;(N \geq n)$ rectangular matrix of i.i.d. entries
with distribution $ {\rm N}[0,1]+i{\rm N}[0,1] $.
In this setting $ a = N-n $, and so $ a $ is naturally a non-negative integer. 
The spectrum of complex Wishart matrices has found recent application in studies
of wireless communication systems \cite{TV_2004}, where the matrix 
$ Y $ consists of the complex amplitudes of various channels of transmitted 
waves as received by the antennas. 

The matrix structure $ Y^{\dagger}Y $ is relevant to the study of the eigenvalues
of the $ (n+N)\times(n+N) $ Hermitian matrix
\begin{equation}
  \tilde{X} := \begin{pmatrix}
               0_{N\times N} & Y \\ Y^{\dagger} & 0_{n\times n} \end{pmatrix} .
\end{equation} 
Thus one has that $ \tilde{X} $ has in general $ N-n $ zero eigenvalues, with
the remaining eigenvalues given by $ \pm $ the positive square roots of the
eigenvalues of $ Y^{\dagger}Y $ (see e.g. \cite{rmt_Fo}). This matrix 
structure has application to the study of the Dirac equation in the context of
quantum chromodynamics \cite{Ve_1994}. There most interest is in the scaling
behaviour of the smallest eigenvalues.

In the study of matrix Lie algebras one encounters antisymmetric matrices 
($ X^{T}=-X $) with pure imaginary complex elements. Specifically, such matrices
are the Hermitian part of the matrix Lie algebra
\begin{equation}
   i\times( so(n,\mathbb{C}) ) :=
   \{\text{$i$ times $n\times n$ skew symmetric complex matrices}\} .
\end{equation} 
If the independent imaginary complex elements are i.i.d with distribution 
$ i{\rm N}[0,1] $, then the p.d.f. of the positive eigenvalues is proportional 
to
\begin{align}
  \text{$ n=2m $ even:} & \qquad
  \prod^m_{j=1}\exp(-\lambda^2_j) \prod_{1 \leq j<k \leq m}(\lambda^2_k-\lambda^2_j)^2,
  \label{LUE_minushalf} \\
  \text{$ n=2m+1 $ odd:} & \qquad
  \prod^m_{j=1}\lambda^2_j\exp(-\lambda^2_j) \prod_{1 \leq j<k \leq m}(\lambda^2_k-\lambda^2_j)^2.
  \label{LUE_plushalf} 
\end{align}
This ensemble will be denoted by $ AS(n) $. Under the change of variables 
$ \lambda^2_j \mapsto \lambda_j $ these reduce to the ${\rm LUE}_{n,a}$ with parameters 
$ a=-1/2,1/2 $ respectively.

Antisymmetric Hermitian matrices $ X $ can be used to parameterise real orthogonal 
matrices $ R $ with determinant $+1$ (and thus, by definition, members of the 
classical group $ O^{+}(n) $). Thus we can write $ R $ according to a Cayley
transformation
\begin{equation}
   R = \frac{\mathbb{I}_n+iX}{\mathbb{I}_n-iX} .
\end{equation}
Note from this that the property that the eigenvalues of $ X $ come in 
$ \pm $ pairs is consistent with the property that the eigenvalues of $ R $
come in complex conjugate pairs $ e^{\pm i\theta} $. This can be used (see e.g.
\cite{rmt_Fo}) to show that with the matrix $ R $ chosen with uniform (Haar)
measure, the eigenvalue p.d.f for the eigenvalues with angles $ 0 \leq \theta \leq \pi $
is proportional to 
\begin{align}
  \text{$ n=2m $ even:} & \qquad
  \prod_{1 \leq j<k \leq m}(\cos\theta_k-\cos\theta_j)^2,
  \\
  \text{$ n=2m+1 $ odd:} & \qquad
  \prod^m_{j=1}(1-\cos\theta_j)\prod_{1 \leq j<k \leq m}(\cos\theta_k-\cos\theta_j)^2.
\end{align}
Note that for $ \theta_l \to 0 $ these have the same leading behaviour as 
(\ref{LUE_minushalf}), (\ref{LUE_plushalf}) with $ \lambda_l \to 0 $. This is
consistent with the fact that the $ m \to \infty $ scaled joint distribution
function for the $ k $ smallest eigenvalues, $ p_{(k)} $ say, is the same for
both ensembles,
\begin{equation}
  \lim_{m\to \infty}\left(\frac{\pi}{2\sqrt{m}}\right)^k
  p_{(k)}^{AS(n)}(\frac{\pi X_1}{2\sqrt{m}},\ldots,\frac{\pi X_k}{2\sqrt{m}})
 =\lim_{m\to \infty}\left(\frac{\pi}{m}\right)^k
  p_{(k)}^{O^+(n)}(\frac{\pi X_1}{m},\ldots,\frac{\pi X_k}{m}) ,
\label{scaled_rho}
\end{equation}
where $ n=2m,2m+1 $. This scaling is such that the average spacing between 
eigenvalues approaches unity as $ k \to \infty $. The distribution (\ref{scaled_rho})
is of primary importance in the study of the spectral interpretation of 
$L$-functions \cite{KS_1999},\cite{Ru_2005}. From the remark below (\ref{LUE_plushalf})
we know that
\begin{equation}
  \left(\frac{1}{2\sqrt{x_1}}\right) \cdots \left(\frac{1}{2\sqrt{x_k}}\right)
  p_{(k)}^{AS(n)}(\sqrt{x_1},\ldots,\sqrt{x_k})
 = p_{(k)}^{{\rm LUE}_{m,a}}(x_1,\ldots,x_k) ,
\label{}
\end{equation}
where for $ n=2m $, $ a=-1/2 $, while for $ n=2m+1 $, $ a=1/2 $. Consequently
\begin{multline}
  \lim_{m\to \infty}\left(\frac{\pi}{m}\right)^k
  p_{(k)}^{O^+(n)}(\frac{\pi X_1}{m},\ldots,\frac{\pi X_k}{m}) \\
 =\lim_{m\to \infty}\left(\frac{\pi^2}{2m}\right)^k
  X_1\ldots X_k\,p_{(k)}^{{\rm LUE}_{m,a}}(\frac{\pi^2X^2_1}{4m},\ldots,\frac{\pi^2X^2_k}{4m}) .
\label{scaled_On}
\end{multline}
Thus knowledge of the distribution $ p_{(k)}^{{\rm LUE}_{m,a}} $ for 
$ a=\pm 1/2 $ suffices to compute the scaled limit of $ p_{(k)}^{O^+(n)} $.

In the case $ k=1 $ a number of different characterisations of 
$ p_{(1)}^{{\rm LUE}_{m,a}} $, which is the distribution of the smallest 
eigenvalue in $ {\rm LUE}_{m,a} $, are known. First, with $ n \mapsto n+1 $ in 
(\ref{LUE_pdf}) for convenience, the p.d.f. of the smallest eigenvalue is given
by fixing one of the coordinates at $ x_1 $, and integrating the remaining
over $ [x_1,\infty) $. Thus
\begin{multline}
 p_{(1)}^{{\rm LUE}_{n+1,a}}(x_1) = 
 \frac{1}{n!c_{n+1,n+1+a}}e^{-x_1}x_1^a \\
 \times 
 \int^{\infty}_{x_1}d\lambda_1 \ldots \int^{\infty}_{x_1}d\lambda_n\;
 \prod^{n}_{j=1}e^{-\lambda_j}\lambda_j^a(\lambda_j-x_1)^2
 \prod_{1 \leq j<k \leq n}(\lambda_k-\lambda_j)^2 .
\label{LUE_3-pdf}
\end{multline}
One has that
\begin{equation}
   p_{(1)}^{{\rm LUE}_{n+1,a}}(x_1) = -\frac{d}{dx_1}E^{{\rm LUE}_{n+1,a}}(x_1) ,
\label{}
\end{equation}
where
\begin{equation}
   E^{{\rm LUE}_{n+1,a}}(t) :=
   \int^{\infty}_{t}d\lambda_1\cdots  \int^{\infty}_{t}d\lambda_{n+1}
    p(\lambda_1,\ldots,\lambda_{n+1}) ,
\label{}
\end{equation}
is the probability (gap probability) that no eigenvalues are in the interval
$ (0,t) $. 

For integer values of the parameter $ a $ in (\ref{LUE_pdf}), 
$ E^{{\rm LUE}_{n,a}}(t) $ was studied by orthogonal polynomial techniques in 
\cite{Fo_1994a}, where it was evaluated as an $ a\times a $ determinant, and by
the method of Jack polynomials in \cite{Fo_1993}, giving an $a$-dimensional
integral form. In \cite{TW_1994} (see also \cite{FW_2002a}), for general 
$ {\rm Re}(a) > -1 $, it was expressed in terms of a fifth Painlev\'e transcendent.
Explicitly, it was found that
\begin{equation}
   E^{{\rm LUE}_{n,a}}(t) = \exp\left( \int^t_0 \frac{ds}{s}\;U_{\rm V}(s) \right) ,
\label{LUE_PVgap}
\end{equation}
where $ U_{\rm V}(s) $ satisfies the Jumbo-Miwa-Okamoto $\sigma$-form of the
Painlev\'e V equation
\begin{multline}
  (t\sigma'')^2-\left(\sigma-t\sigma'+2(\sigma')^2+(\nu_0+\nu_1+\nu_2+\nu_3)\sigma'\right)^2
  \\
   +4(\nu_0+\sigma')(\nu_1+\sigma')(\nu_2+\sigma')(\nu_3+\sigma') = 0 ,
\label{}
\end{multline}
with parameters
\begin{equation}
   \nu_0=\nu_1 = 0, \quad \nu_2 = n+a, \quad \nu_3 = n .
\label{}
\end{equation}
Alternatively the conventional Painlev\'e V parameters are 
\begin{equation}
   \alpha = \frac{1}{2}a^2, \quad \beta = 0, \quad \gamma = -2n-a-1, \quad \delta = -\frac{1}{2},
\label{}
\end{equation}
and in terms of the Okamoto parameters they are
\begin{equation}
   v_2-v_1 = 0, \quad v_3-v_1 = n+a, \quad v_4-v_1 = n, \quad v_3-v_4 = a.
\label{}
\end{equation}
Because the eigenvalue density is strictly zero for $ \lambda < 0 $, the 
neighbourhood of the smallest eigenvalue is referred to as the hard edge, and is
denoted by $ {\rm HE}_a $. As is consistent with (\ref{scaled_On}), a well
defined limit of (\ref{LUE_PVgap}) is obtained by the scaling $ t \mapsto t/4n $ and
$ n \to \infty $. Thus \cite{TW_1994b}
\begin{equation}
   \lim_{n\to\infty}E^{{\rm LUE}_{n,a}}(t/4n) :=
   E^{{\rm HE}_a}(t) = \exp\left( \int^t_0 \frac{ds}{s}\;U_{\rm III'}(s) \right) ,
\label{}
\end{equation}
where $ U_{\rm III'}(t) $ satisfies the Jimbo-Miwa-Okamoto $\sigma$-form of the 
Painlev\'e \IIId equation 
\begin{equation}
  (t\sigma'')^2-v_1v_2(\sigma')^2+\sigma'(4\sigma'-1)(\sigma-t\sigma')
  -\frac{1}{4^3}(v_1-v_2)^2 = 0 ,
\label{}
\end{equation}
with parameters
\begin{equation}
   v_1 = v_2 = a ,
\label{}
\end{equation}
and subject to the boundary condition
\begin{equation}
  U_{\rm III'}(t) \mathop{\sim}\limits_{t\to 0}
  -\frac{1}{2^{2a+2}\Gamma(a+1)\Gamma(a+2)}t^{a+1} .
\label{}
\end{equation}

Our interest in this paper is the distribution $ p_{(2)}^{{\rm LUE}_{m,a}} $
and its hard edge scaled limit. Analogous to (\ref{LUE_3-pdf}), we see from 
(\ref{LUE_pdf}) that
\begin{multline}
 p^{{\rm LUE}_{n+2,a}}_{(2)}(x_1,x_2) = 
 \frac{1}{n!c_{n+2,n+2+a}}e^{-x_1-x_2}(x_1-x_2)^2(x_1x_2)^a \\
 \times 
 \int^{\infty}_{x_2}d\lambda_1 \ldots \int^{\infty}_{x_2}d\lambda_n\;
 \prod^{n}_{j=1}e^{-\lambda_j}\lambda_j^a(\lambda_j-x_1)^2(\lambda_j-x_2)^2
 \prod_{1 \leq j<k \leq n}(\lambda_k-\lambda_j)^2 ,
\label{LUE_2-pdf}
\end{multline}
where $ x_1 $ denotes the smallest eigenvalue and $ x_2 $ the second smallest 
eigenvalue. In \cite{FH_1994}, for $ a $ integer, this was expressed as an
$ (a+2)\times(a+2) $ determinant. In the hard edge scaled limit this gave
\begin{multline}
   \lim_{n \to \infty}\frac{1}{16n^2}p^{{\rm LUE}_{n+2,a}}_{(2)}(\frac{x_1}{4n},\frac{x_2}{4n})
  =: p^{{\rm HE}_a}_{(2)}(x_1,x_2) 
   \\ 
  = 2^{-4}\left(\frac{x_2}{x_1}\right)^{a}e^{-x_2/4}      
   \\ \times
              \det\left[ \begin{array}{c}
   \left[ I_{j+2-k}(\sqrt{x_2}) \right]_{{\ScSt j=1,\ldots,a}\atop{\ScSt k=1,\ldots a+2}} \\
   \left[ \left(\frac{\displaystyle x_2-x_1}{\displaystyle x_2}\right)^{(k-j)/2}
          I_{j+2-k}(\sqrt{x_2-x_1})
   \right]_{{\ScSt j=1,2}\atop{\ScSt k=1,\ldots a+2}}
                         \end{array} \right] .
\label{}
\end{multline} 
We seek a Painlev\'e type characterisation of (\ref{LUE_2-pdf}) and its scaled
limit, valid for general $ {\rm Re}(a) > -1 $.

One use of knowledge of $ p_{(2)}^{{\rm LUE}_{m,a}} $ is the computation of
the distribution of the spacing between the smallest and the second smallest 
eigenvalues. Denoting this distribution by $ A_{n,a} $ for $ {\rm LUE}_{n+2,a} $,
we have
\begin{equation}
  A_{n,a}(y) := \int^{\infty}_{0} dx_1\; p^{{\rm LUE}_{n+2,a}}_{(2)}(x_1,x_1+y),
 \quad y\in \RR_{+} .
\label{LUE_2-1dbn}
\end{equation}
Important to our subsequent working is a rewrite of (\ref{LUE_2-pdf}) and 
(\ref{LUE_2-1dbn}) in terms of an integral of the form
\begin{multline}
   D_n(x_1,x_2)[w(\lambda)] \\
  := \frac{1}{n!} \int_{I} d\lambda_1 \ldots \int_{I} d\lambda_n\; \prod^n_{l=1}w(\lambda_l)
                \prod^n_{l=1}(\lambda_l-x_1)(\lambda_l-x_2) 
                \prod_{1\leq j<k \leq n}(\lambda_k-\lambda_j)^2 ,
\label{UE_deform_Hankel}
\end{multline}
where $ I $ denotes the support of the weight $ w(\lambda) $. We have
\begin{multline}
   p_{(2)}(x,x+y) = \\
   \frac{1}{c_{n+2,n+2+a}} e^{-(n+1)(x+y)-x}y^2[x(x+y)]^a
    D_n(-y,-y)[\lambda^2(\lambda+x+y)^ae^{-\lambda}\chi_{>0}] ,
\label{LUE_2-pdf.b}
\end{multline}
and 
\begin{equation}
   A_{n,a}(y) = \frac{y^2e^{y}}{c_{n+2,n+2+a}}
   \int^{\infty}_{y}dt\; t^a(t-y)^ae^{-(n+2)t}
        D_n(-y,-y)[\lambda^2(\lambda+t)^ae^{-\lambda}\chi_{>0}] .
\label{LUE_2-1dbn.b}
\end{equation}
These equations exhibit the occurrence of a deformation of the Laguerre weight
\begin{equation}
  w(x;t) := x^2(x+t)^a e^{-x}, \quad x\in \RR_{+} .
\label{deform_Lwgt}
\end{equation}
This deformed weight actually interpolates between two Laguerre weights -
when $ t \to 0 $ then we have the general parameter $ a+2 $ case, whilst if
$ t \to \infty $ in the sector $ -\pi < \arg(t) \leq \pi $
we have the special parameter situation with an exponent of $ 2 $. In fact virtually
all of our analysis can be carried over to the more general situation where the
exponent $ 2 $ is an arbitrary complex parameter suitably restricted.

We begin in Section 2 by revising appropriate results from orthogonal polynomial
system theory and apply this to the particular deformed Laguerre weight
(\ref{deform_Lwgt}). This allows us, in Section 3, to characterise the 
distributions (\ref{LUE_2-pdf.b}) and (\ref{LUE_2-1dbn.b}) by a solution of the 
fifth Painlev\'e equation and its associated linear isomonodromic system (see 
Proposition 3.2). Section 4 is devoted to the determinant evaluations of those
distributions for positive integer values of the parameter $ a $. We proceed in
Section 5 to the study of the hard edge limits
\begin{gather}
  p^{{\rm HE}_a}_{(2)}(x_1,x_2) 
  := \lim_{n\to \infty}\frac{1}{16n^2}
           p^{{\rm LUE}_{n,a}}_{(2)}(\frac{x_1}{4n},\frac{x_2}{4n}) ,
  \label{HE_2pdf}\\
  A_a(z) := \lim_{n\to \infty}\frac{1}{4n}A_{n,a}(\frac{z}{4n}) .
\end{gather}
It is found that these scaled distributions can be characterised by 
the solution of a certain Painlev\'e \IIId equation and its associated linear 
isomonodromic system (see Propositions 5.2, 5.5 and Remark 5.3). In Section 6
this characterisation is used to obtain the high precision numerical values of
statistical characteristics of $ A_a(z) $ for various integer values of 
$ a $ and for the values $ a =\pm 1/2 $, the latter being relevant to 
(\ref{LUE_minushalf}), (\ref{LUE_plushalf}) with the change of variable 
$ \lambda^2_j \mapsto \lambda_j $. Let $ p^{\pm}_{(2)}(x_1,x_2) $ denote the
scaled distribution of the eigenvalues 
$ e^{i\theta_1},e^{i\theta_2} (\theta_1,\theta_2 > 0) $ closest to the origin
in $ O^+(2n+1) $ and $ O^+(2n) $ respectively. With the scaling chosen so that
the bulk density is unity, it follows from (\ref{scaled_On}) and (\ref{HE_2pdf}) that
\begin{equation}
  p^{\pm}_{(2)}(x_1,x_2) = 4\pi^2x_1x_2p^{{\rm HE}_{\pm 1/2}}_{(2)}(\pi^2x^2_1,\pi^2x^2_2) . 
\label{HE_jpdf_pm}
\end{equation}
Consequently
\begin{align}
   A^{\pm}(y) :=& \int^{\infty}_0 dx\, p^{\pm}_{(2)}(x,x+y)
   \nonumber\\
   =&\; 4\pi^2\int^{\infty}_0 dx\, x(x+y)p^{{\rm HE}_{\pm 1/2}}_{(2)}(\pi^2x^2,\pi^2(x+y)^2) .
\label{HE_A_pm}
\end{align}
We use our results for $ p^{{\rm HE}_{\pm 1/2}}_{(2)} $ to provide the high 
precision numerical values of statistical characteristics of $ A^{\pm}(y) $.

\section{Orthogonal Polynomial System}
\setcounter{equation}{0}
\subsection{Semi-classical Orthogonal Polynomials}

Consider the general {\it orthogonal polynomial system} $ \{p_n(x)\}^{\infty}_{n=0} $
defined by the orthogonality relations
\begin{equation}
  \int_{I} dx\; w(x)p_n(x)x^m = \begin{cases} 0 \quad 0\leq m<n \\ h_n \quad m=n \end{cases} ,
\label{ops_orthog}
\end{equation}
with $ I $ denoting the support of the {\it weight} $ w(x) $. 
We give special notation for the coefficients of $ x^n $ and $ x^{n-1} $ in $ p_n(x) $,
\begin{equation}
   p_n(x) = \gamma_n x^n + \gamma_{n,1}x^{n-1} + \ldots .
\label{ops_poly}
\end{equation}
The corresponding {\it monic} polynomials are then
\begin{equation}
   \pi_n(x) = \frac{1}{\gamma_n}p_n(x) .
\label{ops_monic}
\end{equation}
It follows from (\ref{ops_orthog}) that
\begin{equation*}
  \int_{I} dx\; w(x)(p_n(x))^2 = \gamma_nh_n ,
\end{equation*}
and thus for $ p_n(x) $ to be normalised as well as orthogonal we set $ \gamma_nh_n=1 $. 
A consequence of the orthogonality relation is the three term recurrence relation
\begin{equation}
   a_{n+1}p_{n+1}(x) = (x-b_n)p_n(x) - a_np_{n-1}(x), \quad n \geq 1,
\label{ops_threeT}
\end{equation}
and we consider the set of orthogonal polynomials with initial values $ p_{-1} = 0 $ 
and $ p_0 = \gamma_0 $.
The three term recurrence coefficients are related to the polynomial coefficients
by \cite{ops_Sz}, \cite{ops_Fr}
\begin{equation}
   a_n = \frac{\gamma_{n-1}}{\gamma_n}, \quad
   b_n = \frac{\gamma_{n,1}}{\gamma_n}-\frac{\gamma_{n+1,1}}{\gamma_{n+1}}, \quad n \geq 1,
\label{ops_coeffRn}
\end{equation}
along with
\begin{equation}
   b_0 = -\frac{\gamma_{1,1}}{\gamma_{1}},\quad a_0 = 0, \quad \gamma_{0,1}=0 .
\label{ops_coeffRn1}
\end{equation}
A well known consequence of (\ref{ops_threeT}) is the Christoffel-Darboux summation
\begin{equation}
   \sum^{n-1}_{j=0}p_{j}(x)p_{j}(y) 
   = a_n\frac{[p_{n}(x)p_{n-1}(y)-p_{n-1}(x)p_{n}(y)]}{x-y} .  
\label{ops_C-D}
\end{equation}

Central objects in our probabilistic model are the {\it Hankel determinants}
\begin{equation}
   \Delta_n := \det[ \mu_{j+k-2} ]_{j,k=1,\ldots,n} , n\geq 1,
   \quad \Delta_0 := 1 ,
   \label{ops_Hdet}
\end{equation}
and
\begin{equation}
   \Sigma_n := \det\left[
               \begin{array}{cccc}
               \mu_{0} & \cdots & \mu_{n-2} & \mu_{n} \\
               \vdots  & \vdots & \vdots    & \vdots  \\
               \mu_{n-1} & \cdots & \mu_{2n-3} & \mu_{2n-1} \\
               \end{array} \right] , n\geq 1,
   \quad \Sigma_0 := 0 ,
   \label{ops_Sdet}
\end{equation}
defined in terms of the {\it moments} $ \{\mu_n\}_{n=0,1,\ldots,\infty} $ of 
the weight, 
\begin{equation}
   \mu_n := \int_{I} dx\; w(x)x^{n} .
   \label{ops_moment}
\end{equation}
We have integral representations for $ \Delta_n $
\begin{equation}
   \Delta_n = \frac{1}{n!} \int_{I} dx_1 \ldots \int_{I} dx_n\; \prod^n_{l=1}w(x_l) 
                 \prod_{1\leq j<k \leq n}(x_k-x_j)^2 ,\quad n\geq 1 ,
   \label{ops_Hint}
\end{equation}
and $ \Sigma_n $
\begin{equation}
   \Sigma_n = \frac{1}{n!} \int_{I} dx_1 \ldots \int_{I} dx_n\; \prod^n_{l=1}w(x_l) 
              \left(\sum^{n}_{j=1}x_j\right)\prod_{1\leq j<k \leq n}(x_k-x_j)^2 ,
   \quad n \geq 1 .
   \label{ops_Sint}
\end{equation}
The three-term recurrence coefficients are related to these determinants by standard 
result in orthogonal polynomial theory \cite{ops_Sz}, \cite{ops_Fr}
\begin{align}
   a^2_n & = \frac{\Delta_{n+1}\Delta_{n-1}}{\Delta^2_{n}} , \quad n\geq 1 ,
   \label{ops_aSQDelta}\\
   b_n & = \frac{\Sigma_{n+1}}{\Delta_{n+1}}-\frac{\Sigma_n}{\Delta_n} , \quad n\geq 0 , 
   \label{ops_bDelta}\\
   \gamma^2_n & = \frac{\Delta_{n}}{\Delta_{n+1}} , \quad n\geq 0 ,
   \label{ops_gammaDelta}
\end{align}
with initial values
\begin{equation}
   a^2_1 = \frac{\mu_{0}\mu_{2}-\mu^2_{1}}{\mu^2_{0}}, \quad
   b_0 = \frac{\mu_{1}}{\mu_{0}}, \quad \mu_0\gamma^2_0 = 1 .
   \label{TTcoeff_initial}
\end{equation}
The orthogonal polynomials themselves also have a determinantal representation
\begin{equation}
  \sqrt{\Delta_n\Delta_{n+1}}p_n(x) = 
  \det\left[ \begin{array}{ccc}
               \mu_{0} & \cdots & \mu_{n} \\
               \vdots  & \vdots & \vdots  \\
               \mu_{n-1} & \cdots & \mu_{2n-1} \\
               1 & \cdots & x^{n} \\
             \end{array} \right] , \quad n\geq 1,
  \label{ops_Pdet}
\end{equation}
and the integral representation
\begin{equation}
  \sqrt{\Delta_n\Delta_{n+1}}p_n(x) = 
  \frac{1}{n!} \int_{I} dx_1 \ldots \int_{I} dx_n\; \prod^n_{l=1}w(x_l)(x-x_l) 
                 \prod_{1\leq j<k \leq n}(x_k-x_j)^2 .
  \label{ops_Pint}
\end{equation}

Another set of polynomial solutions to the three term recurrence relation are the 
{\it associated polynomials} $ \{p^{(1)}_n(x)\}^{\infty}_{n=0} $, defined by
\begin{equation}
   p^{(1)}_{n-1}(x) := \int_{I} ds\; w(s)\frac{p_n(s)-p_n(x)}{s-x}, \quad n \geq 0 .
\label{ops_assoc}
\end{equation}
In particular these polynomials satisfy
\begin{equation}
   a_{n+1}p^{(1)}_{n}(x) = (x-b_n)p^{(1)}_{n-1}(x)-a_np^{(1)}_{n-2}(x) ,
\label{}
\end{equation}
with the initial conditions $ p^{(1)}_{-1}(x)=0 $, $ p^{(1)}_{0}(x)=\mu_{0}\gamma_{1} $.
Note the shift by one decrement in comparison to the three-term recurrence 
(\ref{ops_threeT}) for the polynomials $ \{p_n(x)\}^{\infty}_{n=0} $. 
We also need the definition of the moment generating function or {\it Stieltjes function}
\begin{align}
   f(x) :&= \int_{I} ds\;\frac{w(s)}{x-s} , \quad x \notin I ,\\
         &= \sum^{\infty}_{n=0} \frac{\mu_n}{x^{n+1}}, \quad x \notin I,\quad x \to \infty .
\label{ops_stieltjes}
\end{align}
We define non-polynomial {\it associated functions} 
$ \{\epsilon_n(x)\}^{\infty}_{n=0} $ by
\begin{equation}
   \epsilon_n(x) := f(x)p_n(x)-p^{(1)}_{n-1}(x) ,
\label{ops_eps}
\end{equation}
which also satisfy the three term recurrence relation (\ref{ops_threeT}), namely
\begin{equation}
   a_{n+1}\epsilon_{n+1}(x) = (x-b_n)\epsilon_{n}(x)-a_n\epsilon_{n-1}(x) ,
\label{}
\end{equation}
subject to the initial values $ \epsilon_{-1}(x)=0 $, $ \epsilon_{0}(x)=\gamma_0f(x) $.
The associated functions have an integral representation analogous to (\ref{ops_Pint})
\begin{multline}
  \sqrt{\Delta_n\Delta_{n+1}}\epsilon_n(x) \\ = 
  \frac{1}{(n+1)!} \int_{I} dx_1 \ldots \int_{I} dx_{n+1}\; \prod^{n+1}_{l=1}\frac{w(x_l)}{x-x_l} 
                 \prod_{1\leq j<k \leq n+1}(x_k-x_j)^2 , \quad x \notin I .
  \label{ops_Eint}
\end{multline}
The polynomials and their associated functions satisfy the Casoratian relation
\begin{equation}
   p_n(x)\epsilon_{n-1}(x)-p_{n-1}(x)\epsilon_{n}(x) = \frac{1}{a_n} , \quad n \geq 1.
\label{ops_casoratian}
\end{equation}

Extending (\ref{ops_poly}) and (\ref{ops_coeffRn}) we have
\begin{multline}
   p_n(x) = \gamma_n \Bigg[ x^n - \left( \sum^{n-1}_{i=0}b_i \right)x^{n-1} \\
   + \left( \sum_{0\leq i<j<n}b_ib_j-\sum^{n-1}_{i=1}a^2_i \right)x^{n-2} 
   + {\rm O}(x^{n-3}) \Bigg] ,
   \label{ops_pExp}
\end{multline}
valid for $ n \geq 1 $, while for the associated functions
\begin{multline}
   \epsilon_n(x) = \gamma^{-1}_n \Bigg[ x^{-n-1} + \left( \sum^{n}_{i=0}b_i \right)x^{-n-2} \\
   + \left( \sum_{0\leq i\leq j\leq n}b_ib_j+\sum^{n+1}_{i=1}a^2_i \right)x^{-n-3} 
   + {\rm O}(x^{-n-4}) \Bigg] ,
   \label{ops_eExp}
\end{multline}
valid for $ n \geq 0 $.

\begin{proposition}[\cite{BC_1990},\cite{Ba_1990},\cite{Ma_1995a}]\label{ops_spectralM}
Let 
\begin{equation}
    \frac{1}{w(x)}\frac{d}{dx}w(x) = \frac{2V(x)}{W(x)} ,
\label{ops_logDerw}
\end{equation}
for $ V,W $ irreducible.
The orthogonal polynomials and associated functions satisfy a system of coupled first 
order linear differential equations with respect to $ x $ ( $ ' \equiv d/dx$)
\begin{align}
  W p'_n & = (\Omega_n-V)p_n-a_n\Theta_np_{n-1} ,\quad n \geq 1,
  \label{ops_spectralD:a} \\
  W p'_{n-1} & = a_n\Theta_{n-1}p_{n}-(\Omega_n+V)p_{n-1} ,\quad n \geq 0,
  \label{ops_spectralD:b}
\end{align}
for certain coefficient functions $ V(x), W(x), \Theta_n(x), \Omega_n(x) $. The associated 
functions $ \epsilon_n, \epsilon_{n-1} $ satisfy precisely the same set of equations.
\end{proposition}

If we define the $ 2\times 2 $ matrix variable
\begin{equation}
   Y_n(x;t) = \begin{pmatrix} p_n(x) & \frac{\displaystyle\epsilon_{n}(x)}{\displaystyle w(x)} \\
                              p_{n-1}(x) & \frac{\displaystyle\epsilon_{n-1}(x)}{\displaystyle w(x)}
              \end{pmatrix}
\label{ops_Ydefn}
\end{equation}
then the above coupled system can be written as
\begin{equation}
   \frac{d}{dx}Y_n(x)
   = \frac{1}{W(x)}\begin{pmatrix} \Omega_n(x)-V(x) & -a_n\Theta_n(x) \\
                                   a_n\Theta_{n-1}(x) & -\Omega_n(x)-V(x)
                   \end{pmatrix}Y_n(x)
\label{ops_YspD}
\end{equation}
It follows that the coefficient functions are specified by
\begin{align}
  \Theta_n & = W\left[ \epsilon_{n}p'_{n} - \epsilon'_{n}p_{n} \right] + 2V\epsilon_{n}p_{n} ,
  \quad n \geq 0, \quad \Theta_{-1}=0 ,
  \label{ops_theta} \\
  \Omega_n & = a_{n}W\left[ \epsilon_{n-1}p'_{n} - \epsilon'_{n}p_{n-1} \right] 
               + a_{n}V\left[ \epsilon_{n}p_{n-1}+\epsilon_{n-1}p_{n} \right] ,
  \quad n \geq 1, \quad \Omega_{0}=0 . 
  \label{ops_omega}
\end{align}

\begin{proposition}[\cite{Ma_1995a}]
The coefficient functions arising in Proposition \ref{ops_spectralM} satisfy the recurrence 
relations
\begin{gather}
 (\Omega_{n+1}-\Omega_n)(x-b_n) = W + a^2_{n+1}\Theta_{n+1}-a^2_n\Theta_{n-1} ,
 \quad n \geq 0 ,
 \label{ops_RR:a} \\
 \Omega_{n+1}+\Omega_n = (x-b_n)\Theta_{n} ,
 \quad n \geq 0
 \label{ops_RR:b}
\end{gather}
\end{proposition}

We will find it necessary to study the {\it zeros} of the orthogonal polynomial $ p_n(x) $
which we denote $ \{x_{1,n} < \ldots < x_{j,n} < \ldots < x_{n,n}\} $. 
They have an electrostatic 
interpretation as the equilibrium positions of the mobile unit charges, and 
there is a set of equations governing these equilibrium positions known as 
the {\it Bethe Ansatz equations}.
\begin{proposition}[\cite{Is_2001}]\label{BAE}
The zeros $ \{x_{j,n}\}^{n}_{j=1} $ of the polynomial $ p_n(x) $ satisfy the coupled 
functional equations
\begin{equation}
  2\sum_{k\neq j}\frac{1}{x_{j,n}-x_{k,n}} 
  = \frac{\Theta'_n(x_{j,n})}{\Theta_n(x_{j,n})}
    - \frac{W'(x_{j,n})+2V(x_{j,n})}{W(x_{j,n})} ,
\end{equation}
for all $ 1 \leq j \leq n $.
\end{proposition}
One can also represent many useful quantities in terms of sums over the zeros and 
we illustrate this with an example. Firstly the consecutive ratios of 
the orthogonal polynomials have a partial fraction decomposition
\begin{equation} 
  \frac{p_{n-1}(x)}{p_n(x)} 
 = -\frac{1}{a_n}\sum^n_{j=1}\frac{W(x_{j,n})}{\Theta_n(x_{j,n})}\frac{1}{x-x_{j,n}} ,
\label{ops_ParFrac}
\end{equation} 
along with
\begin{equation} 
 a^2_n = -\sum^n_{j=1}\frac{W(x_{j,n})}{\Theta_n(x_{j,n})} .
\end{equation}
 
Of particular relevance to our application are the {\it semi-classical class} of orthogonal
polynomial systems \cite{Ma_1987} defined by the property that $ V(x) $ and $ W(x) $ 
in (\ref{ops_logDerw}) are polynomials in $ x $. The zeros of $ W(x) $ define 
finite {\it singularities} of the system of ordinary differential equations (\ref{ops_spectralD:a}),
(\ref{ops_spectralD:b}) and will feature prominently in this study. Let $ x_{r} $ be such
a point with $ r \geq 1 $. Then at $ x_{r} $ the relations (\ref{ops_RR:a}) and 
(\ref{ops_RR:b}) can be combined and integrated to yield
\begin{equation}        
  \Omega^2_n(x_{r}) - V^2(x_{r}) = a^2_n\Theta_n(x_{r})\Theta_{n-1}(x_{r}),
  \quad n \geq 1.
\label{ops_RR:c}
\end{equation}
In fact, at a given finite singular point $ x_{r} $, we can deduce the following bi-linear 
identities, that factorise the one above.
\begin{corollary}
The coefficient functions evaluated at a finite singular point $ x_{r} $ are related to
evaluations of the orthogonal polynomials and associated functions by the 
relations
\begin{align}
  \Omega_{n}(x_{r})+V(x_{r}) & = 2a_nV(x_{r})p_{n}(x_{r})\epsilon_{n-1}(x_{r}) ,
  \quad n \geq 1 ,
  \label{ops_bL:a} \\
  \Omega_{n}(x_{r})-V(x_{r}) & = 2a_nV(x_{r})p_{n-1}(x_{r})\epsilon_{n}(x_{r}) ,
  \quad n \geq 1 ,
  \label{ops_bL:b} \\
  \Theta_{n}(x_{r}) & = 2V(x_{r})p_{n}(x_{r})\epsilon_{n}(x_{r}) .
  \quad n \geq 0 ,
  \label{ops_bL:c}
\end{align}
\end{corollary}

From the theory of Uvarov the following general result for (\ref{UE_deform_Hankel}) is known.
\begin{proposition}[\cite{Uv_1959}]
The quantity $ D_n(x,x)[w(\lambda)] $, 
defined by the equal argument form of (\ref{UE_deform_Hankel}),
is evaluated in terms of the polynomials $ p_n(x) $ orthogonal with 
respect to $ w(x) $ and coefficients $ \gamma_n, \Delta_n $ of this system as 
\begin{equation}
   D_n(x,x)[w(\lambda)] = \frac{\Delta_n}{\gamma_n\gamma_{n+1}}
            [p_n(x)p'_{n+1}(x)-p_{n+1}(x)p'_{n}(x)] .
\label{ops_Uv2}
\end{equation}
\end{proposition}
\begin{proof}
This is a specialisation of Uvarov's general result to the case $ k=0 $ and $ l=2 $ 
where the integral $ D_n(x_1,x_2)[w(\lambda)] $ is a Hankel determinant with respect to the 
weight $ w_{0,2}(x) $, defined by
\begin{equation}
    w_{0,2}(x)dx = d\rho_{0,2}(x), \qquad
    \rho_{0,2}(x) = \int^{x}(s-x_1)(s-x_2)w(s)ds .
\end{equation}
The non-confluent form of the corresponding identity states that
\begin{equation}
   D_n(x_1,x_2)[w(\lambda)] = \left( \Delta_{n+2}\Delta_{n} \right)^{1/2}
                \frac{\det[p_{n+k-1}(x_j)]_{j,k=1,2}}{x_2-x_1} ,
\label{ops_Uv3}
\end{equation}
and the result follows under the confluence $ x_2 \to x_1 $.
\end{proof}
We see from (\ref{ops_Uv2}) that our main task is to obtain appropriate
characterisations of the orthogonal polynomials and their derivatives associated
with the weight (\ref{deform_Lwgt}).

\subsection{Deformed Laguerre Orthogonal Polynomials}

As we noted in the introduction we see the appearance of a deformed Laguerre weight
(\ref{deform_Lwgt}) which is a member of the semi-classical class with the polynomials
$ V,W $ in (\ref{ops_logDerw}) specified by
\begin{equation}
  2V(x;t) = -x^2+(a+2-t)x+2t, \quad W(x;t) = x(x+t) ,
\label{dL_param}
\end{equation}
and has finite singularities at $ x=0,-t $.
The moments have the simple evaluation
\begin{equation}
  \mu_n(t) = t^{a+n+3}\Gamma(n+3)U(n+3,a+n+4;t) , 
  \quad n \geq 0 , \quad |{\rm arg}(t)| < \pi , 
\label{dL_moment}
\end{equation}
where $ U(\alpha,\gamma;z) $ is the confluent hypergeometric that is not analytic at
$ z=0 $. The moments can be written as a sum of two parts one of which is analytic and 
the other non-analytic about $ t=0 $
\begin{multline}
  \mu_n(t) = \Gamma(a+n+3){}_1F_1(-a;-a-n-2;t) \\
  +(-1)^{n+3}\frac{\Gamma(a+1)\Gamma(n+3)}{\Gamma(a+n+4)}t^{a+n+3}{}_1F_1(n+3;a+n+4;t) ,
\label{dL_momentExp}
\end{multline}
where we have to exclude the cases $ a\in \ZZ_{\geq 0} $.

Within the semi-classical class the coefficient functions $ \Theta_n(x), \Omega_n(x) $ 
are polynomials with degree fixed independently of the index $ n $. In particular
we can relate these polynomials to the coefficients of the orthogonal polynomials
themselves.
\begin{proposition}
The coefficient functions are
\begin{gather}
  \Theta_n(x) = 2n+a+3-t-b_n-x, \quad n \geq 0 ,
  \label{dL_thetaEq}\\
  \Omega_n(x) = -\frac{1}{2}x^2 + \frac{1}{2}(2n+a+2-t)x+(n+1)t-a^2_n-\frac{\gamma_{n,1}}{\gamma_n},
   \quad n \geq 1 .
  \label{dL_omegaEq}
\end{gather}
\end{proposition}
\begin{proof}
From the theory of \cite{Ma_1995a} we note that the degrees of the coefficient functions
are $ {\rm deg}\Theta_n \leq \max\{{\rm deg}W-2,{\rm deg}V-1\}=1 $ and
$ {\rm deg}\Omega_n \leq \max\{{\rm deg}W-1,{\rm deg}V-1,{\rm deg}\Theta_n-1,{\rm deg}U+1\}=2 $.
To obtain explicit forms for these we use the definitions (\ref{ops_theta}) and 
(\ref{ops_omega}) and the large $ x \to \infty $ expansions of the polynomials
and associated functions given in (\ref{ops_pExp}) and (\ref{ops_eExp}).
The first equalities in (\ref{dL_thetaEq}) and (\ref{dL_omegaEq}) then follow.
\end{proof}

We will also find it convenient to make the following definitions motivated by 
the above result,
\begin{align}
   \theta_n & := 2n+a+3-t-b_n ,
  \label{dL_thetaDef}\\
   \kappa_n & := (n+1)t-a^2_n-\frac{\gamma_{n,1}}{\gamma_n} .
  \label{dL_kappaDef}
\end{align}

\begin{proposition}\label{dL_spectral}
The spectral derivatives of the polynomials $ p_n(x) $ are
( $'{\hbox{ }} \equiv \partial/\partial x$)
\begin{align}
  x(x+t)p'_{n} & = (nx+\kappa_n-t)p_{n}-a_n(\theta_n-x)p_{n-1}
  \label{dL_spectralD:a} \\
  x(x+t)p'_{n-1} & = a_n(\theta_{n-1}-x)p_{n}-(-x^2+(n+a+2-t)x+t+\kappa_n)p_{n-1}
  \label{dL_spectralD:b}
\end{align}
\end{proposition}
\begin{proof}
This follows from the general form of the spectral derivatives (\ref{ops_spectralD:a})
and (\ref{ops_spectralD:a}), along with the explicit particular forms 
(\ref{dL_thetaEq}) and (\ref{dL_omegaEq}).
\end{proof}

\begin{proposition}\label{dL_deform}
The deformation derivatives of the orthogonal polynomials are 
( $\dot{\hbox{ }} \equiv \partial/\partial t$)
\begin{align}
  t(x+t)\dot{p}_{n} & = \left[(n+1)t-\kappa_n-\frac{1}{2}(x+t)(\theta_n+t)\right]p_{n}
                        +a_n(\theta_n+t)p_{n-1}
  \label{dL_deformD:a} \\
  t(x+t)\dot{p}_{n-1} & = -a_n(\theta_{n-1}+t)p_{n}
                          +\left[\frac{1}{2}(x+t)(\theta_{n-1}+t)+\kappa_n-(n+a+1)t\right]p_{n-1}
  \label{dL_deformD:b}
\end{align}
\end{proposition}
\begin{proof}
We will opt to establish this relation directly from the orthonormality conditions on the
polynomials
\begin{equation}
   \int_{I} dx\;w(x)p_n(x)p_{n-i}(x) = \delta_{i,0}, \quad 0\leq i\leq n .
\nonumber
\end{equation}
Differentiating this with respect to $ t $ leaves us with the relation
\begin{equation}
  0 = a\int_{I} dx\;w(x)\frac{p_np_{n-i}}{x+t}
      +\int_{I} dx\;w(x)\dot{p}_np_{n-i} + \int_{I} dx\;w(x)p_n\dot{p}_{n-i} ,
\label{DD_1}
\end{equation}
where use of the logarithmic derivative of $ w(x) $ has been made. Now we employ
\begin{equation}
  \dot{p}_{n-i} = \frac{\dot{\gamma}_{n-i}}{\gamma_{n-i}}p_{n-i} + \Pi_{n-i-1} , 
\nonumber
\end{equation}
to write the last term of (\ref{DD_1}) as
\begin{equation}
  \int_{I} dx\;w(x)p_n\dot{p}_{n-i} = \frac{\dot{\gamma}_{n-i}}{\gamma_{n-i}}\delta_{i,0}
  = \frac{\dot{\gamma}_{n}}{\gamma_{n}}\int_{I} dx\;w(x)p_np_{n-i} .
\nonumber
\end{equation}
Considering the first term of (\ref{DD_1}) we note 
\begin{align}
  \int_{I} dx\;w(x)\frac{p_n(x)p_{n-i}(x)}{x+t}
  = & \int_{I} dx\;w(x)p_n(x)\frac{p_{n-i}(x)-p_{n-i}(-t)}{x+t} 
\nonumber
  \\
    & \qquad + p_{n-i}(-t)\int_{I} dx\;w(x)\frac{p_{n}(x)}{x+t} ,
\nonumber
  \\
  = \, & p_{n-i}(-t)\int_{I} dx\;w(x)\frac{p_{n}(x)}{x+t} ,
\nonumber
  \\
  = & -p_{n-i}(-t)\epsilon_n(-t) .
\nonumber
\end{align}
But we can recast $ p_{n-i}(-t) $ as
\begin{align}
  p_{n-i}(-t)
  = & \sum^{n}_{j=0} p_{n-j}(-t)\delta_{i,j} ,
\nonumber
  \\
  = & \sum^{n}_{j=0} p_{n-j}(-t)\int_{I} dx\;w(x)p_{n-j}(x)p_{n-i}(x) ,
\nonumber
  \\
  = & \int_{I} dx\;w(x)p_{n-i}(x)\sum^{n}_{j=0}p_{n-j}(-t)p_{n-j}(x) ,
\nonumber
  \\
  = & \int_{I} dx\;w(x)p_{n-i}(x)\sum^{n}_{j=0}p_{j}(-t)p_{j}(x) .
\nonumber
\end{align}
Combining these we deduce that
\begin{equation}
  \int_{I} dx\;w(x)p_{n-i}(x) \left\{
   \dot{p}_n(x)+\frac{\dot{\gamma}_{n}}{\gamma_{n}}p_n(x) 
     -a\epsilon_n(-t)\sum^{n}_{j=0}p_{j}(-t)p_{j}(x) \right\} = 0 ,
\nonumber
\end{equation}
for $ i=0,\ldots,n $. The factor in curly brackets in the integrand must be a polynomial
in $ x $ with degree less than or equal to $ n $, and yet is orthogonal to all polynomials
$ p_j $ for $ j=0,\ldots,n $, and thus must be identically zero. This gives us our first
form for the deformation derivative of $ p_n $,
\begin{equation}
  \dot{p}_n(x) = -\frac{\dot{\gamma}_{n}}{\gamma_{n}}p_n(x)
  + a\epsilon_n(-t)\sum^{n}_{j=0}p_{j}(-t)p_{j}(x) .
\label{DD_2}
\end{equation}
Equating coefficients of $ p_n(x) $ in this relation we find
\begin{equation}
   2\frac{\dot{\gamma}_{n}}{\gamma_{n}} = ap_{n}(-t)\epsilon_n(-t) .
\label{DD_3}
\end{equation}
Furthermore using the Christoffel-Darboux formula (\ref{ops_C-D}), and the above equation, 
we can 
express this derivative solely as a linear combination of $ p_n $ and $ p_{n-1} $
\begin{multline}
  \dot{p}_n(x) 
  = a\epsilon_n(-t)\left[ \frac{1}{2}p_{n}(-t)+a_n\frac{p_{n-1}(-t)}{x+t} \right]p_n(x)
  \\
    -a a_n\epsilon_n(-t)\frac{p_{n}(-t)}{x+t}p_{n-1}(x) .
\label{dL_deformD:c}
\end{multline}
Using the bilinear product relations (\ref{ops_bL:b}) and (\ref{ops_bL:c}) we arrive at
the result (\ref{dL_deformD:a}). The second of the two relations can be found by 
shifting $ n \mapsto n-1 $ in (\ref{dL_deformD:c}) and using the three term recurrence relation.
\end{proof}

In the matrix formulation the spectral derivative take the particular form
\begin{equation}
   \partial_x Y_n(x;t)
  = \left\{ {\mathcal A}_{\infty}+\frac{{\mathcal A}_0}{x}+\frac{{\mathcal A}_t}{x+t}
    \right\} Y_n(x;t) .
\label{dL_YspD}
\end{equation}
The residue matrices are explicitly given by
\begin{align}
  {\mathcal A}_{0}
  & = \frac{1}{t}\begin{pmatrix}
     \kappa_n-t & -a_n\theta_n \\ a_n\theta_{n-1} & -\kappa_n-t
    \end{pmatrix}, \quad \chi_0=0,-2
\label{dL_A0} \\
  {\mathcal A}_{t}
  & = \frac{1}{t}\begin{pmatrix}
     (n+1)t-\kappa_n & a_n(\theta_n+t) \\ -a_n(\theta_{n-1}+t) & \kappa_n-(n+a+1)t
    \end{pmatrix}, \quad \chi_t=0,-a
\label{dL_At} \\
  {\mathcal A}_{\infty}
  & = \begin{pmatrix} 0 & 0 \\ 0 & 1
    \end{pmatrix}
\label{dL_Ainfinity}
\end{align}
Our linear system has two regular singularities at $ x=0,-t $ and an irregular 
singularity at $ x=\infty $ with Poincar\'e index $ 1 $.
In this formulation the deformation derivative is
\begin{equation}
   \partial_t Y_n(x;t)
  = \left\{ {\mathcal B}+\frac{{\mathcal A}_t}{x+t}
    \right\} Y_n(x;t)
\label{dL_YdefD}
\end{equation}
\begin{equation}
  {\mathcal B}
  = \frac{1}{2t}\begin{pmatrix}
     -\theta_n-t & 0 \\ 0 & \theta_{n-1}+t
    \end{pmatrix}
\label{dL_B}
\end{equation}
As we will see in Section \ref{PV_identity} (\ref{dL_YspD}) and (\ref{dL_YdefD}) form  
the monodromy preserving system corresponding to the fifth Painlev\'e
equation with a form equivalent to that discussed by Jimbo \cite{Ji_1982}, in contrast to
other forms studied in 
\cite{JM_1981b}, \cite{FN_1980}, \cite{FMZ_1992} or \cite{IN_1986}.

\begin{corollary}
The polynomial coefficients satisfy the following coupled, first order mixed deformation 
derivative and difference equations
\begin{align}
   2t\frac{\dot{a_n}}{a_n} & = 2+b_{n-1}-b_{n}, \quad n \geq 1,
   \label{dL_deformD:d} \\
   t\dot{b_n} & = a^2_{n}-a^2_{n+1}+b_n ,\quad n \geq 0 ,
   \label{dL_deformD:e}
\end{align}
with the initial $ t=0 $ values for $ b_n $ and $ a^2_n $ given by (\ref{bExpand}) 
and (\ref{aSQExpand}) respectively.
This system of differential equations is equivalent to the Schlesinger equations.
\end{corollary}
\begin{proof}
There are several methods of proof available here. The first is using the general
result of \cite{Ma_1995a}, expressing the deformation derivatives of the polynomial
coefficients in terms of a sum of the coefficient functions over the movable finite singular
points
\begin{align}
   \frac{\dot{a}_n}{a_n} 
  & = \frac{1}{2}\sum_{{r}=1}^{m}\frac{\Theta_{n}(x_{r})-\Theta_{n-1}(x_{r})}{W'(x_{r})}\dot{x}_{r} ,
  \quad n \geq 1 , \\
   \dot{b}_n
  & = \sum_{{r}=1}^{m}\frac{\Omega_{n+1}(x_{r})-\Omega_{n}(x_{r})}{W'(x_{r})}\dot{x}_{r} 
  \quad n \geq 1 .
\end{align}
The only finite singular point contributing here is $ x=-t $.

Alternatively one can find these derivatives from the polynomial derivatives by examining
selected coefficients. For example by considering the coefficients of $ x^{n-1} $ in
(\ref{DD_2}) we have
\begin{equation}
   \dot{\gamma}_{n,1} = \frac{1}{2}a\epsilon_n(-t)p_n(-t)\gamma_{n,1}
                        +a\epsilon_n(-t)p_{n-1}(-t)\gamma_{n-1} ,
\end{equation}
and therefore
\begin{equation}
   \dot{\left(\frac{\gamma_{n,1}}{\gamma_n}\right)} = a\epsilon_n(-t)p_{n-1}(-t) .
\end{equation}
Consequently, using (\ref{ops_coeffRn}), we find that
\begin{align}
   \frac{\dot{a}_n}{a_n} 
  & = \frac{1}{2}a\left[ \epsilon_{n-1}(-t)p_{n-1}(-t)-\epsilon_n(-t)p_n(-t) \right] ,
  \\
   \dot{b}_n
  & = a\left[ a_n\epsilon_{n}(-t)p_{n-1}(-t)-a_{n+1}\epsilon_{n+1}(-t)p_n(-t) \right] ,
\end{align}
which are identical to (\ref{dL_deformD:d}) and (\ref{dL_deformD:e}) respectively.
\end{proof}

\begin{corollary}
The recurrences for the polynomial coefficients are
\begin{multline}
   a^2_{n+2}-a^2_{n} \\
  = 2t+b_{n+1}\left[ 2n+a+6-t-b_{n+1} \right]-b_{n}\left[ 2n+a+2-t-b_{n} \right], \quad n \geq 1 .
  \label{dL_diff:a}
\end{multline}
and
\begin{multline}
   a^2_{n+1}\left[ 2n+a+5-t-b_n-b_{n+1} \right]
   -a^2_{n}\left[ 2n+a+1-t-b_{n-1}-b_{n} \right] \\
  = -b_{n}(b_{n}+t), \quad n \geq 1 .
  \label{dL_diff:b}
\end{multline}
The initial data $ b_0(t) $ and $ a^2_1(t) $ are given by (\ref{TTcoeff_initial}) 
with the evaluation of the moments (\ref{dL_moment}).
\end{corollary}
\begin{proof}
The first relation (\ref{dL_diff:a}) follows by substituting the explicit forms for
the coefficient functions, (\ref{dL_thetaEq}) and (\ref{dL_omegaEq}), into the 
relation (\ref{ops_RR:b}) and requiring equality of the polynomials in $ x $.
Equality is trivial for $ x^2 $ and $ x^1 $, whilst the non-trivial equality for
$ x^0 $ gives (\ref{dL_diff:a}). The second relation (\ref{dL_diff:b}) follows
from the same procedure applied to the recurrence relation (\ref{ops_RR:a}) and
again the only nontrivial result occurs for the $ x^0 $ part.
\end{proof}

This result can also be recovered from a specialisation of work in \cite{BR_1994}.
In their system of monic orthogonal polynomials the three term recurrence 
coefficients $ \beta_n, \gamma_n $ (not to be confused with our use of these symbols
subsequently) are related to ours by
\begin{equation}
   \beta_n = b_n, \quad \gamma_n = a^2_n .
\end{equation}
Equation (39) of this work implies 
\begin{equation}
   \gamma_{n+1}+\gamma_{n}=(2n+3)t+(2n+a+4-t)\beta_n+2\sum^{n-1}_{k=0}\beta_k-\beta^2_n ,
\end{equation}
and differencing this once leads directly to (\ref{dL_diff:a}). In addition their
equation (40) implies
\begin{equation}
  \gamma_{n+1}\beta_{n+1}=(2n+a+5-t-\beta_n)\gamma_{n+1}+2\sum^{n}_{k=1}\gamma_{k}
    +\sum^{n}_{k=0}\beta_k(\beta_{k}+t) .
\end{equation}
Again differencing this once one finds precisely (\ref{dL_diff:b}).
 
A check of the above results can be made when $ t \to 0 $ whilst all parameters are
kept fixed since this, as noted before, corresponds to the Laguerre weight with 
parameter $ a+2 $. Therefore we have 
\begin{align}
   a^2_n(0) &= n(n+a+2) ,
   \label{aSQZero} \\
   b_n(0)   &= 2n+a+3 ,
   \label{bZero} \\
   \Delta_n(0) & = \frac{\prod^n_{j=1}j!\Gamma(j+a+2)}{n!} .
   \label{DeltaZero}
\end{align}
In our normalisation we have 
\begin{equation}
   p_n(x;0) = (-1)^n\left\{ \frac{n!}{\Gamma(n+a+3)} \right\}^{1/2}L^{(a+2)}_n(x) ,
\end{equation}
where $ L^{(\alpha)}_n(x) $ are the standard associated Laguerre polynomials of
degree $ n $ and index $ \alpha $.
We see that 
the spectral derivative equations (\ref{dL_spectralD:a}) and (\ref{dL_spectralD:b}) 
reduce to the standard expressions for the derivative of the Laguerre polynomials,
the coefficients in the right-hand sides of the deformation derivatives 
(\ref{dL_deformD:a}) and (\ref{dL_deformD:b}) vanish,
the recurrence relations (\ref{dL_diff:a}) and (\ref{dL_diff:b}) are identically
satisfied, and
the right-hand sides of (\ref{dL_deformD:d}) and (\ref{dL_deformD:e}) are zero.

In fact we will need to develop expansions about $ t=0 $ in order to characterise
our quantities as particular solutions of difference and differential equations
in Section 3. To this end we have the following result.

\begin{proposition} 
For fixed $ n $, $ a\notin \ZZ_{\geq 0} $ the Hankel determinant (\ref{ops_Hdet}) with the weight
(\ref{deform_Lwgt}) has the expansion about $ t=0 $
\begin{multline}
   \Delta_n(t) = \Delta_n(0) \Big\{ 1 + \frac{a}{a+2}nt \\
   +\frac{1}{4}\frac{a}{a+2}\left(\frac{(a-1)(n+1)}{a+1}+\frac{(a+1)(n-1)}{a+3}
                            \right)nt^2 + {\rm O}(t^3) \\
   -\frac{2\Gamma(a+1)}{\Gamma(a+3)\Gamma^2(a+4)}\frac{\Gamma(a+n+3)}{\Gamma(n)}t^{a+3}
   \left(1+{\rm O}(t)\right)
   + {\rm O}(t^{2a+6}) \Big\} ,
\label{DeltaExpand}
\end{multline}
with $ |{\rm arg}(t)| < \pi $.
Consequently, under the same conditions, the three-term recurrence coefficients have
the expansions about $ t=0 $
\begin{multline}
  b_n(t) = 2n+a+3-\frac{a}{a+2}t
  + \frac{2a(2n+a+3)}{(a+3)(a+2)^2(a+1)}t^2 + {\rm O}(t^3) \\
  + \frac{2}{(a+2)(a+1)\Gamma^2(a+3)}\frac{\Gamma(a+n+3)}{\Gamma(n+1)}t^{a+3}
    \left(1+{\rm O}(t)\right)
   + {\rm O}(t^{2a+6}) ,
\label{bExpand}
\end{multline}
and
\begin{multline}
  a^2_n(t) = n(n+a+2) \Big\{ 1 
  - \frac{2a}{(a+3)(a+2)^2(a+1)}t^2 + {\rm O}(t^3) \\
  - \frac{2}{(a+1)\Gamma(a+3)\Gamma(a+4)}\frac{\Gamma(a+n+2)}{\Gamma(n+1)}t^{a+3}
    \left(1+{\rm O}(t)\right)
  + {\rm O}(t^{2a+6}) \Big\} .
\label{aSQExpand}
\end{multline}
\end{proposition} 
\begin{proof}
We adopt the method of expanding the Hankel determinant by expanding the moments
to leading order 
\begin{multline} 
   \mu_n(t) = \Gamma(a+n+3)+a\Gamma(a+n+2)t+\frac{1}{2}a(a-1)\Gamma(a+n+1)t^2
  + {\rm O}(t^3) \\
  + (-1)^{n+1}\frac{\Gamma(a+1)\Gamma(n+3)}{\Gamma(a+n+4)}t^{n+a+3}+{\rm O}(t^{n+a+4}) ,
\end{multline} 
as $ t \to 0 $ using (\ref{dL_momentExp}).
The determinant can be expanded to leading orders in $ t, t^{a+3} $ and the
resulting determinants evaluated using the identity \cite{Nd_2004}
\begin{equation} 
  \det(\Gamma(z_k+j))_{j,k=0,\ldots,n-1} 
  = \prod^{n-1}_{j=0}\Gamma(z_j)\prod_{0\leq j<k \leq n-1}(z_k-z_j) ,
\end{equation} 
where $ \{z_0,\ldots,z_{n-1} \} $ is an arbitrary sequence not necessarily in 
arithmetic progression.
\end{proof}

We conclude this section by noting some identities relating the polynomial coefficients
and the zeros of the polynomials. Firstly we give the Bethe Ansatz equations for the
zeros of the deformed Laguerre orthogonal polynomials which can be directly deduced from
Proposition \ref{BAE}.
\begin{corollary}
The zeros $ x_{j,n} $ of the deformed Laguerre orthogonal polynomials $ p_n(x) $ 
satisfy the functional equations
\begin{equation}
  \frac{3}{x_{j,n}}+\frac{a+1}{x_{j,n}+t}-\frac{1}{x_{j,n}-\theta_n}
  +2\sum_{k\neq j}\frac{1}{x_{j,n}-x_{k,n}} = 1, \quad 1 \leq j \leq n .
\label{dL_BAE}
\end{equation}
\end{corollary}
According to the electrostatic interpretation,
the terms of (\ref{dL_BAE}) can be interpreted in the following way - the first is the 
interaction of the mobile unit charge at $ x_{j,n} $ with the fixed charge of size $ 3 $ at the
singularity $ x=0 $,
the second with the fixed charge of size $ a+1 $ at the singularity $ x=-t $, 
the third with a fixed charge of size $-1$ at the apparent singularity $ x=\theta_n $,
the fourth the mutual repulsion with the other mobile charges and the term on the right-hand side
is the linear confining potential.
From the partial fraction decomposition (\ref{ops_ParFrac}) specialised to the arguments
$ x=0,-t $ we have the summation identities.
\begin{proposition}
The following summations over the zeros have the explicit evaluations
\begin{align}
  \frac{\kappa_{n}-(n+1)t}{\theta_n+t} 
  & = \sum^n_{j=1}\frac{x_{j,n}}{\theta_n-x_{j,n}},
  \\
  \frac{\kappa_{n}-t}{\theta_n} 
  & = \sum^n_{j=1}\frac{t+x_{j,n}}{\theta_n-x_{j,n}},
  \\  
  \frac{1}{\theta_n+t}\left[n+\frac{\kappa_{n}-t}{\theta_n}\right]
  & = \sum^n_{j=1}\frac{1}{\theta_n-x_{j,n}},
  \label{ops_Zsum1} \\
  a^2_n
  & = \sum^n_{j=1}\frac{x_{j,n}(x_{j,n}+t)}{x_{j,n}-\theta_n} .
\end{align}
\end{proposition}
In addition we can characterise the motion of the zeros with respect to the deformation
variable.
\begin{proposition}
The zeros $ x_{j,n}(t) $ satisfy the differential equation with respect to $ t $
\begin{equation}
   t\dot{x}_{j,n} = \frac{\theta_n+t}{\theta_n-x_{j,n}}x_{j,n} .
\label{dL_ZeroDE}
\end{equation}
\end{proposition}
\begin{proof}
This follows by equating
\begin{equation}
   \frac{\dot{p}_n(x)}{p_n(x)} = \frac{\dot{\gamma}_n}{\gamma_n}
   -\sum^n_{j=1}\frac{\dot{x}_{j,n}}{x-x_{j,n}} ,
\end{equation}
and (\ref{dL_deformD:a}), and then employing (\ref{ops_ParFrac}) along with (\ref{dL_param}),
(\ref{dL_thetaEq}).
\end{proof}

\section{Difference and Differential Equations}
\setcounter{equation}{0}                    

\subsection{Difference Equations}
In the first subsection we derive an alternative difference system in terms of the
new variables $ \theta_n(t), \kappa_n(t) $ as specified by 
(\ref{dL_thetaDef}), (\ref{dL_kappaDef}).
\begin{proposition}
The auxiliary functions $ \theta_n(t), \kappa_n(t) $ satisfy a system of coupled 
first order recurrence relations
\begin{gather}
  \kappa_{n+1}+\kappa_{n} = \theta_n(\theta_n+t-2n-a-3) , \quad n \geq 0,
  \label{dL_diffEQN:a} \\
  \frac{\theta_{n}}{\theta_{n}+t}\frac{\theta_{n-1}}{\theta_{n-1}+t}
  = \frac{(\kappa_n-t)(\kappa_n+t)}{[\kappa_n-(n+a+1)t][\kappa_n-(n+1)t]} ,
  \quad n \geq 1 .
  \label{dL_diffEQN:b}
\end{gather}
The initial values $ \theta_0 $ and $ \kappa_0 $ are given by 
\begin{equation}
  \theta_0(t) 
  = -2t\frac{\int^{\infty}_{0}dx\; e^{-x}x(t+x)^a}
            {\int^{\infty}_{0}dx\; e^{-x}x^2(t+x)^a},
  \quad \kappa_0 = t .
\label{dL_diffEQN:c}
\end{equation}
\end{proposition}
\begin{proof}
The first of the recurrence relations (\ref{dL_diff:a}) can be exactly summed and
the result is
\begin{equation}
  a^2_{n+1}+a^2_{n}=(2n+3)t+(2n+a+4-t)b_n+2\sum^{n-1}_{i=0}b_i-b^2_n .
\label{}
\end{equation}
Recalling the second relation of (\ref{ops_coeffRn}) the summation appearing here
can done by recasting the equation in terms of the new variables and yields
\begin{equation}
  \kappa_{n+1}+\kappa_{n} = -\theta_nb_n ,
\label{}
\end{equation}
which is (\ref{dL_diffEQN:a}).
The second member of the coupled set is most easily found from the general relation
(\ref{ops_RR:c}) evaluated at the finite singular points $ x=0,-t $ and employing 
the new variables. These two key identities are
\begin{gather}
  (\kappa_n+t)(\kappa_n-t) = a^2_n\theta_n\theta_{n-1} , 
  \label{dL_RR4} \\                   
  [\kappa_n-(n+a+1)t][\kappa_n-(n+1)t] = a^2_n(\theta_n+t)(\theta_{n-1}+t) .
  \label{dL_RR5}
\end{gather} 
The ratio of these two identities yields the relation (\ref{dL_diffEQN:b}). 
\end{proof}

There are other recurrence relations which will be used subsequently, and the
first is
\begin{equation}
   a^2_n(\theta_n+\theta_{n-1}+t) = -(2n+a+2)\kappa_n+[n^2+(n+1)(a+2)]t .
  \label{dL_RR1}
\end{equation}
This follows from the subtraction of (\ref{dL_RR4}) from (\ref{dL_RR5}).
The second relation
\begin{equation}
   a^2_{n+1}-a^2_n-b_n-t = 2\kappa_n+b_n\theta_n ,
  \label{dL_RR2}
\end{equation}
is derived by writing the definition of $ \kappa_{n+1}-\kappa_n $ in terms of
the old variables and then employing (\ref{dL_diffEQN:a}). 
The last relation
\begin{equation}
   a^2_{n+1}\theta_{n+1}-a^2_n\theta_{n-1} = b_n(2\kappa_n+b_n\theta_n) ,
  \label{dL_RR3}
\end{equation}
is a consequence of (\ref{dL_diff:b}) along with the use of (\ref{dL_RR2}).

As a consequence of relations (\ref{dL_RR4}) from (\ref{dL_RR5}) we have
\begin{align}
 \frac{\theta_n}{\theta_n+t}\left[n+\frac{\kappa_n-t}{\theta_n}\right]
 & = \frac{\theta_n}{t}\left\{ \frac{\kappa_n-t}{\theta_n}-\frac{\kappa_n-(n+1)t}{\theta_n+t}
                       \right\} ,
 \nonumber\\
 & = \frac{\theta_n}{t}\left\{  a^2_n\frac{\theta_{n-1}}{\kappa_n+t}
                                -a^2_n\frac{\theta_{n-1}+t}{\kappa_n-(n+a+1)t}
                       \right\} ,
 \nonumber\\
 & = -\frac{\kappa_n-t}{\kappa_n-(n+a+1)t}\left[n+a+2+\frac{\kappa_n+t}{\theta_{n-1}} \right] ,
 \label{dL_RR6}
\end{align}
and
\begin{equation}
  n+\frac{\kappa_n-t}{\theta_n} 
  + \left.\left\{ n+a+2+\frac{\kappa_n+t}{\theta_{n-1}}\right\}\right|_{n\mapsto n+1}
   = \theta_n+t .
\label{dL_RR7}
\end{equation}

\subsection{Reduction to Painlev\'e V}\label{PV_identity}
Here we will identify the fifth Painlev\'e system as the solution to our system
of equations characterising the deformed Laguerre orthogonal polynomial system.
This is most simply seen in terms of the new variables  $ \theta_n, \kappa_n $
rather than the basic orthogonal polynomial variables $ a_n, b_n $.
\begin{proposition}\label{dL_CODE}
The auxiliary quantities $ \theta_n(t), \kappa_n(t) $ satisfy the coupled
first order ordinary differential equations
\begin{equation}
   t\dot{\theta}_{n} = 2\kappa_n+\theta_n(2n+a+3-t-\theta_n) ,
\label{PV:a}
\end{equation}
and
\begin{multline}
  t\dot{\kappa}_n = \left( \frac{1}{\theta_n+t}+\frac{1}{\theta_n} \right)\kappa^2_n
            + \left( 2n+a+3-(2n+a+2)\frac{t}{\theta_n+t} \right)\kappa_n
  \\
  -[n^2+(n+1)(a+2)]t-\frac{t^2}{\theta_n}+(n+1)(n+a+1)\frac{t^2}{\theta_n+t} .
\label{PV:b}
\end{multline}
Equations (\ref{PV:a}) and (\ref{PV:b}) can be solved in terms of the fifth
Painlev\'e system 
\begin{equation}
  \theta_n = t\frac{q}{1-q}, \qquad \kappa_n = t(1+qp) ,
\label{PV:c}
\end{equation}
where $ q,p $ are the Hamiltonian variables of the Okamoto \PV \cite{Ok_1987b} system
with the parameters
\begin{equation}
  \alpha = \frac{a^2}{2}, \quad \beta = -2, \quad 
  \gamma = -(2n+a+3), \quad \delta = -\frac{1}{2} ,
\label{PV:d}
\end{equation}
or 
\begin{equation}
  v_2-v_1 = -2, \quad v_3-v_1 = n+a, \quad v_4-v_1 = n, \quad v_3-v_4 = a .
\label{}
\end{equation} 
The solutions satisfy the boundary value data at $ t=0 $
\begin{multline}                                                                
  \theta_n(t) \mathop{=}\limits_{t \to 0}
  - \frac{2}{a+2}t-\frac{2a(2n+a+3)}{(a+3)(a+2)^2(a+1)}t^2+{\rm O}(t^3) \\
  - \frac{2}{(a+2)(a+1)\Gamma^2(a+3)}\frac{\Gamma(a+n+3)}{\Gamma(n+1)}t^{a+3}
    \left(1+{\rm O}(t)\right)+{\rm O}(t^{2a+6}) ,
\label{dL_BV:a}
\end{multline}
and
\begin{multline}                                                                
  \kappa_n(t) \mathop{=}\limits_{t \to 0}
    \frac{2n+a+2}{a+2}t
  + \frac{4n(n+a+2)a}{(a+3)(a+2)^2(a+1)}t^2+{\rm O}(t^3) \\
  + \frac{2}{(a+2)(a+1)\Gamma^2(a+3)}\frac{\Gamma(a+n+3)}{\Gamma(n)}t^{a+3}
    \left(1+{\rm O}(t)\right)+{\rm O}(t^{2a+6}) ,
\label{dL_BV:b}
\end{multline}
provided $ a\notin \ZZ_{\geq 0} $ and $ |{\rm arg}(t)| < \pi $.
\end{proposition}
\begin{proof}
If we employ (\ref{dL_RR2}) in (\ref{dL_deformD:e}) and then substitute for $ b_n $
in terms of $ \theta_n $ then the result is (\ref{PV:a}). Let us define the shorthand
notation
\begin{equation}
  \Gamma_n := \frac{\gamma_{n,1}}{\gamma_n} .
\label{Gamma_defn}                                         
\end{equation}
Furthermore when the deformation derivative (\ref{dL_deformD:e}) is summed on the free 
index the result is 
\begin{equation}                                                   
  t\dot{\Gamma}_{n} = a^2_n+\Gamma_n .
  \label{dL_PV1}
\end{equation}
Now if compute the deformation derivative of $ \kappa_n $ and use (\ref{dL_deformD:d})
along with (\ref{dL_PV1}) we arrive at
\begin{equation}
  t\dot{\kappa}_{n} = \kappa_n-a^2_n(\theta_n-\theta_{n-1}) .
  \label{dL_PV2}
\end{equation}
Now the idea is to eliminate $ a^2_n $ and $ \theta_{n-1} $, which appear in 
(\ref{dL_PV2}), through use of the recurrence relations. Equation (\ref{dL_diffEQN:b})
is a linear equation for $ \theta_{n-1} $ in terms of the unshifted variables and the
solution can be substituted into (\ref{dL_RR1}) yielding a linear relation for 
$ a^2_n $ in terms of unshifted variables. Then both solutions can be substituted
into (\ref{dL_PV2}) and the result is (\ref{PV:b}).
One can easily verify that the transformation to the Hamiltonian variables $ q,p $
(\ref{PV:c}) with the parameters (\ref{PV:d}) yields the Hamilton equations of 
motion for the Hamiltonian in \cite{Ok_1987b}.
\end{proof}

\begin{remark}
A few remarks can now be made regarding the identification of the recurrences
(\ref{dL_diffEQN:a}) and (\ref{dL_diffEQN:b}). This is different in appearance from the 
discrete integrable equations that arose in the study of the Laguerre unitary 
ensemble \cite{FW_2002a} which were explicitly identified with the system 
in the Sakai scheme, with the rational surface $ D^{(1)}_5 \to E^{(1)}_6 $,
and has a continuous limit of Painlev\'e IV . In fact we find that the variables of
the latter system can be expressed in terms of our own
\begin{equation}
   x_n = \frac{\kappa_n+\theta_n(n+a+1-t-\theta_n)}{\theta_n+t}, \quad
   y_n = -\frac{t}{\theta_n} ,
\end{equation}
and it is clear that one cannot transform (\ref{dL_diffEQN:a}) and (\ref{dL_diffEQN:b})
into this system using such a transformation. Also the two systems arise as different
Schlesinger-type transformations - in our case as a sequence where 
$ \alpha_0 \mapsto \alpha_0+1, \alpha_2 \mapsto \alpha_2-1 $ whereas in the other case as
$ \alpha_0 \mapsto \alpha_0+1, \alpha_3 \mapsto \alpha_3-1 $.
\end{remark}

\begin{remark}
In \cite{FW_2002a} two fundamental quantities were studied - the $\tau$-function
$ \tau[n](t) $ and its logarithmic derivative the $\sigma$-function 
$ V_n(t;a,\mu) $. Their relation to the objects of the present work are
\begin{equation}
   \tau[n] = c(n,a)n!e^{-nt}t^{-n^2-n(a+4)}\Delta_n(t) ,
\label{}                                         
\end{equation}
where $ c(n,a) $ is an unspecified constant and
\begin{equation}
   V_n(t;a,2) = -nt-4n+t\frac{d}{dt}\log\Delta_n(t) .
\label{}                                         
\end{equation}
We also note that
\begin{equation}
   V_n(t;a,2) = \Gamma_n+n(n+a-2-t)
\label{}
\end{equation}
and consequently
\begin{align}
   \Gamma_n & = -n(n+a+2)+t\frac{d}{dt}\log\Delta_n(t)
   \label{GammaDelta} \\
   b_n & = 2n+a+3+t\frac{d}{dt}\log\frac{\Delta_n(t)}{\Delta_{n+1}(t)}
   \label{thetaDelta}
\end{align}
\end{remark}

\begin{remark}
The new variable $ \Gamma_n(t) $ possess an expansion as $ t \to 0 $ which can be 
directly found from (\ref{DeltaExpand}) and (\ref{bExpand})
\begin{multline}              
  n(n+a+2)+\Gamma_n(t) 
  = \frac{a}{a+2}nt
  - \frac{2n(n+a+2)a}{(a+3)(a+2)^2(a+1)}t^2+{\rm O}(t^3) \\
  - \frac{2}{(a+2)(a+1)\Gamma(a+3)\Gamma(a+4)}\frac{\Gamma(a+n+3)}{\Gamma(n)}t^{a+3}
    \left(1+{\rm O}(t)\right)+{\rm O}(t^{2a+6}) ,
\label{dL_BV:c}
\end{multline}
again provided $ a\notin \ZZ_{\geq 0} $.
\end{remark}

\begin{remark}
The use of Proposition \ref{dL_CODE} is in the computation of the orthogonal
polynomials in (\ref{ops_threeT}) corresponding to the weight (\ref{deform_Lwgt}).
For this we note from (\ref{dL_thetaDef}) that
\begin{equation}
  b_n = 2n+a+3-t-\theta_n ,
\end{equation}
while (\ref{dL_kappaDef}) together with (\ref{ops_coeffRn}), (\ref{ops_coeffRn1})
show 
\begin{align}
  a^2_n & = (n+1)t-\frac{\gamma_{n,1}}{\gamma_n}-\kappa_n \\
        & = (n+1)t+\sum^{n-1}_{j=0}b_j-\kappa_n ,
\end{align}
(the quantity $ a_n $ is positive, so the positive square root of this equation
is to be taken). Further, it follows from (\ref{ops_aSQDelta}) and (\ref{ops_coeffRn})
that
\begin{equation}
  \frac{1}{\gamma^2_0\gamma^2_1 \cdots \gamma^2_{n-1}} = \Delta_n ,
\end{equation}
where each $ \gamma_j $ is the coefficient of $ x^j $ in $ p_j(x) $ as specified
by (\ref{ops_poly}). All terms in the equation (\ref{ops_Uv2}) for $ D_n(x,x) $
are then known, and the task is then to compute the integral as required by 
(\ref{LUE_2-1dbn.b}). 
\end{remark}

An alternative system of coupled first order ordinary differential equations 
which will be used for scaling to the hard edge is given in the following
proposition.
\begin{proposition}\label{dL_theta-GammaDE}
The variables $ \theta_n(t), \Gamma_n(t) $ satisfy the coupled first
order ordinary differential equations
\begin{multline}
   t\dot{\theta}_n = \theta_n 
  - \Big\{ \theta^4_n-2(2n+a+2-t)\theta^3_n
  \\
    +[4\Gamma_n+(2n+a+2-t)^2-4(n+1)t]\theta^2_n 
  \\
    +4t[\Gamma_n+n(n+a+2-t)+a+2-t]\theta_n + 4t^2 \Big\}^{1/2} ,
\label{dL_thetaDE}
\end{multline}
and
\begin{multline}
   t\dot{\Gamma}_n = (n+1)t + \frac{1}{2}\theta_n(2n+a+2-t-\theta_n)
  + \Big\{ \theta^4_n-2(2n+a+2-t)\theta^3_n 
  \\
    +[4\Gamma_n+(2n+a+2-t)^2-4(n+1)t]\theta^2_n 
  \\
    +4t[\Gamma_n+n(n+a+2-t)+a+2-t]\theta_n + 4t^2 \Big\}^{1/2} .
\label{dL_GammaDE}
\end{multline}
\end{proposition}
\begin{proof}
We proceed in a series of steps. Firstly we use (\ref{dL_RR4}) to solve for
$ \theta_{n-1} $ in terms of $ \theta_n, \kappa_n $ and $ \Gamma_n $. In
the second step we substitute this solution for $ \theta_{n-1} $ into 
(\ref{dL_RR5}) and solve the following quadratic equation for 
$ \kappa_n $ in terms of $ \theta_n $ and $ \Gamma_n $,
\begin{equation}
  \kappa^2_n+\theta_n(2n+a+2-t-\theta_n)\kappa_n
  +[(n+1)t-\Gamma_n]\theta_n(\theta_n+t)-[n^2+(n+1)(a+2)]t\theta_n-t^2=0 .
\end{equation}
The choice of the sign of the square-root branch follows from the expansions 
(\ref{dL_BV:a}) and (\ref{dL_BV:c}) on one hand, and on the other hand noting that 
as $ t \to 0 $
\begin{multline}
  \Big\{ \theta^4_n-2(2n+a+2-t)\theta^3_n \\
  +[4\Gamma_n+(2n+a+2-t)^2-4(n+1)t]\theta^2_n \\
    +4t[\Gamma_n+n(n+a+2-t)+a+2-t]\theta_n + 4t^2 \Big\}^{1/2} \\
  = \frac{2a(2n+a+3)}{(a+3)(a+2)^2(a+1)}t^2+{\rm O}(t^3) .
\label{}
\end{multline}
In the final step we use these solutions for $ \theta_{n-1} $ and $ \kappa_n $ in
(\ref{dL_PV2}) and (\ref{dL_PV1}).
\end{proof}

In preparation for the hard edge scaling limit we need to make evaluations
of the polynomials at the finite singular points. Firstly considering $ x=0 $
we note that
\begin{equation}
  \frac{\pi_n(0)}{\pi_{n-1}(0)} = a_n\frac{p_n(0)}{p_{n-1}(0)}
  = \frac{\kappa_n+t}{\theta_{n-1}}
  = a^2_n\frac{\theta_n}{\kappa_n-t} ,
\label{dL_Pratio:a}
\end{equation}
as follows immediately from (\ref{dL_spectralD:a}). Furthermore we also have
\begin{equation}
   t\dot{(\log\pi_n(0))} = t\dot{(\log p_n(0))}+\frac{1}{2}(\theta_n+t)
  = n+\frac{\kappa_n-t}{\theta_n} ,
\label{dL_PDderiv}
\end{equation}
which follows from (\ref{dL_deformD:a}) and the above result. The corresponding
result for the polynomial ratio at $ x=-t $ is
\begin{equation}
  \frac{p_n(-t)}{p_{n-1}(-t)}
  = \frac{1}{a_n}\frac{\kappa_n-(n+a+1)t}{\theta_{n-1}+t}
  = a_n\frac{\theta_n+t}{\kappa_n-(n+1)t} .
\label{dL_Pratio:b}
\end{equation}

After \cite{Ap_1993} we define the orthogonal polynomial ratios
\begin{equation}
   Q_n(x;t) := \frac{p_n(x;t)}{p_n(0;t)} ,
\label{dL_Qdefn}
\end{equation}
because we are interested in the scaling properties of the orthogonal polynomial
system at the edge of their interval of orthogonality, $ x=0 $.
Then (\ref{ops_threeT}) and Propositions \ref{dL_spectral} and \ref{dL_deform} 
can be translated into the following three corollaries.
\begin{corollary}
The three-term recurrence for $ \{Q_n\}_{n=0,1,\ldots} $ system is
\begin{equation}
  b_n(Q_{n+1}+Q_{n-1}-2Q_{n})
  +(b_n+2\frac{\kappa_n-t}{\theta_n})(Q_{n+1}-Q_{n-1})+2xQ_{n} = 0 .
\label{dL_TTRR}
\end{equation}
\end{corollary}

\begin{corollary}\label{QQ_SD}
The spectral derivatives of $ Q_{n}, Q_{n-1} $ are
\begin{align}
  x(x+t)Q'_n & = nxQ_{n}+(\kappa_n-t)\left[ Q_{n}-Q_{n-1}+\frac{x}{\theta_n}Q_{n-1} \right] ,
\label{dL_S:c} \\
  x(x+t)Q'_{n-1} & = x[x-(n+a+2-t)]Q_{n-1}
                      +(\kappa_n+t)\left[ Q_{n}-Q_{n-1}-\frac{x}{\theta_{n-1}}Q_{n} \right] .
\label{dL_S:d}
\end{align}
\end{corollary}

\begin{corollary}\label{QQ_DD}
The deformation derivatives of $ Q_{n}, Q_{n-1} $ are
\begin{align}
  t(x+t)\dot{Q}_n & = -x(n+\frac{\kappa_n-t}{\theta_n})Q_{n}
                      +(\kappa_n-t)\frac{\theta_n+t}{\theta_n}\left[ Q_{n-1}-Q_{n} \right] ,
\label{dL_D:c} \\
  t(x+t)\dot{Q}_{n-1} & = x[n+a+2+\frac{\kappa_n+t}{\theta_{n-1}}]Q_{n-1}
                      +(\kappa_n+t)\frac{\theta_{n-1}+t}{\theta_{n-1}}\left[ Q_{n-1}-Q_{n} \right] .
\label{dL_D:d}
\end{align}
\end{corollary}

As we noted earlier the polynomial ratio $ Q_n(x;t) $ has a product representation
\begin{equation}
  Q_n(x;t) = \prod^{n}_{j=1}\left(1-\frac{x}{x_{j,n}}\right) ,
\end{equation}
where again $ x_{j,n} $ is the $j$-th zero of the polynomial.
We can use this fact to compute sums of the inverse powers of the zeros from the above
differential equations.
\begin{proposition}
The increment of the sum of the reciprocals of the zeros going from $ n-1 $ to 
$ n $ is given by
\begin{equation}
  3t\left\{ -\sum^{n}_{j=1}\frac{1}{x_{j,n}}+\sum^{n-1}_{j=1}\frac{1}{x_{j,n-1}} \right\}
  = n+\frac{\kappa_n-t}{\theta_n}+n+a+2+\frac{\kappa_n+t}{\theta_{n-1}}-t .
\label{dL_SumRecipZero}
\end{equation}
\begin{proof}
The required increment of the sum of the reciprocals of the zeros is the order $ x $ term
in the expansion of $ \psi_n(x;t) := Q_n(x;t)/Q_{n-1}(x;t) $ about $ x=0 $, and this can 
evaluated from the same
expansion of the differential equation for $ \psi_n $ with respect to $ x $. This
latter differential equation is easily found from (\ref{dL_S:c}) and (\ref{dL_S:d}) and is
\begin{equation}
  x(x+t)\psi'_n = \frac{x-\theta_n}{\theta_n}(\kappa_n-t)
 + \big( 2\kappa_n+x[-x+2n+a+2-t] \big)\psi_n
 + \frac{x-\theta_{n-1}}{\theta_{n-1}}(\kappa_n+t)\psi^2_n .
\end{equation}
\end{proof}
\end{proposition}

At the other finite singular point, $ x=-t $, we have as a consequence of these corollaries
\begin{equation}
  \frac{Q_n(-t;t)}{Q_{n-1}(-t;t)}
  = \frac{\kappa_n-t}{\kappa_n-(n+1)t}\frac{\theta_n+t}{\theta_n} 
  = \frac{\kappa_n-(n+a+1)t}{\kappa_n+t}\frac{\theta_{n-1}}{\theta_{n-1}+t} ,
\label{dL_singQ:a}
\end{equation}
and
\begin{equation}
  \frac{d}{dt}Q_n(-t;t)
  = \frac{t-\kappa_n-n\theta_n}{\theta_n(\theta_n+t)}Q_n(-t;t) .
\label{dL_singQ:b}
\end{equation} 
Note that the derivative with respect to $ t $ in the latter equation is a total derivative.

To complete our preparations for the hard edge scaling we need to identify two 
polynomial variables
that will scale to independent variables in the scaling limit. The first is the 
orthogonal polynomial ratio $ Q_n $, and for the second a number of choices could be
made but a simple choice is 
\begin{equation}
  R_n := Q_n-Q_{n-1} .
\label{dL_Rdefn}
\end{equation}

\begin{corollary}
The spectral derivatives of $ Q_n, R_n $ are
\begin{align}
  x(x+t)Q'_n & = x\left( n+\frac{\kappa_n-t}{\theta_n} \right)Q_{n}
               +(\kappa_n-t)\frac{\theta_n-x}{\theta_n}R_{n} ,
\label{dL_S:a} \\
  x(x+t)R'_{n} & = 
  x\left[ n+\frac{\kappa_n-t}{\theta_n}+n+a+2+\frac{\kappa_n+t}{\theta_{n-1}}-x-t \right]Q_{n}
  \label{dL_S:b} \\ 
  & +\left[ -x\left( n+\frac{\kappa_n-t}{\theta_n} \right)+x(x+t)-(a+2)x-2t \right]R_{n} .
\nonumber
\end{align}
\end{corollary}
\begin{proof}
This follows from Corollary \ref{QQ_SD}.
\end{proof}

\begin{corollary}
The deformation derivatives of $ Q_n, R_n $ are
\begin{align}
  t(x+t)\dot{Q}_n & = -x\left( n+\frac{\kappa_n-t}{\theta_n} \right)Q_{n}
                      -(\kappa_n-t)\frac{\theta_n+t}{\theta_n}R_{n} ,
\label{dL_D:a} \\
  t(x+t)\dot{R}_{n} & = 
     -x\left[ n+\frac{\kappa_n-t}{\theta_n}+n+a+2+\frac{\kappa_n+t}{\theta_{n-1}} \right]Q_{n}
  \label{dL_D:b} \\
  & +\left[ x\left( n+a+2+\frac{\kappa_n+t}{\theta_{n-1}} \right)
             +(\kappa_n+t)\frac{\theta_{n-1}+t}{\theta_{n-1}}
             -(\kappa_n-t)\frac{\theta_{n}+t}{\theta_{n}} \right]R_{n} .
\nonumber 
\end{align}
\end{corollary}
\begin{proof}
This follows from Corollary \ref{QQ_DD}.
\end{proof}

\subsection{Inequalities and Bounds}
A key step in proving our hard edge scaling limits will be bounds on the variables 
$ \theta_n, \kappa_n $ and some auxiliary quantities. 
The first step is the following result.

\begin{lemma}\label{Lemma_Ineq}
The variables $ \theta_n(t), \kappa_n(t) $ satisfy the inequalities
\begin{gather}
  \frac{\theta_n+t}{\kappa_n-(n+1)t} < \frac{\theta_n}{\kappa_n-t} < 0 ,
  \label{dL_Ineq:a} \\
  \frac{\theta_{n-1}}{\kappa_n+t} < \frac{\theta_{n-1}+t}{\kappa_n-(n+a+1)t} < 0 ,
  \label{dL_Ineq:b}
\end{gather}
for all positive, real and bounded $ t $ and $ n \geq 1 $.
\end{lemma}
\begin{proof}
That the ratios given in (\ref{dL_Ineq:a}) and (\ref{dL_Ineq:b}) are negative is a 
consequence of the fact that the polynomial $ p_n(x) $ evaluated on the negative
real axis, i.e. exterior to the interval of orthogonality, has a fixed sign.
Specifically $ (-1)^np_n(-y) > 0 $ for real, positive $ y $. Using the ratio
relations (\ref{dL_Pratio:a}) and (\ref{dL_Pratio:b}) we have the upper bounds.
From the Christoffel-Darboux formula (\ref{ops_C-D}) at equal arguments we note
that
\begin{equation}
   p_{n-1}(x)p'_{n}(x)-p'_{n-1}(x)p_{n}(x) > 0 ,
\end{equation}
and from the above $ p_{n-1}(x)p_{n}(x) < 0 $ for $ x \in -\RR_+ $ we conclude
\begin{equation}
   \frac{p'_{n}(x)}{p_{n}(x)} < \frac{p'_{n-1}(x)}{p_{n-1}(x)} ,
\end{equation}
under the conditions on $ x $. Integrating this inequality from $ 0 $ to 
$ -y \in -\RR_+ $ we arrive at
\begin{equation}
   \frac{p_{n}(-y)}{p_{n-1}(-y)} < \frac{p_{n}(0)}{p_{n-1}(0)} < 0.
\end{equation}
Then identifying these ratios with (\ref{dL_Pratio:a}) and (\ref{dL_Pratio:b}) in the
case $ y=t $ leads to the relative inequalities.
\end{proof} 

The above set of inequalities must all apply simultaneously and we see in fact that
it implies restrictions on the variables $ \theta_n, \kappa_n $.
\begin{lemma}
For bounded $ t \in \RR_+ $ and $ n \geq 0 $ the variables $ \theta_n, \kappa_n $ satisfy 
inequalities which place them in one of three cases, as illustrated in Figure 1 -\\
\noindent
Case I:
\begin{gather}
  0 < \theta_n ,
  \label{dL_Ineq:c} \\
  \theta_n(\theta_n+t-2n-a-3)+t < \kappa_n < t-n\theta_n ,
  \label{dL_Ineq:d}
\end{gather}
\noindent
Case II:
\begin{gather}
  -t \leq \theta_n \leq 0 ,
  \label{dL_Ineq:e} \\
  t-n\theta_n \leq \kappa_n \leq t+\min\{nt,\theta_n(\theta_n+t-2n-a-3)\} \leq (n+1)t ,
  \label{dL_Ineq:f}
\end{gather}
\noindent
Case III:
\begin{gather}
  \theta_n < -t ,
  \label{dL_Ineq:g} \\
  (n+1)t < \kappa_n < t-n\theta_n .
  \label{dL_Ineq:h}
\end{gather}
\end{lemma}
\begin{proof}
From the three inequalities implied by (\ref{dL_Ineq:a}) we see that 
$ \theta_n \gtrless 0 $ according as $ \kappa_n \lessgtr t $, 
$ \theta_n+t \gtrless 0 $ according as $ \kappa_n \lessgtr (n+1)t $, and
$ \theta_n(\theta_n+t) \gtrless 0 $ according as $ \kappa_n \lessgtr t-n\theta_n $.
In (\ref{dL_Ineq:b}) we make the replacement $ n \mapsto n+1 $ and employ
(\ref{dL_diffEQN:a}) to eliminate $ \kappa_{n+1} $. This inequality now reads
\begin{equation}
   -t-\frac{\kappa_n-(n+1)t}{\theta_n+t} < -\frac{\kappa_n-t}{\theta_n}
   < b_n=2n+a+3-t-\theta_n .
\end{equation}
Consequently these three inequalities imply 
$ \theta_n \gtrless 0 $ according as $ \kappa_n-t \gtrless \theta_n(\theta_n+t-2n-a-3) $,
$ \theta_n+t \gtrless 0 $ according as $ \kappa_n-(n+1)t \gtrless (\theta_n+t)(\theta_n-2n-a-3) $,
and
$ \theta_n(\theta_n+t) \gtrless 0 $ according as $ \kappa_n \lessgtr \theta_n(\theta_n+t-n)+t $.
Combining these sets of inequalities leads to the three cases above. 
\end{proof}

\begin{figure}\label{Fig1}
\setlength{\unitlength}{0.24pt}
\linethickness{0.2mm}
\begin{picture}(1500,900)(0,0)
\put(1118,471){\makebox(0,0){$\theta_n$}}
\put(530,875){\makebox(0,0){$\kappa_n$}}
\put(636,376){\makebox(0,0){\footnotesize I}}
\put(460,581){\makebox(0,0){\footnotesize II}}
\put(238,657){\makebox(0,0){\footnotesize III}}
\put(3,603){\makebox(0,0){$\scriptstyle (n+1)t$}}
\put(46,517){\makebox(0,0){$\scriptstyle t$}}
\put(390,24){\makebox(0,0){$\scriptstyle -t$}}
\put(1060,150){\makebox(0,0){$\scriptstyle t-n\theta_n$}}
\put(370,120){\makebox(0,0){$\scriptstyle (\theta_n+t)(\theta_n-2n-a-3)+(n+1)t$}}
\put(629,813){\makebox(0,0){$\scriptstyle \theta_n(\theta_n+t-n)+t$}}
\put(724,496){\makebox(0,0){$\scriptstyle \theta_n(\theta_n+t-2n-a-3)+t$}}
\put(510,40){\vector(0,1){820}}
\put(398,40){\line(0,1){820}}
\put(60,474){\vector(1,0){1024}}
\put(60,517){\line(1,0){1024}}
\put(60,603){\line(1,0){1024}}
\curve(62,861, 1070,87)
\curve(274.8,863.15, 521.0,516.28333333333, 767.6,863.15)
\curve(398.0,839.5, 706.0,297.52083333333, 1014.0,839.5)
\curve(330.8,865.3, 706.0,61.020833333333, 1081.0,865.3)
\end{picture}
\caption{A pictorial form of the inequalities taking the example of $ a=1/2 $,
$ n=2 $ and $ t=5/3 $, which illustrates the generic situation for $ n+a+3>t $.}
\end{figure}
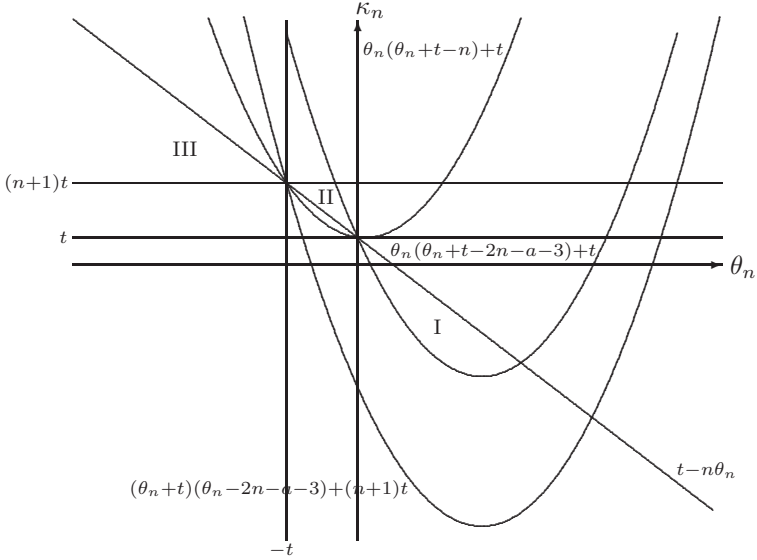

We see that Case II applies in our situation.
\begin{lemma}\label{dL_bound}
For all $ n $ and $ t \in \RR_+ $ we have
$ -t \leq \theta_n \leq 0 $ and $ t-n\theta_n \leq \kappa_n \leq (n+1)t $.
\end{lemma}
\begin{proof}
From the residue formula (\ref{ops_bL:c}) we recall that
$ \Theta_n(0)=\theta_n=2tp_n(0)\epsilon_n(0) $ and
$ \Theta_n(-t)=\theta_n+t=-atp_n(-t)\epsilon_n(-t) $. As we noted in the proof of 
Lemma (\ref{Lemma_Ineq}) it is immediate from the integral representations of the
polynomials and their associated functions, (\ref{ops_Pint}) and (\ref{ops_Eint}),
that $ (-1)^np_n(-t)\geq 0 $ and $ (-1)^{n+1}\epsilon_n(-t)\geq 0 $ for all real $ t \geq 0 $.
This places $ \theta_n $ in the range applying to Case II.
\end{proof}
 
We can also draw some conclusions concerning the zeros of the orthogonal polynomials
which will be important subsequently.
\begin{corollary}
Each zero $ x_{j,n}(t) $ is a monotonically decreasing function of $ t $ and 
interpolates between the Laguerre zero with $ t=0 $ and exponent $ a+2 $ and 
the Laguerre zero with $t=\infty $ and exponent $ 2 $,
\begin{equation}
   x_{j,n}(0) > x_{j,n}(t) > x_{j,n}(\infty) > 0 ,
\end{equation}
for all $ 1 \leq j \leq n $ and bounded $ t>0 $.
\end{corollary}
\begin{corollary}
The following bounds on the reciprocal sums over the zeros hold 
\begin{align}
   \sum^n_{j=1}\frac{1}{x_{j,n}(t)} & \leq \frac{n}{2},
\label{dL_BD:a} \\
   \sum^n_{j=1}\frac{1}{x_{j,n}(t)+t} & \leq \frac{n}{a+3}.
\label{dL_BD:b}
\end{align}
\end{corollary}
\begin{proof}
From the two-sided bound on $ \theta_n $ we can deduce
\begin{equation}
   \frac{1}{x_{j,n}(t)+t} \leq \frac{1}{x_{j,n}(t)-\theta_n} \leq \frac{1}{x_{j,n}(t)} .
\end{equation}
Employing this in the Bethe Ansatz (\ref{dL_BAE}) summed over $ j $ we arrive at the above bounds. 
\end{proof}

\section{Special Case $ a \in \ZZ_{\geq 0} $}
\setcounter{equation}{0}

Our evaluation of the distribution function in terms of the fifth Painlev\'e
system is with all three free parameters variable in some sense - one is fixed
in this application at a positive integer, one is the index $ n \in \ZZ $ and the
remaining one is $ a \in \CC $. Up to this point we have studied in some depth 
the recurrence relations with respect to $ n $ while $ a $ has been left arbitrary
other than being restricted because of the existence considerations.
From the point of view of the Painlev\'e theory it is quite natural that the
transcendental objects become classical when $ a \in \ZZ $ for either positive or
negative subsets of the integers. In particular it is expected that the $\tau$
functions in the theory will have Hankel 
determinantal forms of classical function entries with a rank dependent on $ a $.
It is these cases which have been studied in the
past \cite{FH_1994},\cite{Fo_1994a} using methods which transform the integral into
the determinantal representations and then employ confluent Vandermonde identities. 

\begin{proposition}
When $ a \in \ZZ_{> 0} $ we have the evaluation for the Hankel determinant
\begin{equation}                                                                           
   \Delta_n(t) = \frac{c_{n+1,n+1+a}}{(n+a)!}
              \det[ L^{(j+k+1-a)}_{n+a+1-j-k}(-t) ]_{j,k=1,\ldots,a} ,
\label{dL_Hankel_aInt}
\end{equation}
and for $ a=0 $
\begin{equation}                                                                           
   \Delta_n(t) = \frac{c_{n+1,n+1}}{n!} .
\label{}
\end{equation}
\end{proposition}
\begin{proof}
In \cite{FH_1994} Eq. (3.18) (after correcting) states
\begin{equation}
  \Delta_n(t) = \frac{c_{n+1,n+1+a}}{(n+a)!}(-1)^{a(a-1)/2}
              \det[ D^{j+k-2}_{x}L^{-(a-3)}_{n+a-1}(x)|_{x=-t} ]_{j,k=1,\ldots,a} ,
\label{dL_Hankel_aInt1}
\end{equation} 
where $ D_x := d/dx $. Using the Laguerre polynomial identity
\begin{equation}
  D^{m}_x L^{(\alpha)}_n(x) = (-1)^mL^{(\alpha+m)}_{n-m}(x) ,\quad m \in \ZZ_{\geq 0} ,
\label{Lag_Id1}
\end{equation}
with the proviso $ L^{(\alpha)}_n(x)=0 $ for $ n<0 $, we arrive at (\ref{dL_Hankel_aInt}).
\end{proof}
As a consequence of the relations (\ref{thetaDelta}) and (\ref{GammaDelta}) the 
variables $ \theta_n(t), \Gamma_n(t) $ will have $ a\times a $ determinant forms,
and in particular for $ a=0 $
\begin{equation} 
  \theta_n(t) = -t, \quad \kappa_n(t) = (n+1)t, \quad \Gamma_n(t) = -n(n+2) .
\end{equation} 

The orthogonal polynomials also have determinantal forms of the following type.
\begin{proposition}
When $ a \in \ZZ_{> 0} $ the orthogonal polynomials are given by
\begin{multline}
   \sqrt{\Delta_n\Delta_{n+1}}p_n(x;t) = 
  (-1)^{n+a+\lfloor\frac{a+1}{2}\rfloor} a!\ldots (n+a)! 1!\ldots (n+1)!
  \\ \times
  (x+t)^{-a} \det\left[ \begin{array}{c}
              \left[ L^{(k+1-a)}_{n+a+1-k}(x) \right]_{k=1,\ldots a+1} \\
              \left[ L^{(j+k-a)}_{n+a+2-j-k}(-t) \right]_{{\ScSt j=1,\ldots,a}\atop{\ScSt k=1,\ldots a+1}}
                        \end{array} \right] ,
\label{dL_poly_aInt}
\end{multline}
and for $ a=0 $
\begin{equation}
   p_{n}(x;t) = (-1)^n\left(\frac{1}{(n+2)(n+1)}\right)^{1/2}L^{(2)}_{n}(x) . 
\label{}
\end{equation}
If $ a > n+1 $ we note that $ L^{(a+1)}_{n+1-a}(-t)=0 $.
Consequently, under the same condition, the polynomial ratio is given by
\begin{equation}
   Q_n(x;t) = 
  \left(\frac{t}{x+t}\right)^{a}\frac
      { \det\left[ \begin{array}{c} 
              \left[ L^{(k+1-a)}_{n+a+1-k}(x) \right]_{k=1,\ldots a+1} \\ 
              \left[ L^{(j+k-a)}_{n+a+2-j-k}(-t) \right]_{{\ScSt j=1,\ldots,a}\atop{\ScSt k=1,\ldots a+1}}
                   \end{array} \right] }
      { \det\left[ \begin{array}{c} 
              \left[ {n+2}\choose{n+a+1-k} \right]_{k=1,\ldots a+1} \\ 
              \left[ L^{(j+k-a)}_{n+a+2-j-k}(-t) \right]_{{\ScSt j=1,\ldots,a}\atop{\ScSt k=1,\ldots a+1}}
                   \end{array} \right] } .
\label{dL_polyR_aInt}
\end{equation}
\end{proposition}
\begin{proof}
Starting with the integral representation (\ref{ops_Pint}) we follow the procedures 
used in \cite{FH_1994}. Taking one factor of the squared product of differences we
write it like
\begin{equation}
   \prod_{1\leq j<k\leq n}(x_k-x_j)
   = \frac{1}{\prod^{n-1}_{l=0}c_l}\det[ L^{(\alpha)}_{k-1}(x_j) ]_{j,k=1,\ldots,n} ,
\label{dL_aInt1}
\end{equation}
using the Vandermonde identity and 
where $ c_n $ is the leading coefficient of $ L^{(\alpha)}_{n}(x) $ and $ \alpha $
is a parameter to be fixed later. Of the remaining factors in the integrand we write 
\begin{multline}
   \prod^{n}_{j=1}(x_j+t)^a(x-x_j) \prod_{1\leq j<k\leq n}(x_k-x_j)
   = \frac{(-1)^{a(n+1)}}{\prod^{a-1}_{l=0}l!\prod^{N-1}_{l=0}c_l} (x+t)^{-a} \\
     \times
     \det\left[ \begin{array}{c}
         \left[ L^{(\alpha)}_{k-1}(x_j) \right]_{{\ScSt j=1,\ldots,n}\atop{\ScSt k=1,\ldots,N}} \\
         \left[ L^{(\alpha)}_{k-1}(x) \right]_{k=1,\ldots,N} \\
         \left[ D^{j-1}_yL^{(\alpha)}_{k-1}(y)|_{y=-t} \right]_{{\ScSt j=1,\ldots,a}\atop{\ScSt k=1,\ldots,N}}
                \end{array} \right] ,
\label{dL_aInt2}
\end{multline}
where the confluent Vandermonde identity has been used and $ N=n+a+1 $.

Reassembling the integral with these two factors, then expanding the determinant in 
(\ref{dL_aInt1}) we multiply each of $ n $ factors into the determinant of 
(\ref{dL_aInt2}). Making use of the antisymmetry of the row ordering in the first 
$ n $ rows of the determinant we can perform the $ n $ integrals as long as we choose
$ \alpha=2 $. Then
\begin{multline}
   \sqrt{\Delta_n\Delta_{n+1}}p_n(x;t)
   = \frac{(-1)^{a(n+1)}}{\prod^{a-1}_{l=0}l!\prod^{n-1}_{l=0}c_l\prod^{N-1}_{l=0}c_l}
     \prod^{n-1}_{j=0}\frac{(j+2)!}{j!}(x+t)^{-a} \\
     \times
     \det\left[ \begin{array}{c}
         \left[ L^{(2)}_{n+k-1}(x) \right]_{k=1,\ldots,a+1} \\
         \left[ D^{j-1}_yL^{(2)}_{n+k-1}(y)|_{y=-t} \right]_{{\ScSt j=1,\ldots,a}\atop{\ScSt k=1,\ldots,a+1}}
                \end{array} \right] .
\end{multline}
Using the identities
\begin{equation}
   L^{(\alpha-1)}_n(x) = L^{(\alpha)}_n(x)-L^{(\alpha)}_{n-1}(x) ,
\label{Lag_Id2}
\end{equation}
along with elementary column operations, and then identity (\ref{Lag_Id1}) we are lead to
(\ref{dL_poly_aInt}).
The evaluation for the polynomial ratio (\ref{dL_polyR_aInt}) is a simple consequence 
of the first evaluation along with
\begin{equation}
   L^{(-a+1+k)}_{n+a+1-k}(0) = {{n+2}\choose{n+a+1-k}} .
\label{Lag_Id3}
\end{equation}
\end{proof}

\begin{proposition}
When $ a \in \ZZ_{\geq 0} $ the deformed Hankel determinant is given by
\begin{multline}
   D_n(x,x) = c_{n+2,n+2+a}(-1)^{\frac{1}{2}(a+1)(a+2)}(x+t)^{-2a}
  \\ \times
     \det\left[ \begin{array}{c}
              \left[ L^{(j+k-1-a)}_{n+a+3-j-k}(-t) \right]_{{\ScSt j=1,\ldots,a}\atop{\ScSt k=1,\ldots a+2}} \\
              \left[ L^{(j+k-1-a)}_{n+a+3-j-k}(x) \right]_{{\ScSt j=1,2}\atop{\ScSt k=1,\ldots a+2}}
     \end{array} \right] .
\label{dL_deformH_aInt}
\end{multline}
\end{proposition}
\begin{proof}
The quantity $ D_n(x,x) $ was essentially computed in \cite{FH_1994} in Eq. (3.20),
which can be recast as
\begin{multline}
   D_n(x,x) = c_{n+2,n+2+a}(-1)^{\frac{1}{2}(a+1)(a+2)}(x+t)^{-2a}
  \\ \times
     \det\left[ \begin{array}{c}
              \left[ D^{j+k-2}_u L^{-(a-1)}_{n+a+1}(u)|_{u=-t} \right]_{{\ScSt j=1,\ldots,a}\atop{\ScSt k=1,\ldots a+2}} \\
              \left[ D^{j+k-2}_u L^{-(a-1)}_{n+a+1}(u)|_{u=x} \right]_{{\ScSt j=1,2}\atop{\ScSt k=1,\ldots a+2}}
     \end{array} \right] .
\end{multline}
Then (\ref{dL_deformH_aInt}) follows by application of the identity (\ref{Lag_Id1}).
\end{proof}

As a consequence the distribution of the first eigenvalue spacing is 
\begin{multline}
   A_{n,a}(y) = (-1)^{\frac{1}{2}(a+1)(a+2)}y^2 e^y\int^{\infty}_{y}dt\;t^a(t-y)^{-a}e^{-(n+2)t}
  \\ \times
     \det\left[ \begin{array}{c}
              \left[ L^{(j+k-a-1)}_{n+a+3-j-k}(-t) \right]_{{\ScSt j=1,\ldots,a}\atop{\ScSt k=1,\ldots a+2}} \\
              \left[ L^{(j+k-a-1)}_{n+a+3-j-k}(-y) \right]_{{\ScSt j=1,2}\atop{\ScSt k=1,\ldots a+2}}
     \end{array} \right] ,
\label{dL_Adbn_aInt}
\end{multline}
for $ a=1,2,3,\ldots $ and 
\begin{equation}
   A_{n,0}(y) = \frac{1}{n+2}y^2e^{-(n+1)y}
     \left[ L^{(1)}_{n+1}(-y)L^{(3)}_{n-1}(-y) -\left( L^{(2)}_{n}(-y) \right)^2 \right] ,
\end{equation}
for $ a=0 $.

\section{Hard Edge Scaling}
\setcounter{equation}{0}
\subsection{General Case}

We define new scaled spectral variables $ s,z $ by
\begin{equation}
  t = \frac{s}{4n}, \qquad x = -\frac{z}{4n} ,
\label{dL_HE_vars}
\end{equation}
in the triangular domain $ s > z > 0 $ and study the scaling of the finite 
distribution (\ref{LUE_2-1dbn.b}) as the polynomial degree $ n \to \infty $.
What is required here is not just the asymptotic scaling of the orthogonal polynomial
coefficients but also of the polynomials themselves in the neighbourhood of 
an endpoint of the interval of orthogonality. For the deformed Laguerre polynomials
this would be a generalisation of the asymptotics of Hilb's type for the
Laguerre polynomials as found in Szeg\"o \cite{ops_Sz} 
\begin{equation}
   e^{x/2}x^{\mu/2}L^{(\mu)}_n(-x) 
  = \frac{\Gamma(\mu+n+1)}{n!(\frac{1}{4}M)^{\mu/2}}I_{\mu}(\sqrt{Mx})
    + {\rm O}(n^{\frac{1}{2}\mu-\frac{3}{4}}) ,
  \label{Lag_Hilb}
\end{equation}
with $ M=4n+2\mu+2 $ as $ n \to \infty $ and $ I_{\mu}(z) $ the standard 
modified Bessel function. Despite a resurgence
in activity around these questions, especially the use of Riemann-Hilbert techniques
on these problems, there are no results available for our particular problem.
A general review of the asymptotics of orthogonal polynomials can be found in 
\cite{Lu_2000}, and an introduction to the Riemann-Hilbert approach to the 
asymptotics is the chapter in \cite{Ku_2003}. However to establish the nature of 
and the existence of limits for our variables we do not require such techniques.

\begin{lemma}
Under the scaling of the (\ref{dL_HE_vars}) $ 4n\theta_n(t) $ and $ \kappa_n(t) $
are bounded for all real positive $ t $ and $ n \geq 1 $.
\end{lemma}
\begin{proof}
From the result of Lemma \ref{dL_bound} we see that
\begin{equation}
  -s \leq 4n\theta_n(s/4n) \leq 0, \quad
   s/4n-4n\theta_n(s/4n) \leq \kappa_n(s/4n) \leq (n+1)s/4n ,
\end{equation}
and the assertion follows.
\end{proof}
\begin{corollary}
Under the above conditions
\begin{equation}
  n+\left.\frac{\kappa_n(t)-t}{\theta_n(t)}\right|_{t=s/4n} = {\rm O}(1), \quad
  \text{as $ n \to \infty $} .
\label{HE_limit1}
\end{equation}
\end{corollary}
\begin{proof}
The above lemma states that
\begin{equation}
  \theta_n(s/4n) = {\rm O}(1), \quad \text{as $ n \to \infty $},
\end{equation}
and we find as a consequence that also
\begin{equation}
  \varepsilon := n(n+a+2)-s/4+\Gamma_n(s/4n) = {\rm O}(1), 
  \quad \text{as $ n \to \infty $}.
\label{dL_GammaB}
\end{equation}
The discriminant $ {\mathcal D} $ appearing in the workings of 
Proposition \ref{dL_theta-GammaDE} can then be written as
\begin{equation}
  {\mathcal D}^2 := \theta^4_n-2(2n+a+2-t)\theta^3_n
  +[4(\varepsilon-t)+(a+2-t)^2-4nt]\theta^2_n
    +4t[\varepsilon+a+2-t]\theta_n + 4t^2,
\end{equation}
and therefore
\begin{equation}
  \left.\frac{{\mathcal D}(t)}{\theta_n(t)}\right|_{t=s/4n} = {\rm O}(1), \quad
  \text{as $ n \to \infty $} .
\end{equation}
From the formula for $ \kappa_n $ in Proposition \ref{dL_theta-GammaDE} we note 
that
\begin{equation}
 n+\frac{\kappa_n-t}{\theta_n} 
 = -\frac{a+2}{2}-\frac{t}{\theta_n}+\frac{1}{2}(\theta_n+t)
   -\frac{{\mathcal D}}{2\theta_n} , 
\end{equation}
and the result then follows.
\end{proof}

\begin{proposition}\label{HE_scale1}
For bounded $ s \in \RR_+ $ under the scaling (\ref{dL_HE_vars}) the variables $ \theta_n(t) $ 
and $ \Gamma_n(t) $ converge to limits in the following manner
\begin{align}
  \lim_{n \to \infty}4n\theta_n(t)|_{t=s/4n}
  & = \mu(s) ,
  \label{dL_HE_theta} \\
  \lim_{n \to \infty} n(n+a+2)+\Gamma_n(t)|_{t=s/4n}
  & = \nu(s) .
  \label{dL_HE_Gamma}
\end{align}
The variable $ \kappa_n(t) $ converges like
\begin{equation}
  \lim_{n \to \infty}\kappa_n(t)|_{t=s/4n} = -\frac{1}{4}\mu(s) .
  \label{dL_HE_kappa}
\end{equation}
\end{proposition}
\begin{proof}
Firstly we note that 
\begin{equation}
  \eta := \left. a^2_n(t)-n(n+a+2) \right|_{t=s/4n} = {\rm O}(1), \quad
  \text{as $ n \to \infty $} ,
\nonumber
\end{equation}
as follows from (\ref{dL_GammaB}). Starting with (\ref{dL_RR1}) we see that
\begin{align}
   \theta_{n}-\theta_{n-1} 
  & = 2\theta_{n}+t + \frac{(2n+a+2)\kappa_n-[n^2+(n+1)(a+2)]t}{a^2_n}, \nonumber\\
  & = \frac{1}{a^2_n}\left\{
    (2n+a+2+\eta)t+2(n+a+2)\theta_n\left(n+\frac{\kappa_n-t}{\theta_n}\right)\right\},
\label{HE_limit2}
\end{align} 
and thus 
\begin{equation}
   n\theta_{n}-(n-1)\theta_{n-1} = \theta_{n-1} 
  +  \frac{n}{a^2_n}\left\{
    (2n+a+2+\eta)t+2(n+a+2)\theta_n\left(n+\frac{\kappa_n-t}{\theta_n}\right)\right\}.
\nonumber
\end{equation}
Thus we have shown
\begin{equation}
  \left. n\theta_{n}(t)-(n-1)\theta_{n-1}(t)\right|_{t=s/4n} = {\rm O}(n^{-1}), \quad
  \text{as $ n \to \infty $} .
\nonumber
\end{equation}
Now we can write the quantity of interest 
\begin{multline}
  \left. n\theta_{n}(t)\right|_{t=s/4n}-\left. (n-1)\theta_{n-1}(t)\right|_{t=s/4(n-1)} \\
 = \left.\left[ n\theta_{n}(t)-(n-1)\theta_{n-1}(t)\right]\right|_{t=s/4n} \\
   + \left.(n-1)\theta_{n-1}(t)\right|_{t=s/4n}-\left.(n-1)\theta_{n-1}(t)\right|_{t=s/4(n-1)} ,
\nonumber
\end{multline}
so we require bounds on the difference of the last two terms on the right-hand side.
Let $ t=s/4n $ and $ t_{>}=s/4(n-1) $.
Because $ \theta_n(t) $ is continuously differentiable and its derivative is given by 
(\ref{PV:a})
\begin{multline}
  (n-1)\left|\theta_{n-1}(t)-\theta_{n-1}(t_{>})\right| \\
\begin{aligned}
  & \leq (n-1)(t_{>}-t)\max_{u \in (t,t_{>})}|\dot{\theta}_{n-1}(u)|, \nonumber\\
  & \leq \frac{s}{4n} \max_{u \in (t,t_{>})}|\dot{\theta}_{n-1}(u)|, \nonumber\\
  & \leq \frac{s}{4n} \max_{u \in (t,t_{>})}u^{-1}
         |2\kappa_{n-1}(u)+\theta_{n-1}(u)(2n+a+1-u-\theta_{n-1}(u))|,
      \nonumber\\
  & \leq \frac{s}{4n} \max_{u \in (t,t_{>})}u^{-1}|\theta_{n-1}(u)|
    \nonumber\\
  & \qquad \times
      \left|2\left(n-1+\frac{\kappa_{n-1}(u)-u}{\theta_{n-1}(u)}\right)
            +a+3-u-\theta_{n-1}(u)+\frac{2u}{\theta_{n-1}(u)}\right|, \nonumber\\
  & \leq t_{>}\max_{u \in (t,t_{>})}
      \left( 2\left|n-1+\frac{\kappa_{n-1}(u)-u}{\theta_{n-1}(u)}\right|
            +a+3+u+|\theta_{n-1}(u)|+\frac{2u}{|\theta_{n-1}(u)|}\right)
  \nonumber\\
  & = {\rm O}(n^{-1}), \quad \text{as $ n \to \infty $}.
\nonumber
\end{aligned}
\nonumber
\end{multline}
Thus the limit shown in (\ref{dL_HE_theta}) exists. Turning our attention to
(\ref{dL_HE_kappa}) we can use (\ref{dL_diffEQN:a}) to compute that
\begin{align}
  \kappa_{n+1}-\kappa_n 
  & = -2\kappa_n-b_n\theta_n, \nonumber\\
  & = -2\kappa_n-(2n+a+3-t-\theta_n)\theta_n, \nonumber\\
  & = -2\theta_n\left( n+\frac{\kappa_n-t}{\theta_n} \right) 
      -2t-(a+3)\theta_n+\theta_n(\theta_n+t),
  \nonumber
\end{align}
and therefore
\begin{equation}
  \left.\left[ \kappa_{n+1}(t)-\kappa_n(t)\right]\right|_{t=s/4n} = {\rm O}(n^{-1}), \quad
  \text{as $ n \to \infty $} .
  \nonumber
\end{equation}
Now in this case the quantity we require is
\begin{multline}
  \left. \kappa_{n+1}(t)\right|_{t=s/4(n+1)}-\left. \kappa_{n}(t)\right|_{t=s/4n} \\
 = \left.\kappa_{n+1}(t)\right|_{t=s/4(n+1)}-\left.\kappa_{n+1}(t)\right|_{t=s/4n}
   + \left.\left[ \kappa_{n+1}(t)-\kappa_{n}(t)\right]\right|_{t=s/4n} ,
\nonumber
\end{multline}
and therefore we need to bound the difference of first two terms on the right-hand side.
Let us denote $ t_{<}=s/4(n+1) $.
Again because $ \kappa_n(t) $ is continuously differentiable with derivative 
(\ref{dL_PV2}) we have
\begin{multline}
  \left|\kappa_{n+1}(t_{<})-\kappa_{n+1}(t)\right| \\
\begin{aligned}
  & \leq (t-t_{<})\max_{u \in (t_{<},t)}|\dot{\kappa}_{n+1}(u)|, \nonumber\\
  & \leq \frac{s}{4n(n+1)} \max_{u \in (t_{<},t)}|\dot{\kappa}_{n+1}(u)|, \nonumber\\
  & \leq \frac{s}{4n(n+1)} \max_{u \in (t_{<},t)}u^{-1}
         |\kappa_{n+1}(u)-a^2_{n+1}(u)(\theta_{n+1}(u)-\theta_{n}(u))|,
      \nonumber\\
  & \leq \frac{1}{n}\max_{u \in (t_{<},t)}
      \Big( |\kappa_{n+1}(u)|+(2n+a+4+|\eta(u)|)u
    \nonumber\\
  & \qquad
    +2(n+a+3)|\theta_{n+1}(u)|\left|n+1+\frac{\kappa_{n+1}(u)-u}{\theta_{n+1}(u)}\right| \Big)
    \nonumber\\
  & = {\rm O}(n^{-1}), \quad \text{as $ n \to \infty $} .
\nonumber
\end{aligned}
\nonumber
\end{multline}
In the last two steps we have used (\ref{HE_limit2}) and the subsequent estimates.
Thus the limit in (\ref{dL_HE_kappa}) follows. 
The fact that the limit of $ \kappa_n(t) $ under the hard edge scaling 
is related to the limit given in (\ref{dL_HE_theta}) follows from the relation (\ref{dL_diffEQN:b}).
The limit (\ref{dL_HE_Gamma}) is a consequence of the limits in the primary variables.
\end{proof}


In addition the following combinations of variables possess scaling limits which
will subsequently be useful.
\begin{corollary}\label{HE_scale2}
For bounded $ s\in \RR_+ $ the following limits as $ n \to \infty $ exist
\begin{align}
  \left.\frac{2\kappa_n(t)+\theta_n(t)b_n(t)}{\theta_n(t)}\right|_{t=s/4n}
  & \mathop{\to}\limits_{n \to \infty} s\frac{\dot{\mu}}{\mu} ,
  \label{dL_HE_comb1} \\
  \left.\left( n+\frac{\kappa_n(t)-t}{\theta_n(t)} \right)\right|_{t=s/4n}
  & \mathop{\to}\limits_{n \to \infty} C(s) ,
  \label{dL_HE_comb2} \\
  \left.\left( n+a+2+\frac{\kappa_n(t)+t}{\theta_{n-1}(t)} \right)\right|_{t=s/4n}
  & \mathop{\to}\limits_{n \to \infty} -C(s) ,
  \label{dL_HE_comb3} \\
  \left.n\left( n+\frac{\kappa_n(t)-t}{\theta_n(t)}
         +n+a+2+\frac{\kappa_n(t)+t}{\theta_{n-1}(t)} \right)\right|_{t=s/4n}
  & \mathop{\to}\limits_{n \to \infty} \xi(s) .
  \label{dL_HE_comb4}
\end{align}
\end{corollary}
\begin{proof}
The limit in (\ref{dL_HE_comb2}) is a consequence of (\ref{HE_limit1}) in a 
previous corollary.
The scaling limit of (\ref{dL_HE_comb3}) follows from that of (\ref{dL_HE_comb2})
and the identity (\ref{dL_RR6}).
The limit (\ref{dL_HE_comb4}) can be derived from (\ref{dL_RR7}) and as a result 
one can deduce that
\begin{equation}
  n+\frac{\kappa_n-t}{\theta_n}+n+a+2+\frac{\kappa_n+t}{\theta_{n-1}}
\nonumber
\end{equation}
is of order $ {\rm O}(n^{-1}) $.
\end{proof}

We note other relations amongst the scaling limit functions, namely
\begin{multline}
   2\mu C(s) = -[(a+2)\mu+2s] \\
   - \left\{ [(a+2)\mu+2s]^2+4\mu(\mu+s)\nu-\mu(\mu+s)^2 \right\}^{1/2} ,
\label{dL_HE_Cdefn}
\end{multline}
and
\begin{equation}
   \xi(s) = -\frac{sC(C+a)}{\mu+s} ,
\label{dL_HE_xidefn}
\end{equation}
and
\begin{equation}
   2C+a+3 = s\frac{\dot{\mu}-2}{\mu} .
\label{HE_Crel}
\end{equation}

\begin{proposition}\label{HE_sigmaFn}
The scaled variables $ \mu(s), \nu(s) $ are characterised by solutions to the
\PIIIprime system with parameters $ v_1=a+2, v_2=a-2 $. In particular 
\begin{equation}
  \nu(s) = -\sigma_{\rm III'}(s)+\frac{1}{4}s-a-2 ,
\label{dL_HE_nu_v}
\end{equation}
where $ \sigma_{\rm III'}(s) $ satisfies the Jimbo-Miwa-Okamoto $ \sigma$-form for 
\PIIIprime with above parameters.
The boundary conditions to uniquely specify the solution $ \nu(s) $ are
\begin{multline}
  \nu(s) \mathop{=}\limits_{s \to 0}
  \frac{a}{a+2}\frac{s}{4} 
  - \frac{a}{8(a+3)(a+2)^2(a+1)}s^2 + {\rm O}(s^3) \\
  - \frac{2}{4^{a+3}(a+2)(a+1)\Gamma(a+3)\Gamma(a+4)}s^{a+3}
    \left(1+{\rm O}(s)\right) + {\rm O}(s^{2a+6}) ,
\label{dL_HE_BCnu_v}
\end{multline}
assuming $ a \notin \ZZ_{\geq 0} $ and $ |{\rm arg}(s)| < \pi $.
\end{proposition}
\begin{proof}
Introducing the scaling ansatzes (\ref{dL_HE_theta}) and (\ref{dL_HE_Gamma}) 
formally into the differential equation (\ref{dL_thetaDE}) under the scaling
(\ref{dL_HE_vars}) we find that the highest order nontrivial relation
\begin{equation}
  \mu(s)+s = 4s\dot{\nu}(s) ,
\label{dL_HE_DEnu}
\end{equation}
at order $ n^{-1} $.
Proceeding in the same manner with the differential equation (\ref{dL_GammaDE})
we find the highest order relation is
\begin{equation}
  s\dot{\mu}(s) = \mu + \left\{ [(a+2)\mu+2s]^2+4\mu(\mu+s)\nu-\mu(\mu+s)^2 \right\}^{1/2} ,
\label{dL_HE_DEmu}
\end{equation}
which occurs at order $  n^{-1} $ as well. Eliminating $ \mu(s) $ using (\ref{dL_HE_DEnu})
we find that (\ref{dL_HE_DEmu}) yields
\begin{equation}
   s^2(\ddot{\nu})^2-(a+2)^2(\dot{\nu})^2
  +\dot{\nu}(4\dot{\nu}-1)(s\dot{\nu}-\nu)+\frac{1}{2}a(a+2)\dot{\nu}
  -\frac{1}{16}a^2 = 0 ,
\label{dL_HE_sigma}
\end{equation}
which is almost the Jimbo-Miwa-Okamoto $ \sigma$-form for \PIIIprime \cite{OK_1987c}.
The boundary conditions follow from the application of the scaling limit (\ref{dL_HE_Gamma}) 
to the expansion about $ t=0 $, Equation (\ref{dL_BV:c}). 
\end{proof}

In addition we find the scaling behaviour of the Hankel determinants and 
polynomial evaluations to be given by the following propositions.
\begin{proposition}
As $ n \to \infty $ under the hard edge scaling the Hankel determinants scale as
\begin{equation}
  \Delta_n(t)|_{t=s/4n} \mathop{\sim}\limits_{n \to \infty}
  1!\ldots (n-1)! \Gamma(a+3) \ldots \Gamma(n+a+2)
  \exp\left( \int^s_0\frac{du}{u}\nu(u) \right) ,
\label{dL_HE_Delta}
\end{equation}
and the monic polynomials evaluated at $ x=0 $ scale as
\begin{equation}
  \pi_n(0;t)|_{t=s/4n} \mathop{\sim}\limits_{n \to \infty}
  (-1)^n(a+3)_n\exp\left( \int^s_0\frac{du}{u}C(u) \right) .
\label{dL_HE_pi}
\end{equation}
\end{proposition}
\begin{proof}
The relation (\ref{dL_HE_Delta}) arises from integrating (\ref{GammaDelta}) with
respect to $ t $ and then employing the scaling form for the integrand as given by
(\ref{dL_HE_Gamma}). The second relation is derived by integrating (\ref{dL_PDderiv})
and using (\ref{dL_HE_comb2}) for the scaling of the resulting integrand. The 
factorial and Pochhammer prefactors arise from the normalisations at $ t=0 $.
\end{proof}

We find that the polynomial ratios have well defined scaling behaviour rather than
the polynomial themselves.
\begin{proposition}\label{HE_QRscale}
The polynomial ratios $ Q_{n}(x;t), R_{n}(x;t) $ scale as 
\begin{align}
    \lim_{n \to \infty}Q_n(x;t)|_{x=-z/4n,t=s/4n}  & = q(z;s) ,
\label{dL_HE_Q} \\
    \lim_{n \to \infty}nR_n(x;t)|_{x=-z/4n,t=s/4n} & = p(z;s) ,
\label{dL_HE_R}
\end{align}
where $ q(z;s), p(z;s) $ are entire functions of $ z $.
\end{proposition}
\begin{proof}
We start with the product form of the scaled polynomial
\begin{equation}
  Q_n(-\frac{z}{4n};\frac{s}{4n})
  = \prod^{n}_{j=1}\left( 1+\frac{z}{4nx_{j,n}(s/4n)} \right) ,
\label{}
\end{equation}
and seek bounds for the logarithm of the ratio of the $n$-th scaled polynomial
to the $(n-1)$st. When $ 0< s,z < \infty $ we can write the general bound as the
sum of four contributions
\begin{multline}
  \left| \log\frac{Q_n(-z/4n;s/4n)}{Q_{n-1}(-z/4(n-1);s/4(n-1))} \right| \\
  \leq \left|\log\left(1+\frac{z}{4nx_{n,n}(t)}\right)\right| \\
  + \sum^{n-1}_{j=1} \left\{
    \left|\log\left(1+\frac{z}{4nx_{j,n}(t)}\right)-\log\left(1+\frac{z}{4(n-1)x_{j,n}(t)}\right)\right|
  \right. \\
  + \left|\log\left(1+\frac{z}{4(n-1)x_{j,n}(t)}\right)-\log\left(1+\frac{z}{4(n-1)x_{j,n}(t_{>})}\right)\right|
          \\
  \left.
  + \left|\log\left(1+\frac{z}{4(n-1)x_{j,n}(t_{>})}\right)-\log\left(1+\frac{z}{4(n-1)x_{j,n-1}(t_{>})}\right)\right| 
  \right\} ,
\label{Qinc_LHS}
\end{multline}
where $ t=s/4n $ and $ t_{>}=s/4(n-1) $. Using the inequality 
\begin{equation}
   \log\frac{1+A}{1+B} < A-B
\nonumber
\end{equation}
for $ A>B>0 $ we can find a simpler bound
\begin{multline}
  \text{LHS of (\ref{Qinc_LHS})} \leq \frac{z}{4nx_{n,n}(t)} \\
  + \sum^{n-1}_{j=1} \left\{
    \left|\frac{z}{4nx_{j,n}(t)}-\frac{z}{4(n-1)x_{j,n}(t)}\right|
  \right. \\
  + \left|\frac{z}{4(n-1)x_{j,n}(t)}-\frac{z}{4(n-1)x_{j,n}(t_{>})}\right|
          \\
  \left.
  + \left|\frac{z}{4(n-1)x_{j,n}(t_{>})}-\frac{z}{4(n-1)x_{j,n-1}(t_{>})}\right| 
  \right\} .
\label{Qinc_bound}
\end{multline}
Considering the first sum of (\ref{Qinc_bound}) we see that this is
\begin{align}
  \sum^{n-1}_{j=1}\left|\frac{z}{4nx_{j,n}(t)}-\frac{z}{4(n-1)x_{j,n}(t)}\right|
  & = \frac{z}{4n(n-1)}\sum^{n-1}_{j=1}\frac{1}{x_{j,n}(t)}, \nonumber\\
  & < \frac{z}{4n(n-1)}\sum^{n}_{j=1}\frac{1}{x_{j,n}(t)}, \nonumber\\ 
  & < \frac{z}{4n(n-1)}\frac{n}{2}
    = z{\rm O}(n^{-1}), \quad \text{as $ n \to \infty $}, \nonumber
\end{align}
for all bounded $ s $.
The second term of (\ref{Qinc_bound}) is
\begin{equation}
   \sum^{n-1}_{j=1}\left|\frac{z}{4(n-1)x_{j,n}(t)}-\frac{z}{4(n-1)x_{j,n}(t_{>})}\right|
  = \frac{z}{4(n-1)}\sum^{n-1}_{j=1}\left| 
    \frac{1}{x_{j,n}(t)}-\frac{1}{x_{j,n}(t_{>})} \right| .
\nonumber
\end{equation}
The summand appearing here can be bounded in the following way
\begin{align}
  \left|\frac{1}{x_{j,n}(t)}-\frac{1}{x_{j,n}(t_{>})}\right|
  & = \frac{|x_{j,n}(t_{>})-x_{j,n}(t)|}{x_{j,n}(t_{>})x_{j,n}(t)}, \nonumber\\
  & \leq \frac{1}{x_{j,n}(t_{>})x_{j,n}(t)}(t_{>}-t)\max_{u\in(t,t_{>})}|\dot{x}_{j,n}(u)| .
\label{Qinc_2nd}
\end{align}
Now from (\ref{dL_ZeroDE}) we note that
\begin{equation}
  |\dot{x}_{j,n}(u)| = u^{-1}\frac{\theta_n(u)+u}{-\theta_n(u)+x_{j,n}(u)}x_{j,n}(u) ,
\nonumber
\end{equation}
and so furnishes a bound on (\ref{Qinc_2nd})
\begin{align}
  \left|\frac{1}{x_{j,n}(t)}-\frac{1}{x_{j,n}(t_{>})}\right|
  & \leq \frac{1}{x_{j,n}(t_{>})x_{j,n}(t)}\frac{s}{4n(n-1)}\frac{4n}{s}
      \max_{u\in(t,t_{>})}\frac{\theta_n(u)+u}{-\theta_n(u)+x_{j,n}(u)}x_{j,n}(u),
  \nonumber\\
  & \leq \frac{1}{n-1}\frac{1}{x_{j,n}(t_{>})x_{j,n}(t)}x_{j,n}(t)
         \max_{u\in(t,t_{>})}\frac{\theta_n(u)+u}{-\theta_n(u)+x_{j,n}(u)},
  \nonumber\\
  & \leq \frac{1}{n-1}\frac{1}{x_{j,n}(t_{>})}
         \max_{u\in(t,t_{>})}\frac{\theta_n(u)+u}{-\theta_n(u)+x_{j,n}(u)} .
  \nonumber
\end{align}
This means that 
\begin{align}
  \sum^{n-1}_{j=1}\left|\frac{1}{x_{j,n}(t)}-\frac{1}{x_{j,n}(t_{>})}\right|
  & \leq \frac{1}{n-1}\sum^{n-1}_{j=1}
         \max_{u\in(t,t_{>})}\frac{\theta_n(u)+u}{-\theta_n(u)+x_{j,n}(u)}
                      \frac{1}{x_{j,n}(t_{>})}, \nonumber\\
  & \leq \frac{1}{n-1}
      \sum^{n-1}_{j=1}\max_{u\in(t,t_{>})}\frac{\theta_n(u)+u}{-\theta_n(u)+x_{j,n}(u)}
      \sum^{n-1}_{j=1}\frac{1}{x_{j,n}(t_{>})}, \nonumber\\
  & < \frac{1}{n-1}
      \sum^{n}_{j=1}\max_{u\in(t,t_{>})}\frac{\theta_n(u)+u}{-\theta_n(u)+x_{j,n}(u)}
      \sum^{n}_{j=1}\frac{1}{x_{j,n}(t_{>})}, \nonumber\\
  & < \frac{n}{2(n-1)}
      \sum^{n}_{j=1}\max_{u\in(t,t_{>})}\frac{\theta_n(u)+u}{-\theta_n(u)+x_{j,n}(u)},
  \nonumber\\
  & < \frac{n}{2(n-1)}
      \max_{u\in(t,t_{>})}\sum^{n}_{j=1}\frac{\theta_n(u)+u}{-\theta_n(u)+x_{j,n}(u)},
  \nonumber\\
  & < \frac{n}{2(n-1)}\max_{u\in(t,t_{>})}
      \left|n+\frac{\kappa_n(u)-u}{\theta_n(u)}\right|,
  \nonumber
\end{align}
where we have used (\ref{ops_Zsum1}) in the last step.
The total contribution of the second term is therefore bounded by
\begin{equation}
   \frac{zn}{8(n-1)^2}{\rm O}(1) = z{\rm O}(n^{-1}) .
\nonumber
\end{equation}
The third sum in (\ref{Qinc_bound}) is
\begin{align}
  \frac{z}{4(n-1)}\sum^{n-1}_{j=1}\left|\frac{1}{x_{j,n}(t_{>})}-\frac{1}{x_{j,n-1}(t_{>})}\right|
  & = \frac{z}{4(n-1)}\sum^{n-1}_{j=1}
      \left(\frac{1}{x_{j,n}(t_{>})}-\frac{1}{x_{j,n-1}(t_{>})}\right),
  \nonumber\\
  & < \frac{z}{4(n-1)}\left(\sum^{n}_{j=1}\frac{1}{x_{j,n}(t_{>})}
                           -\sum^{n-1}_{j=1}\frac{1}{x_{j,n-1}(t_{>})}\right) .
  \nonumber
\end{align}
From the identity (\ref{dL_SumRecipZero}) and the scaling of the variables involved as given
in (\ref{dL_HE_comb4}) we conclude that the contribution of this sum is bounded by
\begin{equation}
   \frac{z}{4(n-1)}{\rm O}(1) = z{\rm O}(n^{-1}) .
\nonumber
\end{equation}
Finally we note that the isolated term in (\ref{Qinc_bound}) is of order 
$ z{\rm O}(n^{-2}) $ as a leading estimate of the largest zero $ x_{n,n} $ is
of order $ {\rm O}(n) $. This establishes that $ Q_n $ under the hard edge scaling
of the independent variables converges to a limit as $ n \to\infty $, for all 
real, positive and bounded $ z,s $.
\end{proof}

The spectral and deformation derivatives of the $ Q_n,R_n $ system scale to the
corresponding derivatives of the $ q,p $ system as in the following result.
\begin{proposition}\label{HE_qpDer}
Specify scaled quantities as in Proposition \ref{HE_scale1}, Corollary \ref{HE_scale2}
and Proposition \ref{HE_sigmaFn}.
The spectral derivatives of the $ q,p $ system are 
\begin{align}
  (s-z)z\partial_zq & = -zCq-(\mu+z)p ,
\label{HE_S:a} \\
  (s-z)z\partial_zp & = -z\left[ \xi+\frac{1}{4}(z-s) \right]q+[-2s+z(C+a+2)]p ,
\label{HE_S:b}
\end{align}
and their deformation derivatives are 
\begin{align}
  (s-z)s\partial_sq & =  zCq+(\mu+s)p ,
\label{HE_D:a} \\
  (s-z)s\partial_sp & = z\xi q-[s(2C+a)-zC]p ,
\label{HE_D:b}
\end{align}
The boundary conditions satisfied by the solutions $ q(z;s) $ and $ p(z;s) $ of the
above system on the domain $ s \geq z \geq 0 $ along $ z=0 $ are
\begin{align}
         q(0;s) & = 1 , 
  \label{HE_BC:e} \\
         p(0;s) & = 0 ,
  \label{HE_BC:f}
\end{align}
for all $ s > 0 $.
\end{proposition}
\begin{proof}
The first spectral derivative (\ref{HE_S:a}) follows from the scaling of (\ref{dL_S:a})
and employing (\ref{dL_HE_theta}), (\ref{dL_HE_kappa}) and (\ref{dL_HE_comb2}). The
second member (\ref{HE_S:b}) is derived from the scaling of  (\ref{dL_S:b}) and using
(\ref{dL_HE_comb2}) and (\ref{dL_HE_comb4}). The first deformation derivative (\ref{HE_D:a})
follows from the scaling of (\ref{dL_D:a}) and utilising (\ref{dL_HE_theta}), (\ref{dL_HE_kappa})
and (\ref{dL_HE_comb2}). The second deformation derivative (\ref{HE_D:b}) arises from the
scaling of (\ref{dL_D:b}), employing (\ref{dL_HE_comb3}) and (\ref{dL_HE_comb4}) and 
noting that
\begin{equation}
  \left. (\kappa_n+t)\frac{\theta_{n-1}+t}{\theta_{n-1}}-(\kappa_n-t)\frac{\theta_{n}+t}{\theta_{n}}
  \right|_{t=s/4n}
  \mathop{\sim}\limits_{n \to \infty} -\frac{s(2C+a)}{4n} .
\nonumber
\end{equation}
The boundary conditions (\ref{HE_BC:e}) and (\ref{HE_BC:f}) follow from the definitions
(\ref{dL_Qdefn}) and (\ref{dL_Rdefn}) respectively and the scalings in Proposition
\ref{HE_QRscale}.
\end{proof}

\begin{remark}
Both the spectral derivative (\ref{dL_S:c}) and the deformation derivative (\ref{dL_D:c})
scale to
\begin{equation}
  (\mu+s)z\partial_z q + (\mu+z)s\partial_s q + zC(s)q = 0 ,
\label{dL_HE_spec-deformD}
\end{equation}
and this mixed derivative equation can be easily found from Proposition \ref{HE_qpDer} 
by eliminating the variable $ p $ between (\ref{HE_S:a}) and (\ref{HE_D:a}), i.e.
$ (\mu+s) $ times (\ref{HE_S:a}) plus $ (\mu+z) $ times (\ref{HE_D:a}).
The three-term recurrence relation (\ref{dL_TTRR}) scales to
\begin{equation}
  z^2 \partial^2_z q + 2sz\partial_z\partial_s q + s^2\partial^2_s q
  +s\frac{\dot{\mu}-2}{\mu}\left[ z\partial_z q + s\partial_s q \right]
   -\frac{1}{4}zq = 0 .
\label{dL_HE_threeT}
\end{equation}
This can also be recovered from the equations of Proposition \ref{HE_qpDer}. If we
eliminate $ q $ between (\ref{HE_S:a}) and (\ref{HE_D:a}) by adding them we find that 
\begin{equation}
   p = z\partial_zq+s\partial_sq .
\label{}
\end{equation}
Employing this we find that (\ref{dL_HE_threeT}) is equivalent to 
\begin{equation}
   z\partial_zp+s\partial_sp+(-1+s\frac{\dot{\mu}-2}{\mu})p-\frac{1}{4}zq = 0 ,
\label{}
\end{equation}
which can be found by adding (\ref{HE_S:b}) and (\ref{HE_D:b}) and noting the relation
(\ref{HE_Crel}).
\end{remark}

\begin{remark}
In addition to the boundary conditions at $ z=0 $ given by (\ref{HE_BC:e}) and (\ref{HE_BC:f})
there are also relations along $ s=z $
\begin{equation}
    \frac{d}{ds}q(s;s) = -\frac{C}{\mu+s}q(s;s) ,
\label{dL_HE_qEqual}
\end{equation}
and further
\begin{equation}        
    \frac{p(s;s)}{q(s;s)} = \frac{\xi}{C+a} = -\frac{sC}{\mu+s} . 
\label{dL_HE_qpEqual}                               
\end{equation}
However these are a consequence of the spectral and deformation derivatives and so do not
constitute independent boundary conditions.
\end{remark}

\begin{remark}
The compatibility of the two sets of derivatives, (\ref{HE_S:a}) and (\ref{HE_S:b}) on the
one hand, and (\ref{HE_D:a}) and (\ref{HE_D:b}) on the other hand, affords a check on
our results. We find that compatibility of (\ref{HE_S:a}) and (\ref{HE_D:a}) leads us
to conclude that
\begin{equation}
    \xi = s\dot{C}+\frac{1}{4}(\mu+s) ,
\label{HE_compat:a}
\end{equation}
and we also recover (\ref{HE_Crel}). Similar considerations applied to (\ref{HE_S:b}) and 
(\ref{HE_D:b}) imply that
\begin{equation}
    s\dot{\xi} = -(2C+a+2)\xi+\frac{1}{4}s(2C+a) ,
\label{HE_compat:b}
\end{equation}
and we get (\ref{HE_compat:a}) again. Using (\ref{dL_HE_xidefn}) to eliminate $ \mu $
we arrive at a coupled pair of first order ordinary differential equations
\begin{align}
   s\dot{C}   & = \xi+\frac{s}{4\xi}C(C+a) , \\
   s\dot{\xi} & = -(2C+a+2)\xi+\frac{1}{4}s(2C+a) .
\end{align}
Using the latter equation to eliminate $ C $ we obtain a second order ordinary differential 
equation for $ \xi $, which by means of the transformation
\begin{equation}
   \xi(s) = \frac{1}{4}\frac{sy(s)}{y(s)-1} ,
\label{}
\end{equation}
is transformed into the standard equation for the fifth Painlev\'e transcendent. This is
a degenerate case of \PV which reduces to the third Painlev\'e transcendent because the
parameters are 
$ \alpha_{\rm V} = 9/2, \beta_{\rm V} = -a^2/2, \gamma_{\rm V} = 1/2, \delta_{\rm V} = 0 $.
Making an independent variable transformation $ s\mapsto 2s $ so that $ \gamma_V = 1 $
we determine the \PIII parameters to be 
$ \alpha_{\rm III} = 2(2-a), \beta_{\rm III} = 2(a+3), \gamma_{\rm III} = 1, \delta_{\rm III} = -1 $
which is consistent with those in Proposition \ref{HE_sigmaFn}.
\end{remark}

The matrix form of the spectral derivatives (\ref{HE_S:a},\ref{HE_S:b}) and deformation 
derivatives (\ref{HE_D:a},\ref{HE_D:b}) yield the Lax pairs
\begin{align}
  \partial_z \Psi &= 
  \left\{ {\mathcal A}_{\infty}+\frac{{\mathcal A}_{s}}{z-s}+\frac{{\mathcal A}_{0}}{z}\right\}
  \Psi ,
\label{HE_LP:a}\\
  \partial_s \Psi &= 
  \left\{ {\mathcal B}-\frac{{\mathcal A}_{s}}{z-s}\right\} \Psi ,
\label{HE_LP:b}
\end{align}
in the matrix variable
\begin{equation}
   \Psi(z;s) = \begin{pmatrix} q(z;s) \\ p(z;s) \end{pmatrix} .
\label{HE_Psi}
\end{equation}
The system has two regular singularities at $ z=0,s $ and an irregular one at $ z=\infty $
with a Poincar\'e index of 1. This system is essentially equivalent to the isomonodromic system
of the fifth Painlev\'e equation but is the degenerate case. The residue matrices are
\begin{align}
  {\mathcal A}_{0} &= 
  \begin{pmatrix} 0 & -\frac{\displaystyle\mu(s)}{\displaystyle s} \\ 0 & -2 \end{pmatrix} ,
\label{}\\
  {\mathcal A}_{s} &= 
  \begin{pmatrix} C(s) & \frac{\displaystyle\mu(s)+s}{\displaystyle s} \\
                \xi(s) & -C(s)-a \end{pmatrix} ,
\label{}\\
  {\mathcal A}_{\infty} &= 
  \begin{pmatrix} 0 & 0 \\ \frac{1}{4} & 0 \end{pmatrix} ,
\label{}\\
  {\mathcal B} &= 
  \frac{1}{s}\begin{pmatrix} -C(s) & 0 \\ -\xi(s) & -C(s) \end{pmatrix} .
\label{}
\end{align}
Local convergent expansions about the regular singularities take the form for $ z=0 $
\begin{align}
  q(z;s) & = \sum_{m\geq 0}r^0_{m}(s)z^{\chi_0+m} ,
\label{HE_z=0Exp:a}\\
  p(z;s) & = \sum_{m\geq 0}u^0_{m}(s)z^{\chi_0+m} ,
\label{HE_z=0Exp:b}
\end{align}
for $ |z| < s $ 
and the initial relations amongst the coefficients are found to be 
\begin{equation}
   s\chi_0r^0_{0} = -\mu u^0_{0}, \quad s\chi_0u^0_{0} = -2su^0_{0} .
\label{}
\end{equation}
This implies that $ \chi_0=-2 $ or $ \chi_0=0 $ and $ u^0_{0}=0 $. The latter
case applies here as both $ q,p $ are analytic at $ z=0 $, and in addition we
also require $ p=0 $ on $ z=0 $. In addition $ r^0_{0}=1 $. The recurrence relations 
for general $ m $ are
\begin{align}
  s(m+2)u^0_{m} 
  & = (C+a+m+1)u^0_{m-1}+\left(\frac{1}{4}s-\xi\right)r^0_{m-1}-\frac{1}{4}r^0_{m-2} ,
\label{}\\
  smr^0_{m} & = -\mu u^0_{m}+(m-1-C)r^0_{m-1}-u^0_{m-1} .
\label{}
\end{align}
For $ z=s $ we have the convergent expansion
\begin{align}
  q(z;s) & = \sum_{m\geq 0}r^s_{m}(s)(s-z)^{\chi_s+m} ,
\label{HE_z=sExp:a}\\
  p(z;s) & = \sum_{m\geq 0}u^s_{m}(s)(s-z)^{\chi_s+m} ,
\label{HE_z=sExp:b}
\end{align}
for $ |z-s| < s $ and where in this case the initial relations are 
\begin{equation}
   s(-\chi_s+C)r^s_{0} = -(\mu+s)u^s_{0}, \quad (C+a+\chi_s)u^s_{0} = \xi r^s_{0} .
\label{}
\end{equation}
Combining these we get a relation which is identical to the second equality in 
(\ref{dL_HE_qpEqual}) only if $ \chi_s=0, -a $. The former case is the one we must 
choose as $ q,p $ are well-defined and finite on $ z=s $. For general $ m $ the
recurrence relations are
\begin{align}
  s(C-m)r^s_{m}+(\mu+s)u^s_{m} & = (C-m+1)r^s_{m-1}+u^s_{m-1} ,
\label{}\\
  -\xi r^s_{m}+s(C+a+m)u^s_{m} 
  & = -\left(\frac{1}{4}s+\xi\right)r^s_{m-1}+(C+a+m+1)u^s_{m-1}+\frac{1}{4}r^s_{m-2} .
\label{}
\end{align}
This system has a unique
solution only if $ s^2m(a+m) \neq 0 $ for $ s > 0 $ and $ m \geq 1 $ which in
turn means that $ a \neq -\NN $.
We also note that the two sets of coefficients are related by
\begin{equation}
   r^s_{m}(s) = (-1)^m\sum^{\infty}_{n=0} {{m+n}\choose{m}} s^nr^0_{m+n}(s) ,
\label{}
\end{equation}
and an identical relation for $ u^s_{m}(s) $.

\begin{proposition}
The determinant $ D_n(x,x) $ scales as
\begin{equation}
  D_n(x,x)|_{x=-z/4n,t=s/4n} 
  \mathop{\sim}\limits_{n \to \infty} -4\Delta_n \pi_n(0)\pi_{n+1}(0)
  \left[ q\partial_z p - p\partial_z q \right] .
\label{dL_HE_D}
\end{equation}
\end{proposition}
\begin{proof}
We can employ the polynomials $ Q_n,R_n $ in the evaluation (\ref{ops_Uv2})
and find
\begin{equation}
  D_n(x,x) = \Delta_n\pi_n(0)\pi_{n+1}(0)[Q_{n+1}R'_{n+1}-R_{n+1}Q'_{n+1}] .
\nonumber
\end{equation}
Applying the scaling of Proposition \ref{HE_QRscale} to this expression we 
arrive at (\ref{dL_HE_D}). 
\end{proof}

\begin{proposition}
The distribution $ A_{n,a}(y) $ scales to
\begin{equation}
  \frac{1}{4n}A_{n,a}(y)|_{y=z/4n} \mathop{\sim}\limits_{n \to \infty} A_a(z) .
\label{dL_HE_scaleA}
\end{equation}
\end{proposition}
\begin{proof}
This is the only possible scaling consistent with the scaling of the independent 
variable to the hard edge, $ y=z/4n $. 
\end{proof}

\begin{proposition}
The distribution of the first eigenvalue spacing at the hard edge is given by
\begin{multline}
  A_a(z) = \frac{z^2}{4^{2a+3}\Gamma(a+1)\Gamma(a+2)\Gamma^2(a+3)}
  \\ \times
     \int^{\infty}_z ds\; s^a(s-z)^a
  \exp\left( -\frac{s}{4}+\int^s_0\frac{dv}{v}[\nu(v)+2C(v)] \right)
  \left[ q\partial_z p - p\partial_z q \right] .
\label{dL_HE_A}
\end{multline}
\end{proposition}
\begin{proof}
We apply the scaling (\ref{dL_HE_scaleA}) to Eq. (\ref{LUE_2-1dbn.b}) and utilise 
the previous relation for the scaling of the integrand (\ref{dL_HE_D}). For the first 
three factors on the right-hand side of (\ref{dL_HE_D}) we can use the scaling 
results of (\ref{dL_HE_Delta}) and (\ref{dL_HE_pi}), yielding the above integral
representation.
\end{proof}
 
We note that the factor of the integrand of (\ref{dL_HE_A}) can be written as
\begin{equation}
  q\partial_z p - p\partial_z q
  = \frac{1}{4}q^2-\frac{2}{z}qp+\frac{\mu}{sz}p^2
    -\frac{1}{\xi}(\xi q-Cp)\frac{\xi q-(C+a)p}{s-z} .
\end{equation}
The corresponding distribution $ A^{\pm}(z) $ defined in (\ref{HE_jpdf_pm}) and
(\ref{HE_A_pm}) for the special cases $ a=\pm 1/2 $ is given by
\begin{multline}
  A^{\pm}(z) = \frac{z^2}{4^{2a+2}\Gamma(a+1)\Gamma(a+2)\Gamma^2(a+3)}
  \\ \times
     \int^{\infty}_z ds\; s^{2a+1}(s-z)^{2a+1}(2s-z)^2
  \exp\left( -\frac{s^2}{4}+\int^{s^2}_0\frac{dv}{v}[\nu(v)+2C(v)] \right)
  \\ \times
  \left.\left[ q\partial_z p - p\partial_z q \right]
  \right|_{{\ScSt s\mapsto s^2}\atop{\ScSt z\mapsto z(2s-z)}} .
\label{dL_HE_A_pm}
\end{multline}

\subsection{Special Case $ a \in \ZZ_{\geq 0} $}

\begin{proposition}
In the special case $ a \in \ZZ_{\ge 0} $ we have
\begin{equation}
   \Delta_n(t)|_{t=s/4n} \mathop{\sim}\limits_{n \to \infty}
   \frac{c_{n+1,n+1+a}}{(n+a)!} (-1)^{\lfloor a/2\rfloor}(2n)^as^{-a}
   \det\left[ I_{j+2-k}(\sqrt{s}) \right]_{j,k=1,\ldots,a} ,
\label{dL_HE_aInt_Delta}
\end{equation}
and
\begin{equation}
  \nu(s) = s\frac{d}{ds}\log s^{-a}\det\left[ I_{j+2-k}(\sqrt{s}) \right]_{j,k=1,\ldots,a} ,
\label{dL_HE_aInt_nu}
\end{equation}
whilst for $ a=0 $ we find $ \nu(s)=0 $.
\end{proposition}
\begin{proof}
The first relation (\ref{dL_HE_aInt_Delta}) follows from an application of the Hilb type
asymptotic formula (\ref{Lag_Hilb})
to (\ref{dL_Hankel_aInt}) and the second follows by using this result in (\ref{GammaDelta}).
\end{proof}

Another consequence of the Hilb formula is the scaling of the orthogonal polynomial ratio
as given by (\ref{dL_polyR_aInt}).
\begin{proposition}
The scaled orthogonal polynomial ratio $ q(z;s) $ is
\begin{equation}
   q(z;s) = z^{-3/2} \left(\frac{s}{s-z}\right)^{a}\frac
      { \det\left[ \begin{array}{c} 
              \left[ (z/s)^{k/2}I_{3-k}(\sqrt{z}) \right]_{k=1,\ldots a+1} \\ 
              \left[ I_{j+2-k}(\sqrt{s}) \right]_{{\ScSt j=1,\ldots,a}\atop{\ScSt k=1,\ldots a+1}}
                   \end{array} \right] }
      { \det\left[ \begin{array}{c} 
              \left[ 1/8\sqrt{s},1/2s,1/s^{3/2},0, \ldots, 0 \right] \\ 
              \left[ I_{j+2-k}(\sqrt{s}) \right]_{{\ScSt j=1,\ldots,a}\atop{\ScSt k=1,\ldots a+1}}
                   \end{array} \right] }, \quad a \geq 1
\label{dL_HE_aInt_q}
\end{equation}
and for $ a=0 $ is
\begin{equation}
  q(z;s) = \frac{8}{z}I_2(\sqrt{z}) .
\end{equation}
\end{proposition}

Finally the eigenvalue spacing distribution (\ref{dL_Adbn_aInt}) takes the following form
in the scaling limit.
\begin{proposition}
The distribution of the first eigenvalue spacing at the hard edge for $ a \in \ZZ_{>0} $ is
\begin{multline}
   A_a(z) = 2^{-4}z^{1/2}
   \int^{\infty}_{z}\frac{ds}{s^{1/2}}\left(\frac{s}{s-z}\right)^{a}e^{-s/4}      
   \\ \times
              \det\left[ \begin{array}{c}
   \left[ I_{j+2-k}(\sqrt{s}) \right]_{{\ScSt j=1,\ldots,a}\atop{\ScSt k=1,\ldots a+2}} \\
   \left[ (s/z)^{(2-k)/2}I_{j+2-k}(\sqrt{z}) \right]_{{\ScSt j=1,2}\atop{\ScSt k=1,\ldots a+2}}
                         \end{array} \right] .
\label{dL_HE_aInt_A}
\end{multline}
for $ a \geq 1 $ and for $ a=0 $ is
\begin{equation}
  A_0(z) = \frac{1}{4}e^{-z/4}\left[ I^2_{2}(\sqrt{z})-I_{1}(\sqrt{z})I_{3}(\sqrt{z}) \right]  .
\end{equation}
\end{proposition}
\begin{proof}
We apply the Hilb asymptotic formula (\ref{Lag_Hilb}) to (\ref{dL_Adbn_aInt}).
As $ \frac{1}{2}(a+1)(a+2)+1+\lfloor a/2\rfloor $ is always even for $ a \in \ZZ $ we have
(\ref{dL_HE_aInt_A}).
\end{proof}

\section{Analytical Studies at the Hard Edge}
\setcounter{equation}{0}

In this section of our study we intend to develop the analytical and non-formal theory
of the solutions to the defining ordinary and partial differential equations 
described in the previous sections. This is because we wish to compute precision numerical 
data characterising the distribution function of the first eigenvalue spacing at the hard edge 
$ A_a(z) $ for arbitrary parameter $ a $. For this purpose it is not sufficient to employ 
a single local expansion of the $\sigma$-function, about $ s=0 $ say, because it has a finite 
convergence domain and one cannot use this to evaluate the $s$-integrals
on the interval $ [0,\infty) $. For this reason we construct a patchwork of overlapping local
expansions including Taylor series expansions about regular points $ s_0 $ for positive and real 
values. A similar approach was undertaken by Pr\"ahofer and Spohn in their study 
\cite{PS_2004} of the exact scaling functions for one-dimensional stationary KPZ growth.

\subsection{The $\sigma$-function expansion about $ s=0 $}
The terms given in the expansion of the Painlev\'e \IIId $\sigma$-function (\ref{dL_HE_BCnu_v}) 
are the minimum required to specify the full non-analytic Puiseux-type expansion of the 
particular solution for $ \nu(s) $ about $ s=0 $ in the sector $ -\pi\leq {\rm arg}(s) <\pi $.
This is the primary data specifying our particular solution and we need to use this
in various ways in order to compute the distribution $ A_a(z) $.

\begin{proposition}[\cite{Ji_1982},\cite{Ki_1989}]
The Painlev\'e \IIId $\sigma$-function $ \nu(s) $ has a Puiseux-type expansion about the
fixed regular singular point $ s=0 $ of the form
\begin{equation}
  \nu(s) = 
  \sum_{k=0}^{\infty} \sum_{j=0}^{\infty} c_{k,j}s^{j+ka} ,
  \quad a \in \CC, \quad 0\leq {\rm Re}(a) <1 ,
\label{HE_puiseux}
\end{equation}
with $ |{\rm arg}(s)| < \pi $
and is convergent in a finite domain $ s \in \{z\in\CC:|z|<R,|z^a|<R_a\} $.
\end{proposition}
The coefficients $ c_{k,j} $ are determined by recurrences which follow from the 
substitution of expansion (\ref{HE_puiseux}) into the relation (\ref{dL_HE_sigma}).
We will assume that $ a $ is not rational for simplicity.
Considering terms in the resulting equation with $ s^p $ and $ p\in \ZZ_{\geq 0} $ then
the $ p=0 $ case implies that if $ c_{0,0} = 0 $ then this choice fixes 
\begin{equation}
  c_{0,1} = \frac{a}{4(a+2)} .
\label{HE_C1}
\end{equation}
The $ p=1 $ is automatically zero but for the $ p=2 $ case one has the two options
\begin{equation}
  c_{0,2} = -\frac{a}{8(a+3)(a+2)^2(a+1)} , \quad{\rm or}\quad 0.
\label{HE_C2}
\end{equation}
The former case applies here by comparison with (\ref{dL_HE_BCnu_v}). 
For $ p \ge 3 $ the recurrence is
\begin{multline}
 - \frac{1}{2}\frac{a(p+a+1)(p-a-3)}{(a+3)(a+2)^2(a+1)}c_{0,p}
 - \frac{2}{a+2}\sum^{p-1}_{j=2}j(p-j)c_{0,j}c_{0,p+1-j} \\
 + 8c_{0,2}\sum^{p-1}_{j=2}(j-1)(p-j)c_{0,j-1}c_{0,p+1-j} \\
 + \sum^{p-1}_{j=3}j(p+2-j)[(j-1)(p+1-j)-(a+2)^2]c_{0,j}c_{0,p+2-j} \\
 +4\sum^{p-1}_{j=3}\sum^{p+2-j}_{m=1}jm(p+1-j-m)c_{0,j}c_{0,m}c_{0,p+2-j-m} = 0,
\label{PIII_cRecur:a}
\end{multline}
which allows for $ c_{0,p} $ to be recursively found, valid for $ a \notin \ZZ $.
Terms with $ s^{-2+qa} $ for $ q=2k \geq 2 $ imply that $ c_{k,0}=0 $ given that
$ c_{0,0}=0 $ and $ a \neq 3/(k-1), -1/(k+1) $. Consequently all terms with
$ s^{-1+qa} $ vanish also.
The generic recurrence relation following from examining the $ s^{p+qa} $ term is
\begin{multline}
 \frac{1}{2}a(a+2)(p+1+qa)c_{q,p+1}
 = \sum^q_{k=0}\sum^{p+1}_{j=0}(j+ka)[p-j+(q-k)a]c_{k,j}c_{q-k,p+1-j} \\
 - \sum^q_{k=0}\sum^{p+2}_{j=0}(j+ka)[p+2-j+(q-k)a]
                               [(j-1+ka)[p+1-j+(q-k)a]-(a+2)^2]c_{k,j}c_{q-k,p+2-j} \\
 -4\sum^q_{k=0}\sum^{q-k}_{l=0}\sum^{p+2}_{j=0}\sum^{p+2-j}_{m=0}
   (j+ka)(m+la)[p+1-j-m+(q-k-l)a]c_{k,j}c_{l,m}c_{q-k-l,p+2-j-m} ,
\label{PIII_cRecur:b}
\end{multline}
for $ q,p \ge 0 $. The convention is taken that sums with upper limits less than their
lower limits are zero, or equivalently coefficients with negative indices are zero.
Contrary to appearances the highest coefficient in (\ref{PIII_cRecur:b}) will turn out to 
be $ c_{q,p} $ as $ c_{q,p+2} $ occurs with $ c_{0,0} $ as a factor and $ c_{q,p+1} $
has a factor of
\begin{equation}
  \frac{1}{2}a(a+2)+c_{0,0}-2(a+2)^2c_{0,1}-8c_{0,0}c_{0,1} ,
\end{equation}
which is zero as a consequence of $ c_{0,0}=0 $ and (\ref{HE_C1}).
We find that the coefficient $ c_{q,p} $ occurs as a linear term and has a factor of
\begin{equation}
   (p-1+qa)c_{0,1}(4c_{0,1}-1)+4(p+qa)(p-1+qa-(a+2)^2)c_{0,2}-16(p+qa)c_{0,0}c_{0,2} .
\end{equation}
With the above evaluations (\ref{HE_C1}) and (\ref{HE_C2}) this is
\begin{equation}
   -\frac{a(p+1+(q+1)a)(p-3+(q-1)a)}{2(a+3)(a+2)^2(a+1)} .
\end{equation}
This vanishes when $ q=1,p=3 $ and we find that $ c_{1,3} $ is undetermined and therefore
is a free parameter. 
With these choices a triangular subset of the coefficients 
vanish
\begin{equation}
  c_{k,j} = 0, \quad \text{$ j=0,\ldots,3k-1 $ for $ k=1, \ldots, \infty $} ,
\label{}
\end{equation}
so that the initial non-zero term in the $j-$sum is $ c_{k,3k} $. All other terms 
are fixed by $ c_{1,3} $ and $ a $.

For the variable $ \mu(s) $ we can deduce the following Puiseux-type expansion from the
above work 
\begin{multline} 
    \mu(s) \mathop{\sim}\limits_{s \to 0}
   -\frac{2}{a+2}s -\frac{a}{(a+3)(a+2)^2(a+1)}s^2 \\
   - \frac{2}{4^{a+2}(a+2)(a+1)\Gamma^2(a+3)}s^{a+3} .
\end{multline} 
In regard to the quantity $ 2C(s) $ we require as much detail about this as for $ \nu(s) $.
This variable is a $\sigma$-function for an identical problem, where the fixed exponent
$ 2 $ in the weight (\ref{deform_Lwgt}) is replaced by $ 3 $. If we define the expansion 
coefficients for this object
\begin{equation}
  2C(s) = 
  \sum_{k=0}^{\infty} \sum_{j=0}^{\infty} a_{k,j}s^{j+ka} ,
  \quad a \in \CC, \quad 0\leq {\rm Re}(a) <1 ,
\label{}
\end{equation}
then the defining recurrences for these using (\ref{HE_Crel}) are
\begin{equation} 
  a_{0,0} = 0, \quad
  a_{0,1} = 4(a+2)(a+1)c_{0,2} ,
\end{equation}
and for $ j\geq 2 $
\begin{equation} 
   a_{0,j} = 2(a+2)\Big[-(j+1)(j-2-a)c_{0,j+1}+\sum^{j}_{l=2}la_{0,j+1-l}c_{0,l} \Big] .
\end{equation}
For the case $ k \geq 1 $ we have $ a_{k,j}=0 $ for $ j=0,..,3k-1 $ and the remaining 
non-zero terms are given by 
\begin{multline} 
   a_{1,j} = 2(a+2)\Big[-(j+1+a)(j-2)c_{1,j+1} \\
   +\sum^{j}_{l=0}(l+a)a_{0,j+1-l}c_{1,l}
    +\sum^{j+1}_{l=2}la_{1,j+1-l}c_{0,l} \Big] ,
\end{multline}
for $ k=1 $ and the general case $ k\geq 2 $ by
\begin{multline} 
   a_{k,j} = 2(a+2)\Big[-(j+1+ka)(j-2+(k-1)a)c_{k,j+1} \\
   +\sum^{j}_{l=0}(l+ka)a_{0,j+1-l}c_{k,l}
    +\sum^{j+1}_{l=2}la_{k,j+1-l}c_{0,l} \\
     +\sum^{k-1}_{m=1}\sum^{j+1}_{l=0}(l+ma)a_{k-m,j+1-l}c_{m,l} \Big] .
\end{multline}
The first few terms are
\begin{multline} 
    2C(s) \mathop{\sim}\limits_{s \to 0} -\frac{a}{2(a+3)(a+2)}s
  -\frac{a(a^2-5a-18)}{8(a+4)(a+3)^2(a+2)^2(a+1)}s^2 \\
  +\frac{1}{4^{a+2}(a+2)(a+1)\Gamma(a+3)\Gamma(a+4)}s^{a+3} .
\end{multline}
and
\begin{multline} 
    \xi(s) \mathop{\sim}\limits_{s \to 0}\frac{a}{4(a+3)}s
  -\frac{3a}{8(a+4)(a+3)^2(a+2)}s^2 \\
  +\frac{3}{4^{a+7}(a+3)(a+2)(a+1)\Gamma^2(a+4)}s^{a+4} .
\end{multline}

\subsection{The $ q, p $ expansion about $ s=0 $}
In this subsection we seek local expansions about $ s=0 $ for the coefficients 
functions $ r^0_{m}(s),u^0_{m}(s),r^s_{m}(s),u^s_{m}(s) $ appearing in 
(\ref{HE_z=0Exp:a}), (\ref{HE_z=0Exp:b}), (\ref{HE_z=sExp:a}) and (\ref{HE_z=sExp:b}). 
In parallel with the transcendent quantities these will have Puisuex-type expansions
of the form
\begin{align}
  r^0_{m}(s) & = \sum_{k,j \geq 0} r^0_{m,k,j}s^{j+ka} ,
  \label{HE_z=0Exp_s=0:a} \\
  r^s_{m}(s) & = \sum_{k,j \geq 0} r^s_{m,k,j}s^{j+ka} ,
  \label{HE_z=0Exp_s=0:b}
\end{align}
with analogous expansions for the remaining two coefficients. This is immediately 
clear because the recurrence relations for these imply that they are polynomial
functions of the variables $ \mu(s), C(s), \xi(s) $. These recurrence relations
imply the following ones for the $ z=0 $ coefficients $ r^0_{m,k,j}, u^0_{m,k,j} $
\begin{equation}
   (a+m+2)u^0_{m,k,0} = \frac{1}{4}r^0_{m-1,k,0}, \quad
   mr^0_{m,k,0} = u^0_{m,k,0},
\end{equation}
for $ m \geq 1, k \geq 0 $. For the general case $ j \geq 1, k \geq 0 $ we have
\begin{multline}
  (m+2)u^0_{m,k,j-1} 
  = (a+m+1)u^0_{m-1,k,j}+\frac{1}{2}\sum^k_{q=0}\sum^j_{p=0}a_{k-q,j-p}u^0_{m-1,q,p} \\
    +\frac{1}{4}r^0_{m-1,k,j-1}
    -\sum^k_{q=0}\sum^j_{p=0}[j-p+(k-q)a]\left[c_{k-q,j-p}+\frac{1}{2}a_{k-q,j-p}\right]r^0_{m-1,q,p}
    -\frac{1}{4}r^0_{m-2,k,j},
\end{multline}
and
\begin{multline}
  mr^0_{m,k,j-1} 
  = (m-1)r^0_{m-1,k,j}-\frac{1}{2}\sum^k_{q=0}\sum^j_{p=0}a_{k-q,j-p}r^0_{m-1,q,p} \\
    +u^0_{m,k,j-1}-4\sum^k_{q=0}\sum^j_{p=0}[j-p+(k-q)a]c_{k-q,j-p}u^0_{m,q,p}-u^0_{m-1,k,j},
\end{multline}
These equations can be solved for successive values of $ m $ starting with the $ m=0 $ values
for all $ k,j $
\begin{equation}
   u^0_{0,k,j} = 0, \quad r^0_{0,k,j} = 0, \quad\text{except for $ k=j=0 $ where $ r^0_{0,0,0} = 1 $} ,
\end{equation}
which follow from $ r^0_{0}(s)=1, u^0_{0}(s)=0 $.
The next few coefficients can be read off from 
\begin{multline}
  r^0_{1}(s) \mathop{\sim}\limits_{s \to 0} \frac{1}{4(a+3)}+\frac{a}{16(a+4)(a+3)^2(a+2)}s \\
  - \frac{1}{128.4^a(a+3)(a+2)(a+1)\Gamma(a+5)\Gamma(a+4)}s^{a+3} ,
\end{multline}
\begin{multline}
  u^0_{1}(s) \mathop{\sim}\limits_{s \to 0} \frac{1}{4(a+3)}+\frac{a}{8(a+4)(a+3)^2(a+2)}s \\
  - \frac{1}{128.4^a(a+3)(a+2)(a+1)\Gamma^2(a+4)}s^{a+3} ,
\end{multline}
\begin{multline}
  r^0_{2}(s) \mathop{\sim}\limits_{s \to 0} \frac{1}{32(a+4)(a+3)}+\frac{a}{64(a+5)(a+4)(a+3)^2(a+2)}s \\
  - \frac{1}{512.4^a(a+3)(a+2)(a+1)\Gamma(a+6)\Gamma(a+4)}s^{a+3} ,
\end{multline}
\begin{multline}
  u^0_{2}(s) \mathop{\sim}\limits_{s \to 0} \frac{1}{16(a+4)(a+3)}+\frac{3a}{64(a+5)(a+4)(a+3)^2(a+2)}s \\
  - \frac{1}{512.4^a(a+3)(a+2)(a+1)\Gamma(a+5)\Gamma(a+4)}s^{a+3} .
\end{multline}

The analogous results for $ r^s_{m,k,j}, u^s_{m,k,j} $ are 
\begin{equation}
   (a+m+2)u^s_{m,k,0} = -\frac{1}{4}r^s_{m-1,k,0}, \quad
   mr^s_{m,k,0} = u^s_{m,k,0},
\end{equation}
for $ m \geq 1, k \geq 0 $. Again for the general case $ j \geq 1, k \geq 0 $ we have
\begin{multline}
  (a+m)u^s_{m,k,j-1}+\frac{1}{2}\sum^k_{q=0}\sum^{j-1}_{p=0}a_{k-q,j-1-p}u^s_{m,q,p} \\
  = (a+m+1)u^s_{m-1,k,j}+\frac{1}{2}\sum^k_{q=0}\sum^j_{p=0}a_{k-q,j-p}u^s_{m-1,q,p} \\
    +\sum^k_{q=0}\sum^{j-1}_{p=0}[j-1-p+(k-q)a]\left[c_{k-q,j-1-p}+\frac{1}{2}a_{k-q,j-1-p}\right]r^s_{m,q,p} \\
    -\frac{1}{4}r^s_{m-1,k,j-1}
    -\sum^k_{q=0}\sum^j_{p=0}[j-p+(k-q)a]\left[c_{k-q,j-p}+\frac{1}{2}a_{k-q,j-p}\right]r^s_{m-1,q,p} \\
    +\frac{1}{4}r^0_{m-2,k,j},
\end{multline}
and
\begin{multline}
  mr^s_{m,k,j-1}-\frac{1}{2}\sum^k_{q=0}\sum^{j-1}_{p=0}a_{k-q,j-1-p}r^s_{m,q,p} \\
  = (m-1)r^s_{m-1,k,j}-\frac{1}{2}\sum^k_{q=0}\sum^j_{p=0}a_{k-q,j-p}r^s_{m-1,q,p} \\
    -u^s_{m-1,k,j}+4\sum^k_{q=0}\sum^j_{p=0}[j-p+(k-q)a]c_{k-q,j-p}u^s_{m,q,p},
\end{multline}
Again these can be solved for successive values of $ m $ starting with the initial values
given by the relations
\begin{equation}
   r^s_{0,k,j} = \sum^{j}_{n=0}r^0_{n,k,j-n} ,
\end{equation}
along with an identical formula for $ u^s_{0,k,j} $.

\subsection{The $\sigma$-function expansion about $ s=\infty $}
The nature of the expansion of $ \nu(s) $ about $ s=\infty $ is rather different because
this fixed singular point is irregular in the case of the third Painlev\'e transcendent.
\begin{proposition}
The formal asymptotic expansion of $ \nu(s) $ about $ s=\infty $ has the form
\begin{equation}
  \nu(s) \mathop{\sim}\limits_{s \to \infty} \sum^{\infty}_{j=-1} d_j s^{-j/2} ,
\label{}
\end{equation}
where $ d_{-1} = \pm \frac{1}{2}a $.
\end{proposition}
\begin{proof}
We start with the general ansatz of 
\begin{equation}
  \nu(s) = ds^{k\alpha} + {\rm O}(s^{(k-1)\alpha}) ,
\label{}
\end{equation}
where $ k \in \NN $ and $ \alpha \in \CC $ with $ 0 < {\rm Re}(\alpha) < 1 $. 
Using (\ref{dL_HE_sigma}) we find that the only terms which can balance the 
$ {\rm O}(1) $ term are the $ {\rm O}(s^{2k\alpha-1}) $ and $ {\rm O}(s^{3k\alpha-2}) $
terms. If $ k\alpha < 1 $ then the former choice applies and we find $ k\alpha = 1/2 $.
If we assume $ k\alpha > 1 $ then the latter case must be chosen but this leads to
$ k\alpha = 2/3 $ in contradiction to the hypothesis. With the correct choice of
$ k\alpha $ we find the above relation for leading coefficient, $ d_{-1} $. Considering
the sub-leading terms we note that it is only possible for the terms of 
$ {\rm O}(s^{-\alpha}) $ to balance those of $ {\rm O}(s^{-1/2}) $ so that in fact
$ \alpha = 1/2 $ and $ k=1 $. This is entirely consistent with the expansion for the
Painlev\'e III transcendent $ q_{\rm III}(t) $ about $ t=\infty $ as found in 
\cite{GLS_2002}.
\end{proof}

The choice of the sign of the leading coefficient is positive in our application.
Consequently we find the first few terms of the asymptotic expansions of the $\sigma$-function
\begin{multline}
  \nu(s) \mathop{=}\limits_{s \to \infty}
  \frac{1}{2}as^{1/2} - \frac{1}{4}a(a+4) + \frac{15}{16}as^{-1/2} + \frac{15}{16}a^2s^{-1}
  \\
  + \frac{15}{256}a(16a^2-7)s^{-3/2} + {\rm O}(s^{-2}) ,
\label{HE_sLarge:a}
\end{multline}
and the auxiliary variables
\begin{multline}
  \mu(s)+s \mathop{=}\limits_{s \to \infty}
  as^{1/2} - \frac{15}{8}as^{-1/2} - \frac{15}{4}a^2s^{-1}
  \\
  - \frac{45}{128}a(16a^2-7)s^{-3/2} + {\rm O}(s^{-2}) ,
\label{HE_sLarge:b}
\end{multline}
and 
\begin{equation}
  2C(s)+a \mathop{=}\limits_{s \to \infty}
  \frac{5}{2}as^{-1/2}
  + \frac{5}{2}a^2s^{-1}
  + \frac{5}{16}a(8a^2-21)s^{-3/2}
  + {\rm O}(s^{-2}) ,
\label{HE_sLarge:c}
\end{equation}
and 
\begin{equation}
  \xi(s) \mathop{=}\limits_{s \to \infty}
  \frac{1}{4}as^{1/2}
  - \frac{35}{32}as^{-1/2}
  - \frac{35}{16}a^2s^{-1}
  + {\rm O}(s^{-3/2}) .
\label{HE_sLarge:d}
\end{equation}

The large $s$-regime also implies a simplification of spectral derivative 
(\ref{HE_LP:a}) which becomes
\begin{equation}
  \partial_z \Psi =
  \begin{pmatrix} \renewcommand{\arraystretch}{1.2}
                  -\frac{\displaystyle a}{\displaystyle 2(z-s)} &
                   \frac{\displaystyle 1}{\displaystyle z} \\
                   \frac{1}{4} &
                  -\frac{\displaystyle a}{\displaystyle 2(z-s)}
                  -\frac{\displaystyle 2}{\displaystyle z}
  \end{pmatrix} \left\{ \mathbb{I}+{\rm O}(s^{-1/2}) \right\}
  \Psi .
\label{}
\end{equation}
Using the substitution $ \Psi \mapsto \exp(-a\ln(s-z)\mathbb{I}/2)\Psi $ this decouples 
and can be solved in terms of the modified Bessel functions. An application of the 
boundary condition (\ref{HE_BC:e}) implies that 
\begin{align}
  q(z;s) & \sim \left(\frac{s}{s-z}\right)^{a/2}\frac{8}{z}I_{2}(\sqrt{z}) ,
  \label{HE_qInfExp:a} \\
  p(z;s) & \sim \left(\frac{s}{s-z}\right)^{a/2}z
                \frac{d}{dz}\left(\frac{8}{z}I_{2}(\sqrt{z})\right) .
  \label{HE_qInfExp:b}
\end{align}

\subsection{The $\sigma$-function expansion about a regular point}
In this subsection we seek Taylor series expansions for the sigma function and derived
variables about regular points $ s_0 $, taken to be positive and real without any loss
of generality. Let us write
\begin{equation}
  \nu(s) = \sum^{\infty}_{j=0}d_j(s-s_0)^j ,
\label{HE_nuRegExp}
\end{equation}
and using (\ref{dL_HE_sigma}) we find the following recurrence relations for
the coefficients $ d_j $
\begin{equation}
  4s_0^2d_2^2 - \frac{1}{16}a^2 + \frac{1}{2}a(a+2)d_1 + d_0d_1
  - [s_0+(a+2)^2]d^2_1 + 4(s_0d_1-d_0)d^2_1 = 0 ,
\label{}
\end{equation}
and for $ n \geq 1 $, $ d_2 \neq 0 $
\begin{multline}
  4s_0^2(n+2)(n+1)d_2d_{n+2} = - \frac{1}{2}a(a+2)(n+1)d_{n+1} \\
  + [s_0+(a+2)^2]\sum^{n}_{j=0}(j+1)(n-j+1)d_{j+1}d_{n-j+1} \\
  + \sum^{n}_{j=0}(j+1)(n-j-1)d_{j+1}d_{n-j} \\
  - \sum^{n-1}_{j=1}(j+1)j(n-j+1)(n-j)d_{j+1}d_{n-j+1} \\
  - 2s_0\sum^{n-1}_{j=0}(j+2)(j+1)(n-j+1)(n-j)d_{j+2}d_{n-j+1} \\
  - s^2_0\sum^{n-1}_{j=1}(j+2)(j+1)(n-j+2)(n-j+1)d_{j+2}d_{n-j+2} \\
  - 4\sum^{n}_{i=0}\sum^{n-i}_{j=0}(i+1)(j+1)(n-i-j-1)d_{i+1}d_{j+1}d_{n-i-j} \\
  - 4s_0\sum^{n}_{i=0}\sum^{n-i}_{j=0}(i+1)(j+1)(n-i-j+1)d_{i+1}d_{j+1}d_{n-i-j+1} .
\label{}
\end{multline}
These recurrences are solved subject to the initial values of
\begin{align}
   d_0 & = \sum^{\infty}_{k=0}\sum^{\infty}_{j=3k} c_{k,j}s^{j+ka}_0 , \\
   d_1 & = \sum^{\infty}_{k=0}\sum^{\infty}_{j=3k} (j+ka)c_{k,j}s^{j-1+ka}_0 ,
\end{align}
which in turn can be found from the solutions to the recurrences (\ref{PIII_cRecur:a})
and (\ref{PIII_cRecur:b}). In addition we define
\begin{align}
   \mu(s) & = \sum_{j \geq 0}f_{j}(s-s_0)^j , \label{HE_muRegExp}\\
     C(s) & = \sum_{j \geq 0}g_{j}(s-s_0)^j , \label{HE_CRegExp}\\
   \xi(s) & = \sum_{j \geq 0}h_{j}(s-s_0)^j . \label{HE_xiRegExp} 
\end{align}
The coefficients appearing here are computed using the recurrences
\begin{equation}
  f_0=s_0(4d_1-1), \quad f_1=4d_1-1+8s_0d_2, \quad
  f_j=4[jd_j+s_0(j+1)d_{j+1}], \quad j \geq 2 ,
\end{equation}
and subject to $ f_0 \neq 0 $
\begin{align}
  2f_0 g_0 &= -2s_0-(a+3)f_0+s_0f_1, \\
  2f_0 g_1 &= -2-(a+2)f_1+2s_0f_2-2f_1g_0, \\
  2f_0 g_j &= (j-a-3)f_j+(j+1)s_0f_{j+1}-2\sum^{j-1}_{k=0}f_{j-k}g_k, \quad j \geq 2 ,
\end{align}
and provided $ s_0+f_0 \neq 0 $
\begin{align}
  (s_0+f_0)h_0 &= -s_0g_0(a+g_0), \\
  (s_0+f_0)h_1 &= -as_0g_1-2s_0g_0g_1-(a+g_0)g_0-(1+f_1)h_0, \\
  (s_0+f_0)h_j &= -as_0g_j-ag_{j-1}-s_0\sum^{j}_{k=0}g_{j-k}g_k-\sum^{j}_{k=1}g_{j-k}g_{k-1}
  \\ &\phantom{=}\qquad -(1+f_1)h_{j-1}-\sum^{j-2}_{k=0}f_{j-k}h_k, \quad j \geq 2 .
  \nonumber
\end{align}

\subsection{The $ q, p $ expansion about a regular point}
We will seek a Taylor series approximation for the scaled polynomial and
associated function $ q(z;s), p(z;s) $ about the regular point $ (z_0,s_0) $
with $ 0 < z_0 < s_0 $.
Let us write
\begin{align}
   q(z;s) & = \sum_{j,k \geq 0}r_{j,k}(z-z_0)^j(s-s_0)^k , \label{HE_qRegExp}\\
   p(z;s) & = \sum_{j,k \geq 0}u_{j,k}(z-z_0)^j(s-s_0)^k , \label{HE_pRegExp}
\end{align}
Using the first spectral derivative (\ref{HE_S:a}) we obtain the 
recurrence relation
\begin{multline}
  z_0(s_0-z_0)(j+1)r_{j+1,k} + (s_0-2z_0)jr_{j,k} - (j-1)r_{j-1,k}
  + z_0(j+1)r_{j+1,k-1} + jr_{j,k-1} \\
  = - z_0u_{j,k} - u_{j-1,k}
    - \sum^{k}_{l=0}\left[f_{k-l}u_{j,l}+g_{k-l}r_{j-1,l}+z_0g_{k-l}r_{j,l} \right] ,
\label{HE_RR_S:a}
\end{multline}
for $ j,k \geq 1 $. When $ k=0 $ and $ j \geq 1 $ we have the specialisation
\begin{multline}
  z_0(s_0-z_0)(j+1)r_{j+1,0} + [(s_0-2z_0)j+z_0g_0]r_{j,0}
  + [g_0-j+1]r_{j-1,0} \\
  = -(f_0+z_0)u_{j,0} - u_{j-1,0} .
\label{HE_RR_Ss:a}
\end{multline}
The second spectral derivative (\ref{HE_S:b}) yields the recurrence relation
\begin{multline}
  z_0(s_0-z_0)(j+1)u_{j+1,k} + [(s_0-2z_0)j+2s_0-(a+2)z_0]u_{j,k} - (a+j+1)u_{j-1,k} \\
  + z_0(j+1)u_{j+1,k-1} + (j+2)u_{j,k-1} \\
  = \frac{1}{4}z_0(s_0-z_0)r_{j,k} + \frac{1}{4}(s_0-2z_0)r_{j-1,k} - \frac{1}{4}r_{j-2,k} 
  + \frac{1}{4}z_0r_{j,k-1}+ \frac{1}{4}r_{j-1,k-1} \\
  - \sum^{k}_{l=0}h_{k-l}(r_{j-1,l}+z_0r_{j,l})
  + \sum^{k}_{l=0}g_{k-l}(u_{j-1,l}+z_0u_{j,l}) ,
\label{HE_RR_S:b}
\end{multline}
for $ j,k \geq 1 $. When $ k=0 $ and $ j \geq 1 $ we have the specialisation
\begin{multline}
  z_0(s_0-z_0)(j+1)u_{j+1,0} + [(s_0-2z_0)j+2s_0-(a+2)z_0-z_0g_0]u_{j,0}
  - [a+1+g_0+j]u_{j-1,0} \\
  = [\frac{1}{4}z_0(s_0-z_0)-h_0z_0]r_{j,0} + [\frac{1}{4}(s_0-2z_0)-h_0]r_{j-1,0}
    - \frac{1}{4}r_{j-2,0} .
\label{HE_RR_Ss:b}
\end{multline}
Using the first deformation derivative (\ref{HE_D:a}) we obtain the
recurrence relation
\begin{multline}
  s_0(s_0-z_0)(k+1)r_{j,k+1} - s_0(k+1)r_{j-1,k+1} + (2s_0-z_0)kr_{j,k}
  - kr_{j-1,k} + (k-1)r_{j,k-1} \\
  =   s_0u_{j,k} + u_{j,k-1}
    + \sum^{k}_{l=0}\left[f_{k-l}u_{j,l}+g_{k-l}r_{j-1,l}+z_0g_{k-l}r_{j,l} \right] ,
\label{HE_RR_D:a}
\end{multline}
for $ j,k \geq 1 $. When $ j=0 $ and $ k \geq 1 $ we have the specialisation
\begin{multline}
  s_0(s_0-z_0)(k+1)r_{0,k+1} + (2s_0-z_0)kr_{0,k} + (k-1)r_{0,k-1} \\
  = s_0u_{0,k} + u_{0,k-1} + \sum^{k}_{l=0}[f_{k-l}u_{0,l}+z_0g_{k-l}r_{0,l}] .
\label{HE_RR_Ds:a}
\end{multline}
The second deformation derivative (\ref{HE_D:b}) in turn gives us the recurrence
relation
\begin{multline}
  s_0(s_0-z_0)(k+1)u_{j,k+1} - s_0(k+1)u_{j-1,k+1} \\
  + [(2s_0-z_0)k+as_0]u_{j,k} - ku_{j-1,k} + (a+k-1)u_{j,k-1} \\
  =   \sum^{k}_{l=0}h_{k-l}(r_{j-1,l}+z_0r_{j,l})
    - (2s_0-z_0)\sum^{k}_{l=0}g_{k-l}u_{j,l}
    + \sum^{k}_{l=0}g_{k-l}u_{j-1,l} - 2\sum^{k-1}_{l=0}g_{k-1-l}u_{j,l} ,
\label{HE_RR_D:b}
\end{multline}
for $ j,k \geq 1 $. When $ j=0 $ and $ k \geq 1 $ we have the specialisation
\begin{multline}
  s_0(s_0-z_0)(k+1)u_{0,k+1} + [(2s_0-z_0)k+as_0]u_{0,k} + [a+k-1]u_{0,k-1} \\
  =    z_0\sum^{k}_{l=0}h_{k-l}r_{0,l}
    - (2s_0-z_0)\sum^{k}_{l=0}g_{k-l}u_{0,l}
    - 2\sum^{k-1}_{l=0}g_{k-1-l}u_{0,l} .
\label{HE_RR_Ds:b}
\end{multline}
These recurrences can be solved in the following way. First (\ref{HE_RR_Ss:a}) and
(\ref{HE_RR_Ss:b}) are solved for $ r_{j,0},u_{j,0} $ for $ j \geq 1 $ in terms of 
$ r_{0,0},u_{0,0} $. Then these solutions can be substituted into the boundary conditions 
\begin{align}
   \sum_{j \geq 0}r_{j,0} (-z_0)^j & = 1 , \\
   \sum_{j \geq 0}u_{j,0} (-z_0)^j & = 0 ,
\end{align}
and this allows for $ r_{0,0},u_{0,0} $ to be found. Next the sequence $ r_{0,k},u_{0,k} $
can be found for $ k \geq 1 $ using (\ref{HE_RR_Ds:a}) and (\ref{HE_RR_Ds:b}). Finally
the two general systems (\ref{HE_RR_S:a}),(\ref{HE_RR_S:b}) and 
(\ref{HE_RR_D:a}),(\ref{HE_RR_D:b}) can be employed to compute $ r_{j,k},u_{j,k} $
for $ j,k \geq 1 $.

\vfill\eject
\subsection{Numerical studies at the Hard Edge}
Using the integral formula for the distribution $ A_a(z) $ as given by (\ref{dL_HE_A}) 
it is possible to compute values of this and the examples of $ a=0,1,2 $ are
plotted in Figure 2. 
However we wish to characterise it in a
precise quantitative way and evaluate the moments of this distribution
\begin{equation}
   m_k := \int^{\infty}_{0}dz\; z^k A_a(z), \quad k \in \ZZ_{\geq 0} .
\end{equation}
These are easily seen to be
\begin{multline}
   m_k = \frac{1}{4^{2a+3}\Gamma(a+1)\Gamma(a+2)\Gamma^2(a+3)} \\ \times
   \int^{\infty}_{0}ds\;s^{k+2a+3}e^{-s/4}e^{F(s)}
   \int^{1}_{0}du\;u^{k+2}(1-u)^{a}G(us;s) ,
\label{HE_A_mom}
\end{multline}
where
\begin{equation}
   F(s) := \int^{s}_{0}\frac{dv}{v}[\nu(v)+2C(v)] ,
\end{equation}
and 
\begin{equation}
   G(z;s) := q\partial_z p-p\partial_z q .
\end{equation}
We note that by employing the large $s$ asymptotic form of $ q(z;s),p(z;s) $ as given in
(\ref{HE_qInfExp:a},\ref{HE_qInfExp:b}) we can deduce the asymptotic form of the spacing 
distribution is given by
\begin{equation}
   A_a(z) \mathop{\sim}\limits_{z \to \infty}
   C e^{-z/4+(a+2)\sqrt{z}}z^{-1/2-a^2/4} ,
\end{equation}
where $ C $ is a constant which cannot be found from our methods. For small $ z $ we find
that 
\begin{multline}
  A_a(z) \mathop{\sim}\limits_{z \to 0} \frac{1}{4^{2a+3}\Gamma(a+1)\Gamma(a+2)\Gamma^2(a+3)}
  \\ \times
  \int^{\infty}_0 ds\; s^{2a}\left( \frac{1}{12}-\frac{\xi(s)}{3s} \right)
  \exp\left( -\frac{s}{4}+\int^s_0\frac{dv}{v}[\nu(v)+2C(v)] \right)z^2 ,
\end{multline}
where we have used (\ref{HE_z=0Exp:a},\ref{HE_z=0Exp:b}) and the fact that 
$ u^0_1(s) = 1/12-\xi(s)/3s $. Therefore we can conclude that the moments exist for
$ {\rm Re}(k) > -3 $, $ {\rm Re}(a) > -1 $ and $ {\rm Re}(k+2a) > -4 $.
An instance where exact evaluation of the moments can be made is the case $ a=0 $ and
the first four of these are
\begin{gather}
   m_{1} = 4e^2\left[ I_{0}(2)-I_{1}(2) \right] , \quad 
   m_{2} = 32e^2I_{0}(2) , \\ 
   m_{3} = 384e^2\left[ 2I_{0}(2)+I_{1}(2) \right] , \quad 
   m_{4} = 2048e^2\left[ 13I_{0}(2)+9I_{1}(2) \right] .
\end{gather}
We investigated the distributions $ A_a(z) $ and $ A^{\pm}(z) $ for the two 
special cases of $ a=\pm 1/2 $ in some detail because of 
the motivations provided by (\ref{LUE_minushalf}) and (\ref{LUE_plushalf}).
The analogue of (\ref{HE_A_mom}) for $ A^{\pm}(z) $ is
\begin{multline}
   m^{\pm}_k = \frac{1}{2^{4a+5}\Gamma(a+1)\Gamma(a+2)\Gamma^2(a+3)} \\ \times
   \int^{\infty}_{0}ds\;s^{k/2+2a+3}e^{-s/4}e^{F(s)}
   \int^{1}_{0}du\;u^{k+2}(1-u)^{2a+1}(2-u)^2G(u(2-u)s;s) ,
\label{HE_A_pm_mom}
\end{multline}
for $ a=\pm\frac{1}{2} $.
The statistical data for $ A_a(z) $ for the cases $ a=-1/2,0,1/2,1,2 $ are given in 
Table 1 and the data for $ A^{\pm}(z) $ is given in Table 2.

Our strategy is that by employing local Puiseux-type and Taylor expansions for the 
two factors in the integrand, namely $ e^{F(s)} $ and $ G(z;s) $, within a given finite
member of the patchwork of local expansion domains the above integrals 
restricted to this domain can be exactly evaluated. This is essential as numerical
quadrature algorithms implemented in either computer algebra software or compiled language
packages (e.g. QUADPACK) have minimal attained error tolerances which cannot be reduced 
below a fixed bound.  
For the compiled language option with a floating point representation of 64 bits the
best one could expect is a relative error of around $ 10^{-15} $ 
but often it is far worse and around $ 10^{-8}-10^{-9} $. To illustrate this we have 
computed the statistical data for the $ a=1,2 $ cases using QUADPACK routines and the
results are displayed in Table 1.
In the case of the Puiseux-type expansions the integrals are
\begin{equation}
  \int_{0}^{S} ds\;s^{k+3+l+n+(m+2)a}e^{-s/4},
\end{equation} 
and
\begin{equation}
  \int^{1/2}_{0}du\;u^{k+2+n}(1-u)^{a} \quad\text{or}\quad
  \int^{1}_{1/2}du\;u^{k+2}(1-u)^{a+n} ,
\end{equation} 
which for $ k,l,m,n \in \ZZ_{\geq 0} $ can be evaluated in terms of radicals, 
the Gamma function at integer arguments, a rational function of $ a $ and the Whittaker
function $ M(\alpha,\beta;S/4) $ or its specialisations depending on $ a $.
For example in the case $ a=1/2 $ the $s$-integral reduces to the error and exponential
functions.
For the Taylor expansion case we have the double integral
\begin{equation}
  \int_{s_1}^{s_2} ds\;s^{k+3+2a}(s-s_0)^me^{-s/4}\int^{1}_{0}du\;u^{k+2}(1-u)^{a}(us-z_0)^l ,
\end{equation}
which for $ k,l,m \in \ZZ_{\geq 0} $ and $ s_0 \in (s_1,s_2) $ is evaluated in terms of
a rational function of $ a $, a polynomial function of $ s_0, z_0 $ and the Whittaker
functions $ M(\alpha,\beta;s_{1,2}/4) $.  
Therefore the only sources of error are from the truncation of the expansions and the
finite number of intervals, both of which can be adjusted to reduce the contributing
errors. 

The computations were performed, in most cases, using the computer algebra system 
Maple with a sufficiently large number of decimal digits and found that 250 digits was 
more than adequate. In addition we found that we had to tailor
the numerical parameters for each case of $ a=\pm 1/2 $ differently as the errors varied quite 
strongly with $ a $ (this was especially pronounced as $ a $ approached $ -1 $). 
We discuss the case of $ a=1/2 $ first.
In regard to the truncations about the singular point $ s=0 $, 
we found that 1890 terms in the expansion of the transcendent variable (\ref{HE_puiseux})
with $ k \leq 30, j+k \leq 120 $ yielded an error for the second derivative of 
$ \nu(s) $ at $ s=2 $ which was estimated to be 
$ 4.6\times 10^{-98} $. At most 100 terms were retained in each of the expansions of the 
transcendent variables about regular points
(\ref{HE_nuRegExp}),(\ref{HE_muRegExp}),(\ref{HE_CRegExp}) and (\ref{HE_xiRegExp})
because much fewer, of the order of $ j,k \leq 20 $, were required in the corresponding 
expansions of the linear variables.
The number of intervals in the $s$-direction was taken to be $ 19 $ with the sequence of
$ s_0 $ values being 
$ \{0,2,5/2,3,4,6,9,13,19,25,30,38,54,72,90,115,150,200,300,500\} $.
The boundaries of the $s$-interval with node $ s_0 $ were taken to be located at the
midpoints of $ s_0 $ and its neighbouring nodes.
This sequence of nodes was chosen to be close to an optimal situation yielding the largest
separation of each node from its preceding node, yet close enough so to ensure that 
the error in $ \nu''(s) $ at the node was less than $ 9.8\times 10^{-37} $.
For each $s$-interval with node $ s_0 $ two expansion points in the 
$z$-direction were chosen because a single expansion point could never ensure that all of 
the integration region would fall within the domain of convergence about that point.
The two points that together yielded the largest convergence domain were found
to be located at $ z_0=0 $ and $ z_0=s_0 $. 
Subdividing the $z$-interval into three sub-intervals was found to
contribute a variation of less than $ 3.2\times 10^{-19} $ to the normalisation.
Another criteria that the sequence of 
$ s_0 $ nodes had to satisfy was that each $ (s,z) $ integration region fell 
completely within the union of the convergence domains about $ (s_0,0) $ and 
$ (s_0,s_0) $. 
For the expansions of the linear variables about the lines $ z=0,s $ and about the
singular point $ s=0 $, as defined 
in (\ref{HE_z=0Exp_s=0:a},\ref{HE_z=0Exp_s=0:b}) 
along with (\ref{HE_z=0Exp:a},\ref{HE_z=0Exp:b},\ref{HE_z=sExp:a},\ref{HE_z=sExp:b}), 
we chose the cut-off in the sum to be $ 20 $.
The expansions of linear variables about the lines $ z=0,s $ and about a regular point 
$ s=s_0>0 $ as defined in (\ref{HE_qRegExp},\ref{HE_pRegExp}) 
were cutoff at $ 25 $.
An overall estimate of the accuracy is provided by the normalisation, which was unity
to within $ 1.6\times 10^{-18}$.
The second case with $ a=-1/2 $ was more demanding computationally.  
We needed 5150 terms in the expansion of the transcendent variable (\ref{HE_puiseux}) 
with $ k \leq 50, j+k \leq 200 $ 
and computed these with compiled code using the multiple-precision library MPFUN
\cite{Ba_1993a},\cite{Ba_1993b}. 
The error for $ \nu''(s) $ at $ s=2 $ was estimated to be around $ 7.9\times 10^{-59} $. 
Again only 100 terms were retained in each of the expansions of the transcendent variables
about regular points.
A larger number of intervals in the $s$-direction were employed, namely $ 24 $, and the 
sequence of $ s_0 $ values was taken to be 
$ \{0,2,5/2,3,7/2,4,5,6,7,9,11,14,18,22,28,36,46,58,72,90,114,144,180,220,300\} $.
This time the sequence of nodes was chosen to ensure that 
the error in $ \nu''(s) $ at each node was less than $ 3.6\times 10^{-60} $.
And again for each $s$-interval with node $ s_0 $ two expansion points in the 
$z$-direction were chosen at $ z_0=0,s_0 $. 
In the expansions of the linear variables about the lines $ z=0,s $ and about the
singular point $ s=0 $ the cut-off in the sum was chosen to be $ 20 $ as before.
The expansions of these variables about the lines $ z=0,s $ and about a regular point 
$ s=s_0>0 $ was terminated at the cutoff of $ 25 $ also.
The estimate of the accuracy provided by the normalisation was $ 4.6\times 10^{-20}$.

Because the raw moments grow rapidly with order we have computed some standard statistical
quantities instead using the definitions of the variance $ \sigma^2 $, the skewness $ \gamma_1 $ 
and the kurtosis excess $ \gamma_2 $ 
\begin{equation}
   \sigma^2 = \mu_2, \quad
   \gamma_1 = \frac{\mu_3}{\mu_2^{3/2}}, \quad
   \gamma_2 = \frac{\mu_4}{\mu_2^{2}}-3,
\end{equation}
in terms of the central moments
\begin{align}
   \mu_2 & = m_2-m^2_1, \\
   \mu_3 & = m_3-3m_1m_2+2m_1^3, \\
   \mu_4 & = m_4-4m_1m_3+6m_1^2m_2-3m_1^4.
\end{align}


\begin{figure}[h]\label{Fig2}
\vskip0cm
\begin{center}
\epsfig{file={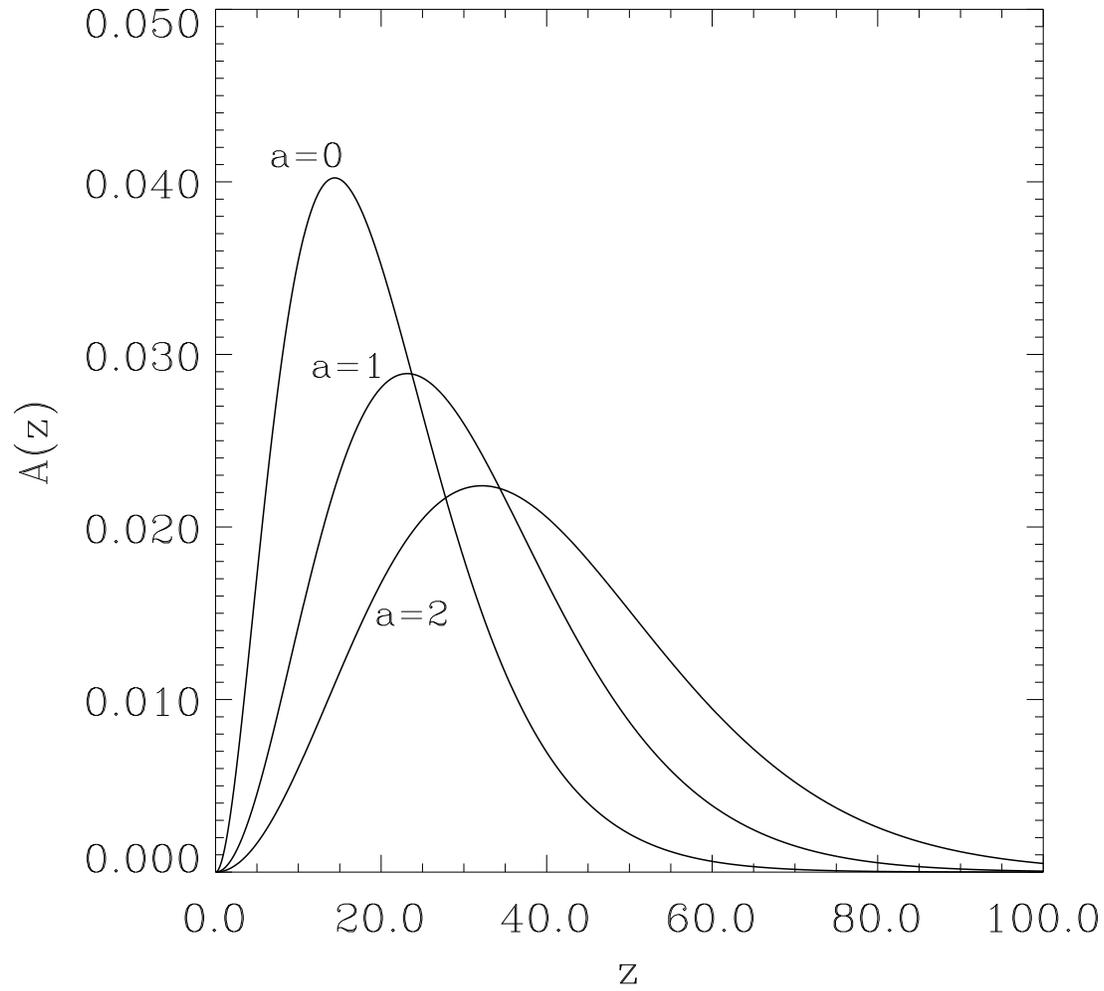},bbllx=18cm,bblly=15cm}
\end{center}
\vskip13.7cm
\caption{The distribution of the first eigenvalue spacing at the hard edge of random 
hermitian matrices $ A_a(z) $ for integral values of $ a=0,1,2 $.}
\end{figure}

\vfill\eject

%

\begin{table}[h]
\renewcommand{\arraystretch}{1.2}
\begin{sideways}
\begin{tabular}{|c|c|c|c|c|}
\hline
 \multicolumn{5}{|c|}{Statistical Data} \\
\hline
        & mean & standard deviation & skewness & kurtosis excess \\
 $ a= $ & $ m_1 $ & $ \sigma $ & $ \gamma_1 $ & $ \gamma_2 $ \\
\hline
 $ -\frac{1}{2} $ & $ 15.629510389275067 $ & $ 9.321582142378338 $ & $ 1.135887787366796 $ & $ 1.858088325290056 $ \\
 $ 0 $ & $ 20.36271491726866 $ & $ 11.152009639447587 $ & $ 1.015815254106358 $ & $ 1.46130592496617698 $ \\
 $  \frac{1}{2} $ & $ 25.125513874146918 $ & $ 12.9994050145813 $ & $ 0.927517582109394 $ & $ 1.19444057999457 $ \\
 $ 1 $ & $ 29.93907618 $ & $ 14.87749950 $ & $ 0.8589576782 $ & $ 1.001245857 $ \\
 $ 2 $ & $ 39.750797 $ & $ 18.742042 $ & $ 0.75999390 $ & $ 0.74334711 $ \\
\hline
\end{tabular}
\end{sideways}
\vskip0.5cm
\caption{Low order statistics of the distribution $ A_a(z) $ for various values of the parameter $ a $.}
\end{table}

\vskip1.0cm

\begin{table}[h]
\renewcommand{\arraystretch}{1.2}
\begin{sideways}
\begin{tabular}{|c|c|c|c|c|}
\hline
 \multicolumn{5}{|c|}{Statistical Data} \\
\hline
        & mean & standard deviation & skewness & kurtosis excess \\
 $ a= $ & $ m_1 $ & $ \sigma $ & $ \gamma_1 $ & $ \gamma_2 $ \\
\hline
 $ -\frac{1}{2} $ & $ 0.948882728945527548 $ & $ 0.38067989233932349 $ & $ 0.3755455785328836 $ & $ -0.060161637308986 $ \\
 $  \frac{1}{2} $ & $ 0.9818310311076319 $   & $ 0.405745365891523 $   & $ 0.42308992831476 $   & $ -0.0168189358644 $ \\
\hline
\end{tabular}
\end{sideways}
\vskip0.5cm
\caption{Low order statistics of the distribution $ A^{\pm}(z) $ for the special cases of the parameter 
$ a=\pm\frac{1}{2} $.}
\end{table}

\section*{Acknowledgements}
This work was supported by the Australian Research Council.

\vfill\eject

\bibliographystyle{plain}
\bibliography{moment,random_matrices,nonlinear}

\end{document}